\documentclass{article}

\usepackage{amsmath}
\usepackage{amsthm}
\usepackage{amssymb}
\usepackage{amscd}
\usepackage{xspace}
\usepackage{verbatim}
\newtheorem{theorem}{Theorem}[section]
\newtheorem{corollary}{Corollary}[section]
\newtheorem{lemma}{Lemma}[section]
\newtheorem{proposition}{Proposition}[section]
\theoremstyle{definition}
\newtheorem{definition}{Definition}[section]

\theoremstyle{remark}
\newtheorem{remark}{Remark}[section]
\numberwithin{equation}{section}

\newcommand{\ov}{\overline}

\newcommand{\e}{\varepsilon}

\renewcommand{\O}{\Omega}

\renewcommand{\liminf}{\varliminf}
\renewcommand{\limsup}{\varlimsup}

\renewcommand{\vec}[1]{\mathbf{#1}}
\newcommand{\field}[1]{\mathbb{#1}}

\newcommand{\R}{\field{R}}

\newcommand{\er}{\eqref}

\DeclareMathOperator{\Div}{div}

\DeclareMathOperator{\supp}{supp}

\renewcommand{\O}{\Omega}

\newcommand{\f}{\varphi}
\renewcommand{\vec}[1]{\boldsymbol{#1}}

\date{}

\begin{document}
\title{Variational resolution for some general classes of nonlinear evolutions. Part II}
\maketitle

\begin{center}
\textsc{Arkady Poliakovsky \footnote{E-mail:
poliakov@math.bgu.ac.il}
}\\[3mm]
Department of Mathematics, Ben Gurion University of the Negev,\\
P.O.B. 653, Be'er Sheva 84105, Israel
\\[2mm]
\end{center}
\begin{abstract}
Using our results in \cite{PI}, we provided existence theorems for
the general classes of nonlinear evolutions. Finally, we give
examples of applications of our results to parabolic, hyperbolic,
Shr\"{o}dinger, Navier-Stokes and other time-dependent systems of
equations.
\end{abstract}
\section{Introduction}

Let $X$ be a reflexive Banach space. Consider the following
evolutional initial value problem:
\begin{equation}\label{abstpr}
\begin{cases}
\frac{d}{dt}\big\{I\cdot
u(t)\big\}+\Lambda_t\big(u(t)\big)=0\quad\quad\text{in}\;\;(0,T_0),
\\
I\cdot u(0)=v_0.
\end{cases}
\end{equation}
Here $I:X\to X^*$ ($X^*$ is the space dual to $X$) is a fixed
bounded linear inclusion operator, which we assume to be
self-adjoint and strictly positive, $u(t)\in L^q\big((0,T_0);X\big)$
is an unknown function, such that $I\cdot u(t)\in
W^{1,p}\big((0,T_0);X^*\big)$ (where $I\cdot h\in X^*$ is the value
of the operator $I$ at the point $h\in X$), $\Lambda_t(x):X\to X^*$
is a fixed nonlinear mapping, considered for every fixed
$t\in(0,T_0)$, and $v_0\in X^*$ is a fixed initial value. The most
trivial variational principle related to \er{abstpr} is the
following one. Consider some convex function
$\Gamma(y):X^*\to[0,+\infty)$, such that
$\Gamma(y)=0$ if and only if $y=0$. Next define the following energy
functional
\begin{multline}\label{abstprhfen}
E_0\big(u(\cdot)\big):=\int_0^{T_0}\Gamma\bigg(\frac{d}{dt}\big\{I\cdot
u(t)\big\}+\Lambda_t\big(u(t)\big)\bigg)dt\\ \forall\, u(t)\in
L^q\big((0,T_0);X\big)\;\;\text{s.t.}\;\;I\cdot u(t)\in
W^{1,p}\big((0,T_0);X^*\big)\;\;\text{and}\;\;I\cdot u(0)=v_0\,.
\end{multline}
Then it is obvious that $u(t)$ will be a solution to \er{abstpr} if
and only if $E_0\big( u(\cdot)\big)=0$. Moreover, the solution to
\er{abstpr} will exist if and only if there exists a minimizer
$u_0(t)$ of the energy $E_0(\cdot)$, which satisfies $E_0\big(
u_0(\cdot)\big)=0$.

 We have the following generalization of this variational principle.
Let $\Psi_t(x):X\to[0,+\infty)$ be some convex G\^{a}teaux
differentiable function, considered for every fixed $t\in(0,T_0)$
and such that $\Psi_t(0)=0$. Next define the Legendre transform of
$\Psi_t$ by
\begin{equation}\label{vjhgkjghkjghjk}
\Psi^*_t(y):=\sup\Big\{\big<z,y\big>_{X\times X^*}-\Psi_t(z):\;z\in
X\Big\}\quad\quad\forall y\in X^*\,.
\end{equation}
It is well known that $\Psi^*_t(y):X^*\to\R$ is a convex function
and
\begin{equation}\label{vhjfgvhjgjkgjkh}
\Psi_t(x)+\Psi^*_t(y)\;\geq\; \big<x,y\big>_{X\times
X^*}\quad\quad\forall\, x\in X,\,y\in X^*\,,
\end{equation}
with equality if and only if $y=D\Psi_t(x)$. Next for
$\lambda\in\{0,1\}$ define the energy
\begin{multline}\label{abstprhfenbjvfj}
E_\lambda\big(u\big):=\int\limits_0^{T_0}\Bigg\{\Psi_t\Big(\lambda
u(t)\Big)+\Psi^*_t\bigg(-\frac{d}{dt}\big\{I\cdot
u(t)\big\}-\Lambda_t\big(u(t)\big)\bigg)+\lambda\bigg<u(t),\frac{d}{dt}\big\{I\cdot
u(t)\big\}+\Lambda_t\big(u(t)\big)\bigg>_{X\times X^*}\Bigg\}dt\\
\forall\, u(t)\in L^q\big((0,T_0);X\big)\;\;\text{s.t.}\;\;I\cdot
u(t)\in W^{1,p}\big((0,T_0);X^*\big)\;\;\text{and}\;\;I\cdot
u(0)=v_0.
\end{multline}
Then, by \er{vhjfgvhjgjkgjkh} we have $E_\lambda\big(\cdot\big)\geq
0$ and moreover, $E_\lambda\big(u(\cdot)\big)=0$ if and only if
$u(t)$ is a solution to
\begin{equation}\label{abstprrrrcn}
\begin{cases}
\frac{d}{dt}\big\{I\cdot
u(t)\big\}+\Lambda_t\big(u(t)\big)+D\Psi_t\big(\lambda
u(t)\big)=0\quad\quad\text{in}\;\;(0,T_0),
\\
I\cdot u(0)=v_0
\end{cases}
\end{equation}
(note here that since $\Psi_t(0)=0$, in the case $\lambda=0$
\er{abstprrrrcn} coincides with \er{abstpr}. Moreover, if
$\lambda=0$ then the energy defined in \er{abstprhfen} is a
particular case of the energy in \er{abstprhfenbjvfj}, where we take
$\Gamma(x):=\Psi^*(-x)\,$). So, as before, a solution to
\er{abstprrrrcn} exists if and only if there exists a minimizer
$u_0(t)$ of the energy $E_\lambda(\cdot)$, which satisfies
$E_\lambda\big( u_0(\cdot)\big)=0$. Consequently, in order to
establish the existence of solution to \er{abstprrrrcn} we need to
answer the following questions:
\begin{itemize}
\item [{\bf (a)}] Does a minimizer to the energy
in \er{abstprhfenbjvfj} exist?
\item [{\bf (b)}] Does the minimizer $u_0(t)$ of the corresponding
energy $E_\lambda(\cdot)$ satisfies
$E_\lambda\big(u_0(\cdot)\big)=0$?
\end{itemize}

 To the best of our knowledge, the energy in \er{abstprhfenbjvfj} with
$\lambda=1$, related to \er{abstprrrrcn}, was first considered for
the heat equation and other types of evolutions by Brezis and
Ekeland in \cite{Brez}. In that work they also first asked question
{\bf (b)}: If we don't know a priori that a solution of the equation
\er{abstprrrrcn} exists, how to prove that the minimum of the
corresponding energy is zero. This question was asked even for very
simple PDE's like the heat equation. A detailed investigation of the
energy of type \er{abstprhfenbjvfj}, with $\lambda=1$, was done in a
series of works of N. Ghoussoub and his coauthors, see the book
\cite{NG}  and also \cite{Gho}, \cite{GosM}, \cite{GosM1},
\cite{GosTz}. In these works they considered a similar variational
principle, not only for evolutions but also for some other classes
of equations. They proved  some theoretical results about  general
self-dual variational principles, which in many cases, can provide
with the existence of a zero energy state (answering questions {\bf
(a)}+{\bf (b)} together) and, consequently, with the existence of
solution for the related equations (see \cite{NG} for details).

 In \cite{PI} we provide  an alternative approach to the
questions {\bf (a)} and {\bf (b)}. We treat them separately and in
particular, for question {\bf (b)}, we derive the main information
by studying the Euler-Lagrange equations for the corresponding
energy. To our knowledge, such an approach was first considered in
\cite{P4} and provided there an alternative proof of existence of
solution for initial value problems for some parabolic systems.
Generalizing these results, we provide in \cite{PI} the answer to
questions {\bf (a)} and {\bf (b)} for some wide classes of
evolutions. In particular, regarding question {\bf (b)}, we are able
to prove that in some general cases not only the minimizer but also
any critical point $u_0(t)$ (i.e. any solution of corresponding
Euler-Lagrange equation) satisfies
$E_\lambda\big(u_0(\cdot)\big)=0$, i.e. is a solution to
\er{abstprrrrcn}.

%
%
%
%
%
%

We can rewrite the definition of $E_\lambda$ in \er{abstprhfenbjvfj}
as follows. Since $I$ is a self-adjoint and strictly positive
operator,
there exists a Hilbert space $H$ and an injective
bounded linear operator $T:X\to H$, whose image is dense in $H$,
such that if we consider the linear operator $\widetilde{T}:H\to
X^*$, defined by the formula
\begin{equation}\label{tildetjbghgjgklhgjkgkgkjjkjkl}
\big<x,\widetilde{T}\cdot y\big>_{X\times X^*}:=\big<T\cdot
x,y\big>_{H\times H}\quad\quad\text{for every}\; y\in
H\;\text{and}\;x\in X\,,
\end{equation}
then we will have
$\widetilde{T}\circ T\equiv I$, see Lemma \ref{hdfghdiogdiofg} for
details. We call $\{X,H,X^*\}$ an evolution triple with the
corresponding inclusion operator $T:X\to H$ and $\widetilde{T}:H\to
X^*$.
Thus, if $v_0=\widetilde{T}\cdot w_0$, for some $w_0\in H$ and
$p=q^*:=q/(q-1)$, where $q>1$, then we have
$$\int_0^{T_0}\bigg<u(t),\frac{d}{dt}\big\{I\cdot
u(t)\big\}\bigg>_{X\times X^*}dt=\frac{1}{2}\big\|T\cdot
u(T_0)\big\|^2_H-\frac{1}{2}\big\|w_0\big\|^2_H$$ (see Lemma
\ref{lem2} for details) and therefore,
\begin{multline}\label{abstprhfenbjvfjghgighjkjg}
E_\lambda\big(u\big)=J\big(u\big):=\\
\int\limits_0^{T_0}\Bigg\{\Psi_t\Big(\lambda
u(t)\Big)+\Psi^*_t\bigg(-\frac{d}{dt}\big\{I\cdot
u(t)\big\}-\Lambda_t\big(u(t)\big)\bigg)+\lambda\Big<u(t),\Lambda_t\big(u(t)\big)\Big>_{X\times
X^*}\Bigg\}dt+\frac{\lambda}{2}\big\|T\cdot
u(T_0)\big\|^2_H-\frac{\lambda}{2}\big\|w_0\big\|^2_H\\
\forall\, u(t)\in L^q\big((0,T_0);X\big)\;\;\text{s.t.}\;\;I\cdot
u(t)\in W^{1,q^*}\big((0,T_0);X^*\big)\;\;\text{and}\;\;I\cdot
u(0)=\widetilde{T}\cdot w_0
\end{multline}

Our first main result in \cite{PI} provides the answer for question
{\bf (b)}, under some coercivity and growth conditions on $\Psi_t$
and $\Lambda_t$:
\begin{theorem}\label{EulerLagrangeInt}
Let $\{X,H,X^*\}$ be an evolution triple with the corresponding
inclusion linear operators $T:X\to H$, which we assume to be
bounded, injective and having dense image in $H$,
$\widetilde{T}:H\to X^*$ be defined by
\er{tildetjbghgjgklhgjkgkgkjjkjkl} and $I:=\widetilde{T}\circ T:X\to
X^*$. Next let $\lambda\in\{0,1\}$, $q\geq 2$, $p=q^*:=q/(q-1)$ and
$w_0\in H$. Furthermore, for every $t\in[0,T_0]$ let
$\Psi_t(x):X\to[0,+\infty)$ be a strictly convex function which is
G\^{a}teaux differentiable at every $x\in X$, satisfying
$\Psi_t(0)=0$ and the condition
\begin{equation}\label{roststrrr}
(1/C_0)\,\|x\|_X^q-C_0\leq \Psi_t(x)\leq
C_0\,\|x\|_X^q+C_0\quad\forall x\in X,\;\forall t\in[0,T_0]\,,
\end{equation}
for some $C_0>0$. We also assume that $\Psi_t(x)$ is a Borel
function of its variables $(x,t)$.
Next, for every $t\in[0,T_0]$ let $\Lambda_t(x):X\to X^*$ be a
function which is G\^{a}teaux differentiable at every $x\in X$, s.t.
$\Lambda_t(0)\in L^{q^*}\big((0,T_0);X^*\big)$ and the derivative of
$\Lambda_t$ satisfies the growth condition
\begin{equation}\label{roststlambdrrr}
\|D\Lambda_t(x)\|_{\mathcal{L}(X;X^*)}\leq g\big(\|T\cdot
x\|_H\big)\,\Big(\|x\|_X^{q-2}+\mu^{\frac{q-2}{q}}(t)\Big)\quad\forall
x\in X,\;\forall t\in[0,T_0]\,,
\end{equation}
for some non-decreasing function $g(s):[0+\infty)\to (0,+\infty)$
and some nonnegative function $\mu(t)\in L^1\big((0,T_0);\R\big)$.
We also assume that $\Lambda_t(x)$ is strongly Borel on the pair of
variables $(x,t)$ (see Definition
\ref{fdfjlkjjkkkkkllllkkkjjjhhhkkk}).
Assume also that $\Psi_t$ and $\Lambda_t$ satisfy the following
monotonicity condition
\begin{multline}\label{Monotonerrr}
\bigg<h,\lambda \Big\{D\Psi_t\big(\lambda x+h\big)-D\Psi_t(\lambda
x)\Big\}+D\Lambda_t(x)\cdot h\bigg>_{X\times X^*}\geq
-\hat g\big(\|T\cdot x\|_H\big)\Big(\|x\|_X^q+\hat
\mu(t)\Big)\,\|T\cdot h\|^{2}_H
\\ \forall x,h\in X,\;\forall t\in[0,T_0]\,,
\end{multline}
for some non-decreasing function $\hat g(s):[0+\infty)\to
(0,+\infty)$ and some nonnegative function $\hat\mu(t)\in
L^1\big((0,T_0);\R\big)$.
Consider the set
\begin{equation}\label{hgffckaaq1newrrrvhjhjhm}
\mathcal{R}_{q}:=\Big\{u(t)\in L^q\big((0,T_0);X\big):\;I\cdot
u(t)\in W^{1,q^*}\big((0,T_0);X^*\big)\Big\}\,,
\end{equation}
and the minimization problem
\begin{equation}\label{hgffckaaq1newrrr}
\inf\Big\{J(u):\,u(t)\in \mathcal{R}_{q}\;\;\text{s.t}\;\;I\cdot
u(0)=\widetilde{T}\cdot w_0\Big\}\,,
\end{equation}
where $J(u)$ is defined by \er{abstprhfenbjvfjghgighjkjg}. Then for
every $u\in\mathcal{R}_{q}$ such that $I\cdot
u(0)=\widetilde{T}\cdot w_0$ and for arbitrary function
$h(t)\in\mathcal{R}_{q}$, such that $I\cdot h(0)=0$, the finite
limit $\lim\limits_{s\to 0}\big(J(u+s h)-J(u)\big)/s$ exists.
Moreover, for every such $u$ the following four statements
are equivalent:
\begin{itemize}
\item[{\bf (1)}]
$u$ is a critical point of \er{hgffckaaq1newrrr}, i.e., for any
function $h(t)\in\mathcal{R}_{q}$, such that $I\cdot h(0)=0$ we have
\begin{equation}\label{nolkjrrr}
\lim\limits_{s\to 0}\frac{J(u+s h)-J(u)}{s}=0\,.
\end{equation}
\item[{\bf(2)}]
$u$ is a minimizer to \er{hgffckaaq1newrrr}.
\item[{\bf (3)}] $J(u)=0$.
\item[{\bf (4)}]
$u$ is a solution to
\begin{equation}\label{abstprrrrcnjkghjkh}
\begin{cases}
\frac{d}{dt}\big\{I\cdot
u(t)\big\}+\Lambda_t\big(u(t)\big)+D\Psi_t\big(\lambda
u(t)\big)=0\quad\quad\text{in}\;\;(0,T_0),
\\
I\cdot u(0)=\widetilde{T}\cdot w_0.
\end{cases}
\end{equation}
\end{itemize}
Finally, there exists at most one function $u\in\mathcal{R}_{q}$
which satisfies \er{abstprrrrcnjkghjkh}.
\end{theorem}
\begin{remark}\label{gdfgdghfjkllkhoiklj}
Assume that, instead of \er{Monotonerrr}, one requires that $\Psi_t$
and $\Lambda_t$ satisfy the following inequality
\begin{multline}\label{Monotone1111gkjglghjgj}
\bigg<h,\lambda \Big\{D\Psi_t\big(\lambda x+h\big)-D\Psi_t(\lambda
x)\Big\}+D\Lambda_t(x)\cdot h\bigg>_{X\times X^*}\geq
\\ \frac{\big|f(h,t)\big|^2}{\tilde g(\|T\cdot
x\|_H)}-\tilde g\big(\|T\cdot
x\|_H\big)\Big(\|x\|_X^q+\hat\mu(t)\Big)^{(2-r)/2}\big|f(h,t)\big|^r\,\|T\cdot
h\|^{(2-r)}_H\quad\forall x,h\in X,\;\forall t\in[0,T_0],
\end{multline}
for some non-decreasing function $\tilde g(s):[0+\infty)\to
(0,+\infty)$, some function $\hat\mu(t)\in L^1\big((0,T_0);\R\big)$,
some function $f(x,t):X\times[0,T_0]\to\R$ and some constant
$r\in(0,2)$.
Then \er{Monotonerrr}
follows by the trivial inequality $(r/2)\,a^2+\big((2-r)/2\big)\,
b^2\geq a^r \,b^{2-r}$.
\end{remark}

Our first result in \cite{PI} about the existence of minimizer for
$J(u)$ is the following Proposition:
\begin{proposition}\label{premainnewEx}
Assume that $\{X,H,X^*\}$, $T,\widetilde{T},I$, $\lambda,q,p$,
$\Psi_t$ and $\Lambda_t$ satisfy all the conditions of Theorem
\ref{EulerLagrangeInt} together with the assumption $\lambda=1$.
Moreover, assume that
$\Psi_t$ and $\Lambda_t$ satisfy the following positivity condition
\begin{multline}\label{MonotonegnewhghghghghEx}
\Psi_t(x)+\Big<x,\Lambda_t(x)\Big>_{X\times X^*}\geq \frac{1}{\tilde
C}\,\|x\|^q_X-
\bar\mu(t)\Big(\|T\cdot x\|^{2}_H+1\Big)
\quad\forall x\in X,\;\forall t\in[0,T_0],
\end{multline}
where
$\tilde C>0$ is some constant and $\bar\mu(t)\in
L^1\big((0,T_0);\R\big)$ is some nonnegative function.
Furthermore, assume that
\begin{equation}\label{jgkjgklhklhj}
\Lambda_t(x)=A_t\big(S\cdot x\big)+\Theta_t(x)\quad\quad\forall\,
x\in X,\;\forall\, t\in[0,T_0],
\end{equation}
where $Z$ is a Banach space, $S:X\to Z$ is a compact operator and
for every $t\in[0,T_0]$ $A_t(z):Z\to X^*$ is a function which is
strongly Borel on the pair of variables $(z,t)$ and G\^{a}teaux
differentiable at every $z\in Z$, $\Theta_t(x):X\to X^*$ is strongly
Borel on the pair of variables $(x,t)$ and G\^{a}teaux
differentiable at every $x\in X$, $\,\Theta_t(0),A_t(0)\in
L^{q^*}\big((0,T_0);X^*\big)$ and the derivatives of $A_t$ and
$\Theta_t$ satisfy the growth condition
\begin{equation}\label{roststlambdgnewjjjjjEx}
\|D\Theta_t(x)\|_{\mathcal{L}(X;X^*)}+\|D A_t(S\cdot
x)\|_{\mathcal{L}(Z;X^*)}\leq g\big(\|T\cdot
x\|\big)\,\Big(\|x\|_X^{q-2}+\mu^{\frac{q-2}{q}}(t)\Big)\quad\forall
x\in X,\;\forall t\in[0,T_0]
\end{equation}
for some nondecreasing function $g(s):[0,+\infty)\to (0+\infty)$ and
some nonnegative function $\mu(t)\in L^1\big((0,T_0);\R\big)$. Next
assume that for every sequence
$\big\{x_n(t)\big\}_{n=1}^{+\infty}\subset L^q\big((0,T_0);X\big)$
such that the sequence $\big\{I\cdot x_n(t)\big\}$ is bounded in
$W^{1,q^*}\big((0,T_0);X^*\big)$ and $x_n(t)\rightharpoonup x(t)$
weakly in $L^q\big((0,T_0);X\big)$ we have
\begin{itemize}
\item
$\Theta_t\big(x_n(t)\big)\rightharpoonup \Theta_t\big(x(t)\big)$
weakly in $L^{q^*}\big((0,T_0);X^*\big)$,
\item
$\liminf_{n\to+\infty}\int_{0}^{T_0}
\Big<x_n(t),\Theta_t\big(x_n(t)\big)\Big>_{X\times X^*}dt\geq
\int_{0}^{T_0}\Big<x(t),\Theta_t\big(x(t)\big)\Big>_{X\times
X^*}dt$.
\end{itemize}
Finally, let $w_0\in H$ be such that $w_0=T\cdot u_0$ for some
$u_0\in X$, or more generally, $w_0\in H$ be such that
$\mathcal{A}_{w_0}:=\big\{u\in\mathcal{R}_{q}:\,I\cdot
u(0)=\widetilde{T}\cdot w_0\big\}\neq\emptyset$. Then there exists a
minimizer to \er{hgffckaaq1newrrr}.
\end{proposition}
As a consequence of Theorem \ref{EulerLagrangeInt} and Proposition
\ref{premainnewEx} we have the following Corollary:
\begin{corollary}\label{CorCorCor}
Assume that we are in the settings of Proposition
\ref{premainnewEx}. Then there exists a unique solution
$u(t)\in\mathcal{R}_{q}$ to
\begin{equation}\label{abstprrrrcnjkghjkhggg}
\begin{cases}
\frac{d}{dt}\big\{I\cdot
u(t)\big\}+\Lambda_t\big(u(t)\big)+D\Psi_t\big(
u(t)\big)=0\quad\quad\text{in}\;\;(0,T_0),
\\
I\cdot u(0)=\widetilde{T}\cdot w_0.
\end{cases}
\end{equation}
\end{corollary}
As an important particular case of Corollary \ref{CorCorCor} we get
in \cite{PI} the following statement:
\begin{theorem}\label{THSldInt}
Let $\{X,H,X^*\}$ be an evolution triple with the corresponding
inclusion linear operators $T:X\to H$, which we assume to be
bounded, injective and having dense image in $H$,
$\widetilde{T}:H\to X^*$ be defined by
\er{tildetjbghgjgklhgjkgkgkjjkjkl} and $I:=\widetilde{T}\circ T:X\to
X^*$. Next let $q\geq 2$. Furthermore, for every $t\in[0,T_0]$ let
$\Psi_t(x):X\to[0,+\infty)$ be a strictly convex function which is
G\^{a}teaux differentiable at every $x\in X$, satisfies
$\Psi_t(0)=0$ and satisfies the growth condition
\begin{equation}\label{roststSLDInt}
(1/C_0)\,\|x\|_X^q-C_0\leq \Psi_t(x)\leq
C_0\,\|x\|_X^q+C_0\quad\forall x\in X,\;\forall t\in[0,T_0]\,,
\end{equation}
and the following uniform convexity condition
\begin{equation}\label{roststghhh77889lkagagSLDInt}
\Big<h,D\Psi_t(x+h)-D\Psi_t(x)\Big>_{X\times X^*}\geq
\frac{1}{C_0}\Big(\big\|x\big\|^{q-2}_X+1
\Big)\cdot\|h\|_X^2\quad\forall x,h\in X,\;\,\forall t\in[0,T_0],
\end{equation}
for some $C_0>0$.
We also assume that $\Psi_t(x)$ is Borel on the pair of variables
$(x,t)$ Next let $Z$ be a Banach space, $S:X\to Z$ be a compact
operator and for every $t\in[0,T_0]$ let $F_t(z):Z\to X^*$ be a
function, such that $F_t$ is strongly Borel on the pair of variables
$(z,t)$ and G\^{a}teaux differentiable at every $z\in Z$, $F_t(0)\in
L^{q^*}\big((0,T_0);X^*\big)$ and the derivatives of $F_t$ satisfies
the growth conditions
\begin{equation}\label{roststlambdgnewSLDInt}
\big\|D F_t(S\cdot x)\big\|_{\mathcal{L}(Z;X^*)}\leq g\big(\|T\cdot
x\|\big)\,\Big(\|x\|_X^{q-2}+1
\Big)\quad\forall x\in X,\;\forall t\in[0,T_0]\,,
\end{equation}
for some non-decreasing function $g(s):[0+\infty)\to (0,+\infty)$.
Moreover, assume that
$\Psi_t$ and $F_t$ satisfy the following positivity condition:
\begin{multline}\label{MonotonegnewhghghghghSLDInt}
\Psi_t(x)+\Big<x,F_t(S\cdot x)\Big>_{X\times X^*}\geq \frac{1}{\bar
C}\,\|x\|^q_X -\bar C
\|S\cdot x\|^{2}_Z-
\bar\mu(t)
\Big(\|T\cdot x\|^{2}_H+1\Big)
\quad\forall x\in X,\;\forall t\in[0,T_0],
\end{multline}
where $\bar C>0$ is some constants and $\bar\mu(t)\in
L^1\big((0,T_0);\R\big)$ is a nonnegative function. Furthermore, let
$w_0\in H$ be such that $w_0=T\cdot u_0$ for some $u_0\in X$, or
more generally, $w_0\in H$ be such that
$\mathcal{A}_{w_0}:=\big\{u\in\mathcal{R}_{q}:\,I\cdot
u(0)=\widetilde{T}\cdot w_0\big\}\neq\emptyset$. Then there exists a
unique
solution $u(t)\in\mathcal{R}_{q}$ to the following equation
\begin{equation}\label{uravnllllgsnachnewSLDInt}
\begin{cases}\frac{d}{dt}\big\{I\cdot u(t)\big\}+F_t\big(S\cdot u(t)\big)+D\Psi_t\big(u(t)\big)=0\quad\text{for
a.e.}\; t\in(0,T_0)\,,\\ I\cdot u(0)=\widetilde T\cdot
w_0\,.\end{cases}
\end{equation}
\end{theorem}
 In this paper using Theorem \ref{THSldInt}
as a basis, by the appropriate approximation, we obtain further
existence Theorems, under much weaker assumption on coercivity and
compactness. The first Theorem
improves the existence part of Corollary \ref{CorCorCor}. (see
Theorem \ref{defHkkkkglkjjjgkjgkjgggk} as an equivalent formulation
and Theorem \ref{defHkkkkglkjjj} as an important particular case).
\begin{theorem}\label{defHkkkkglkjjjgkjgkjgggkbmhmhgm}
Let $q\geq 2$ and $\{X,H,X^*\}$ be an evolution triple with the
corresponding inclusion linear operators $T:X\to H$, which we assume
to be injective and having dense image in $H$, $\widetilde{T}:H\to
X^*$, defined by \er{tildetjbghgjgklhgjkgkgkjjkjkl}, and
$I:=\widetilde{T}\circ T:X\to X^*$. Assume also that the Banach
space $X$ is separable.
Furthermore, for every $t\in[0,T_0]$ let $\Psi_t(x):X\to[0,+\infty)$
be a convex function which is G\^{a}teaux differentiable at every
$x\in X$, satisfies $\Psi_t(0)=0$ and satisfies the growth condition
\begin{equation}\label{roststglkjjjuhuhgbhghgjkhjk}
0\leq \Psi_t(x)\leq C\,\|x\|_X^q+C\quad\forall x\in X,\;\forall
t\in[0,T_0]\,,
\end{equation}
for some $C>0$. We also assume that $\Psi_t(x)$
is Borel on the pair of variables $(x,t)$.
Furthermore, for every $t\in[0,T_0]$ let $\Lambda_t(x):X\to X^*$ be
a function which is G\^{a}teaux differentiable at every $x\in X$,
$\Lambda_t(0)\in L^{q^*}\big((0,T_0);X^*\big)$ and the derivative of
$\Lambda_t$ satisfies the growth condition
\begin{equation}\label{roststlambdglkjjjuhuhhuiyhioujh}
\|D\Lambda_t(x)\|_{\mathcal{L}(X;X^*)}\leq g\big(\|T\cdot
x\|_H\big)\,\big(\|x\|_X^{q-2}+1\big)\quad\forall x\in X,\;\forall
t\in[0,T_0]\,,
\end{equation}
for some nondecreasing function $g(s):[0,+\infty)\to(0,+\infty)$. We
also assume that $\Lambda_t(x)$
is Borel on the pair of variables $(x,t)$. Assume also that
$\Lambda_t$ and $\Psi_t$ satisfy the following monotonicity
condition:
\begin{multline}\label{Monotoneglkjjjuhuhgutuityiyikkkghjkghjghjg}
\Big<x,D\Psi_t(x)+\Lambda_t(x)\Big>_{X\times X^*}\geq\frac{1}{\hat
C}\,\|x\|_X^q-\hat C\|L\cdot x\|^2_V
-\mu(t)\Big(\|T\cdot x\|^{2}_H+1\Big)
\quad\forall x\in X,\;\forall t\in[0,T_0],
\end{multline}
%
%
%
%
%
%
where $V$ is a given Banach space, $L\in\mathcal{L}(X,V)$ is a given
compact operator,
$\mu(t)\in L^1\big((0,T_0);\R\big)$ is some nonnegative function and
$\hat C>0$ is some constant. Finally, assume that for every
$t\in[0,T_0]$ $\big(D\Psi_t+\Lambda_t\big)(x):X\to X^*$ satisfies
the following compactness property:
\begin{itemize}
\item If $x_n\rightharpoonup x$ weakly in $X$, then
$\liminf_{n\to+\infty}\Big<x_n-x,D\Psi_t(x_n)+\Lambda_t(x_n)\Big>_{X\times
X^*}\geq 0$.
\item If $x_n\rightharpoonup x$ weakly in $X$ and
$\lim_{n\to+\infty}\Big<x_n-x,D\Psi_t(x_n)+\Lambda_t(x_n)\Big>_{X\times
X^*}=0$,
then necessarily $D\Psi_t(x_n)+\Lambda_t(x_n)\rightharpoonup
D\Psi_t(x)+\Lambda_t(x)$ weakly in $X^*$.
\end{itemize}
Then for every $w_0\in H$ there exists $u(t)\in
L^q\big((0,T_0);X\big)$, such that
$I\cdot\big(u(t)\big)\in W^{1,q^*}\big((0,T_0); X^*\big)$, where
$q^*:=q/(q-1)$, and $u(t)$ is a solution to
\er{abstprrrrcnjkghjkhggg}.
%
%
%
%
%
%
%
%
%
%
%
%
%
%
%
%
\end{theorem}
The second existence result is useful in the study of Parabolic,
Hyperbolic, Parabolic-Hyperbolic, Shr\"{o}dinger, Navier-Stokes and
other types of equations (see Theorem \ref{thhypppggghhhjggg} as an
equivalent formulation, and Theorem \ref{thhypppggghhhj} and
Corollary \ref{thhypppggg}, as important particular cases).
\begin{theorem}\label{thhypppggghhhjgggjhjghjgh}
Let $q\geq 2$ and let $X$ and $Z$ be reflexive Banach spaces and
$X^*$ and $Z^*$ be the corresponding dual spaces. Furthermore let
$H$ be a Hilbert space. Suppose that $Q:X\to Z$ is an injective
bounded linear operator such that its image is dense on $Z$.
Furthermore, suppose that $P:Z\to H$ is an injective bounded linear
operator such that its image is dense on $H$. Let $T:X\to H$ be
defined by $T:=P\circ Q$. So that $\{X,H,X^*\}$ is an evolution
triple with the corresponding inclusion operators $T:X\to H$,
$\widetilde{T}:H\to X^*$ defined by
\er{tildetjbghgjgklhgjkgkgkjjkjkl} and $I:=\widetilde{T}\circ T$.
Assume also that the Banach space $X$ is separable. Furthermore, for
every $t\in[0,T_0]$ let $\Lambda_t(z):Z\to X^*$ and $A_t(z):Z\to
X^*$ be functions which are G\^{a}teaux differentiable at every
$z\in Z$ and $A_t(0),\Lambda_t(0)\in L^{q^*}\big((0,T_0);X^*\big)$.
Assume that for every $t\in[0,T]$ they satisfy the following bounds
\begin{equation}\label{bghffgghgkoooojkvhgjgjgfffhhjjkhjk}
\big\|D\Lambda_t(z)\big\|_{\mathcal{L}(Z;X^*)}\leq g\big(\|P\cdot
z\|_H\big)\cdot\Big(\|z\|_{Z}^{q-2}+1\Big)\quad\quad\forall z\in
Z,\;\forall t\in[0,T_0]\,,
\end{equation}
\begin{equation}\label{bghffgghgkoooojkvhgjjkkgkgjk}
\big\|\Lambda_t(z)\big\|_{X^*}\leq g\big(\|P\cdot
z\|_H\big)\cdot\Big(\|L_0\cdot
z\|_{V_0}^{q-1}+\tilde\mu^{\frac{q-1}{q}}(t)\Big)\quad\quad\forall
z\in Z,\;\forall t\in[0,T_0]\,,
\end{equation}
and
\begin{equation}\label{bghffgghgkoooojkvhgjgjgfffhfgbfbhgkggjk}
\big\|DA_t(z)\big\|_{\mathcal{L}(Z;X^*)}\leq g\big(\|P\cdot
z\|_H\big)\cdot\Big(\|L_0\cdot z\|_{
V_0}^{q-2}+1\Big)\quad\quad\forall z\in Z,\;\forall t\in[0,T_0]\,,
\end{equation}
where $\tilde\mu(t)\in L^1\big((0,T_0);\R\big)$ is some nonnegative
function, $g(s):[0,+\infty)\to(0,+\infty)$ is some nondecreasing
function, $V_0$ is some
Banach space and $L_0:Z\to V_0$ is some compact linear operator.
Moreover, assume that $\Lambda_t$ and $A_t$ satisfy the following
monotonicity
condition:
\begin{multline}
\label{roststlambdglkFFyhuhtyui99999999hvjghjn} \Big<h,A_t(Q\cdot
h\big)+\Lambda_t \big(Q\cdot h\big)\Big>_{X\times X^*}\geq
\big(1/\bar C\big)\big\|Q\cdot h\big\|^q_Z- \bar C\big|L\cdot(Q\cdot
h)\big|^2_V -\mu(t)\Big(\big\|T\cdot h\big\|^{2}_H+1\Big)
\\
\forall h\in X\,,\;\forall t\in[0,T_0]\,,
\end{multline}
where $V$ is a given Banach space, $L\in\mathcal{L}(Z,V)$ is a given
compact operator,
$\mu(t)\in L^1\big((0,T_0);\R\big)$ is some
nonnegative function and $\bar C>0$ is some constant.
%
%
%
%
%
%
We also assume that $\Lambda_t(z)$ $A_t(z)$ are
Borel on the pair of variables $(z,t)$. Finally assume that there
exists a family of
Banach spaces $\{V_j\}_{j=1}^{+\infty}$ and a family  of compact
bounded linear operators $\{L_j\}_{j=1}^{+\infty}$, where $L_j:Z\to
V_j$, which satisfy the following condition:
\begin{itemize}
\item
If $\{h_n\}_{n=1}^{+\infty}\subset Z$ is a sequence and $h_0\in Z$,
are such that for every fixed $j$ $\lim_{n\to+\infty}L_j\cdot
h_n=L_j\cdot h_0$ strongly in $V_j$ and $P\cdot h_n\rightharpoonup
P\cdot h_0$ weakly in $H$, then for every fixed $t\in(0,T_0)$ we
have $\Lambda_t(h_n)\rightharpoonup \Lambda_t(h_0)$ weakly in $X^*$
and $D A_t(h_n)\to D A_t(h_0)$ strongly in $\mathcal{L}(Z,X^*)$.
\end{itemize}
Then for every $w_0\in H$ there exists $z(t)\in
L^q\big((0,T_0);Z\big)$ such that $w(t):=P\cdot z(t)\in
L^\infty\big((0,T_0);H\big)$, $v(t):=\widetilde T\cdot
\big(w(t)\big)\in W^{1,q^*}\big((0,T_0);X^*\big)$ and $z(t)$
satisfies the following equation
\begin{equation}\label{uravnllllgsnachlklimhjhghhjfhjfhgb}
\begin{cases}\frac{d v}{dt}(t)+A_t\big(z(t)\big) +\Lambda_t\big(z(t)\big)
=0\quad\text{for a.e.}\; t\in(0,T_0),\\
v(a)=\widetilde T\cdot w_0.
\end{cases}
\end{equation}
\end{theorem}

On section \ref{dkgfkghfhkljl} we give examples of the applications
of Theorems \ref{defHkkkkglkjjjgkjgkjgggkbmhmhgm} and
\ref{thhypppggghhhjgggjhjghjgh}, providing the existence results for
various classes of time dependent partial differential equations
including parabolic, hyperbolic, Shr\"{o}dinger and Navier-Stokes
systems.

\section{Notations and preliminaries}

%
Throughout the paper by linear space we mean a real linear space.
\begin{itemize}
\item For given normed space $X$ we denote by $X^*$ the dual space (the space of continuous (bounded) linear functionals from $X$ to $\R$).
\item For given $h\in X$ and $x^*\in X^*$ we denote by $\big<h,x^*\big>_{X\times X^*}$ the value in $\R$ of the functional $x^*$ on the vector $h$.
\item For given two normed linear spaces $X$ and $Y$ we denote by $\mathcal{L}(X;Y)$ the linear space of continuous (bounded) linear operators from $X$ to $Y$.
\item For given $A\in\mathcal{L}(X;Y)$  and $h\in X$ we denote by $A\cdot h\in Y$ the value of the operator $A$ at the point $h$.
\item We set
$\|A\|_{\mathcal{L}(X;Y)}=\sup\{\|A\cdot h\|_Y:\;h\in
X,\;\|h\|_X\leq 1\}$. Then it is well known that $\mathcal{L}(X;Y)$
will be a normed linear space. Moreover $\mathcal{L}(X;Y)$ will be a
Banach space if $Y$ is a Banach space.
\end{itemize}
\begin{definition}\label{2bdf}
Let $X$ and $Y$ be two normed linear spaces. We say that a function
$F:X\to Y$ is G\^{a}teaux differentiable at the point $x\in X$ if
there exists $A\in\mathcal{L}(X;Y)$ such that the following limit
exists in $Y$ and satisfy,
$$\lim\limits_{s\to 0}\frac{1}{s}\Big(F(x+sh)-F(x)\Big)=A\cdot h\quad\forall h\in X\,.$$
In this case we denote the operator $A$ by $DF(x)$ and the value $A\cdot h$ by $DF(x)\cdot h$.
\end{definition}

Next we remind some Definitions and Lemmas of \cite{PI}. Part of
them are well known. The proves of all the following Lemmas can be
found in \cite{PI}.

\begin{definition}\label{fdfjlkjjkkkkkllllkkkjjjhhhkkk}
Let $X$ and $Y$ be two normed linear spaces and $U\subset X$ be a
Borel subset. We say that the mapping $F(x):U\to Y$ is strongly
Borel if the following two conditions are satisfied.
\begin{itemize}
\item
$F$ is a Borel mapping i.e. for every Borel set $W\subset Y$, the
set $\{x\in U:\,F(x)\in W\}$ is also Borel.
\item For every separable subspace $X'\subset X$, the set $\{y\in Y:\,y=F(x),\; x\in U\cap
X'\}$ is also contained in some separable subspace of $Y$.
\end{itemize}
\end{definition}

\begin{definition}\label{3bdf}
For a given Banach space $X$ with the associated norm $\|\cdot\|_X$
and a real interval $(a,b)$ we denote by $L^q(a,b;X)$ the linear
space of (equivalence classes of) strongly measurable (i.e
equivalent to some strongly Borel mapping)
functions $f:(a,b)\to
X$ such that the functional
\begin{equation*}
\|f\|_{L^q(a,b;X)}:=
\begin{cases}
\Big(\int_a^b\|f(t)\|^q_X dt\Big)^{1/q}\quad\text{if }\;1\leq
q<\infty\\
{\text{es$\,$sup}}_{t\in (a,b)}\|f(t)\|_X\quad\text{if }\; q=\infty
\end{cases}
\end{equation*}
is finite. It is known that this functional defines a norm with
respect to which $L^q(a,b;X)$ becomes a Banach space. Moreover, if
$X$ is reflexive and $1<q<\infty$ then $L^q(a,b;X)$ will be a reflexive
space with the corresponding dual space $L^{q^*}(a,b;X^*)$,
where $q^*=q/(q-1)$.
It is also well known that the subspace of continuous functions $C^0([a,b];X)\subset L^q(a,b;X)$ is dense i.e. for every $f(t)\in L^q(a,b;X)$ there exists a sequence
$\{f_n(t)\}\subset C^0([a,b];X)$ such that $f_n(t)\to f(t)$ in the strong topology of $L^q(a,b;X)$.
\end{definition}
%
%
%
%
%
%
%
%
%

\begin{definition}\label{4bdf}
Let $X$ be a reflexive Banach space and let $(a,b)$ be a finite real
interval. We say that $v(t)\in L^q(a,b;X)$ belongs to
$W^{1,q}(a,b;X)$ if there exists $f(t)\in L^q(a,b;X)$ such that for
every $\delta(t)\in C^1\big((a,b);X^*\big)$ satisfying $\supp
\delta\subset\subset (a,b)$ we have
$$\int\limits_a^b\big<f(t),\delta(t)\big>_{X\times X^*}dt=-\int\limits_a^b\Big<v(t),\frac{d\delta}{dt}(t)\Big>_{X\times X^*}dt\,.$$
In this case we denote $f(t)$ by $v'(t)$ or by $\frac{d v}{dt}(t)$.
It is well known that if $v(t)\in W^{1,1}(a,b;X)$ then $v(t)$ is a
bounded and continuous function on $[a,b]$ (up to a redefining of
$v(t)$ on a subset of $[a,b]$ of Lebesgue measure zero), i.e.
$v(t)\in C^0\big([a,b];X\big)$ and for every $\delta(t)\in
C^1\big([a,b];X^*\big)$ and every subinterval
$[\alpha,\beta]\subset[a,b]$ we have
\begin{equation}\label{sobltr}
\int\limits_\alpha^\beta\bigg\{\Big<\frac{dv}{dt}(t),\delta(t)\Big>_{X\times
X^*}+\Big<v(t),\frac{d\delta}{dt}(t)\Big>_{X\times X^*}\bigg\}dt=
\big<v(\beta),\delta(\beta)\big>_{X\times
X^*}-\big<v(\alpha),\delta(\alpha)\big>_{X\times X^*}\,.
\end{equation}
\end{definition}
\begin{lemma}\label{vlozhenie}
Let $X$ and $Y$ be two reflexive Banach spaces,
$S\in\mathcal{L}(X,Y)$ be an injective inclusion (i.e. it satisfies
$\ker S=0$) and $(a,b)$ be a finite real interval. Then if $u(t)\in
L^q(a,b;X)$ is such that $v(t):=S\cdot u(t)\in W^{1,q}(a,b;Y)$ and
there exists $f(t)\in L^q(a,b;X)$ such that $\frac{d
v}{dt}(t)=S\cdot f(t)$ then $u(t)\in W^{1,q}(a,b;X)$ and $\frac{d
u}{dt}(t)=f(t)$.
\end{lemma}

\begin{definition}\label{5bdf}
Let $X$ be a Banach space. We say that a function $\Psi(x):X\to\R$ is convex (strictly convex) if for every $\lambda\in(0,1)$ and for every $x,y\in X$ s.t. $x\neq y$ we have
\begin{equation*}
\Psi\Big(\lambda x+(1-\lambda)y\Big)\;\;\leq\;\big(\,<\,\big)\;\;\;\;\lambda\Psi(x)+(1-\lambda)\Psi(y)\,.
\end{equation*}
It is well known that if $\Psi(x):X\to\R$ is a convex (strictly
convex) function which is G\^{a}teaux differentiable at every $x\in
X$ then for every $x,y\in X$ s.t. $x\neq y$ we have
\begin{equation}\label{conprmin}
\Psi(y)\;\;\geq\;\big(\,>\,\big)\;\;\;\;\Psi(x)+\Big<y-x,D\Psi(x)\Big>_{X\times X^*}\,,
\end{equation}
and
\begin{equation}\label{conprmonsv}
\Big<y-x,D\Psi(y)-D\Psi(x)\Big>_{X\times X^*}\;\;\geq\;\big(\,>\,\big)\;\;\;\;0\,,
\end{equation}
(remember that $D\Psi(x)\in X^*$). Furthermore, $\Psi$ is weakly
lower semicontinuous on $X$. Moreover, if some function
$\Psi(x):X\to\R$ is G\^{a}teaux differentiable at every $x\in X$ and
satisfy either \er{conprmin} or \er{conprmonsv} for every $x,y\in X$
s.t. $x\neq y$, then $\Psi(y)$ is convex (strictly convex).
\end{definition}

\begin{definition}\label{hfguigiugyuyhkjjhlkklkk}
Let $Z$ be a
Banach space and $Z^*$ be a corresponding dual space. We say that
the mapping $\Lambda(z):Z\to Z^*$ is monotone (strictly monotone) if
we have
\begin{equation}
\label{ftguhhhhihggjgjkjggkjgj}
\Big<y-z,\Lambda(y)-\Lambda(z)\Big>_{Z\times Z^*}\,\geq\,(>)\,
0\quad\forall\, y\neq z\in Z\,.
\end{equation}
\end{definition}
\begin{definition}\label{hfguigiugyuyhkjjh}
Let $Z$ be a
Banach space and $Z^*$ be a corresponding dual space. We say that
the mapping $\Lambda(z):Z\to Z^*$ is pseudo-monotone if for every
sequence $\{z_n\}_{n=1}^{+\infty}\subset Z$, satisfying
\begin{equation}
\label{ftguhhhhikk} z_n\rightharpoonup z\;\;\text{weakly
in}\;\;Z\quad\quad\text{and}\quad\quad\limsup_{n\to+\infty}\Big<z_n-z,\Lambda(z_n)\Big>_{Z\times
Z^*}\leq 0
\end{equation}
we have
\begin{equation}
\label{ftguhhhhihggjgjk}
\liminf_{n\to+\infty}\Big<z_n-y,\Lambda(z_n)\Big>_{Z\times
Z^*}\geq\Big<z-y,\Lambda(z)\Big>_{Z\times Z^*}\quad\forall y\in Z\,.
\end{equation}
\end{definition}
\begin{lemma}\label{hhhhhhhhhhhhhhhhhhiogfydtdtyd}
Let $Z$ be a
Banach space and $Z^*$ be a corresponding dual space. Then the
mapping $\Lambda(z):Z\to Z^*$ is pseudo-monotone if and only if it
satisfies the following conditions:
\begin{itemize}
\item[{\bf(i)}] For every sequence $\{z_n\}_{n=1}^{+\infty}\subset Z$, such
that $z_n\rightharpoonup z$ weakly in $Z$ we have
\begin{equation}
\label{ftguhhhhikkjhjhjkjkkkkkkk}
\liminf_{n\to+\infty}\Big<z_n-z,\Lambda(z_n)\Big>_{Z\times Z^*}\geq
0\,.
\end{equation}
\item[{\bf(ii)}] If for some
sequence $\{z_n\}_{n=1}^{+\infty}\subset Z$, such that
$z_n\rightharpoonup z$ weakly in $Z$ we have
\begin{equation}
\label{ftguhhhhikkjhjhjkjkkkkkkkjjjjhgghhh}
\lim_{n\to+\infty}\Big<z_n-z,\Lambda(z_n)\Big>_{Z\times Z^*}= 0\,,
\end{equation}
then $\Lambda(z_n)\rightharpoonup \Lambda(z)$ weakly$^*$ in $Z^*$.
\end{itemize}
\end{lemma}
\begin{remark}\label{fyyjfjyhfgjgghgjkgkjgkgggfhfhkkk}
It is trivially follows from Lemma
\ref{hhhhhhhhhhhhhhhhhhiogfydtdtyd} that, if
$\Lambda_1(z),\Lambda_2(z):Z\to Z^*$ are two pseudo-monotone
mappings, then the sum of them, $\Lambda_1(z)+\Lambda_2(z)$ is also
a pseudo-monotone mapping.
\end{remark}

\begin{lemma}\label{hhhhhhhhhhhhhhhhhhiogfydtdtydjkgkgk}
Let $Z$ be a
Banach space and $Z^*$ be a corresponding dual space. Assume that
the mapping $\Lambda(z):Z\to Z^*$ is monotone. Moreover assume that
$\Lambda(z):Z\to Z^*$ is continuous for every $z\in Z$ or more
generally the function $\zeta_{z,h}(t):\R\to\R$, defined by
\begin{equation}\label{fguyfuyfugyuguhkg}
\zeta_{z,h}(t):=\Big<h,\Lambda\big(z-t h\big)\Big>_{Z\times
Z^*}\quad\forall z,h\in Z\,,\quad\forall t\in\R\,,
\end{equation}
is continuous on $t$ for every $z,h\in Z$. Then the mapping
$\Lambda(z)$ is pseudo-monotone.
\end{lemma}

\begin{lemma}\label{vbnhjjm}
Let $Y$ and $Z$ be two reflexive Banach spaces. Furthermore, let
$S\in \mathcal{L}(Y;Z)$  be an injective operator (i.e. it satisfies
$\ker S=\{0\}$) and  let $S^*\in \mathcal{L}(Z^*;Y^*)$ be the
corresponding adjoint operator, which satisfies
\begin{equation}\label{tildetdall}
\big<y,S^*\cdot z^*\big>_{Y\times Y^*}:=\big<S\cdot
y,z^*\big>_{Z\times Z^*}\quad\quad\text{for every}\; z^*\in
Z^*\;\text{and}\;y\in Y\,.
\end{equation}
Next assume that $a,b\in\R$ s.t. $a<b$. Let $w(t)\in
L^\infty(a,b;Y)$ be such that the function $v:[a,b]\to Z$ defined by
$v(t):=S\cdot \big(w(t)\big)$ belongs to $W^{1,q}(a,b;Z)$ for some
$q\geq 1$. Then we can redefine $w$ on a subset of $[a,b]$ of
Lebesgue measure zero, so that $w(t)$ will be $Y$-weakly
continuous in $t$ on $[a,b]$ ( i.e. $w\in C_w^0(a,b;Y)$ ). Moreover,
for every $a\leq \alpha<\beta\leq b$ and for every $\delta(t)\in
C^1\big([a,b];Z^*\big)$ we will have
\begin{equation}\label{eqmult}
\int\limits_\alpha^\beta\bigg\{\Big<\frac{dv}{dt}(t),\delta(t)\Big>_{Z\times
Z^*}+\Big<v(t),\frac{d\delta}{dt}(t)\Big>_{Z\times Z^*}\bigg\}dt=
\big<w(\beta),S^*\cdot\delta(\beta)\big>_{Y\times
Y^*}-\big<w(\alpha),S^*\cdot\delta(\alpha)\big>_{Y\times Y^*}\,.
\end{equation}
\end{lemma}

%
%
%
%
%
%
%
%

\begin{definition}\label{7bdf}
Let $X$ be a reflexive Banach space and $X^*$ the corresponding dual
space. Furthermore, let $H$ be a Hilbert space and $T\in
\mathcal{L}(X,H)$ be an injective (i.e. it satisfies $\ker T=\{0\}$)
inclusion operator such that its image is dense on $H$. Then we call
the triple $\{X,H,X^*\}$ an evolution triple with the corresponding
inclusion operator $T$. Throughout this paper we assume the space
$H^*$ be equal to $H$ (remember that $H$ is a Hilbert space) but in
general we don't associate $X^*$ with $X$ even in the case where $X$
is a Hilbert space (and thus $X^*$ will be isomorphic to $X$).
Further we define the bounded linear operator $\widetilde{T}\in
\mathcal{L}(H;X^*)$ by the formula
\begin{equation}\label{tildet}
\big<x,\widetilde{T}\cdot y\big>_{X\times X^*}:=\big<T\cdot x,y\big>_{H\times H}\quad\quad\text{for every}\; y\in H\;\text{and}\;x\in X\,.
\end{equation}
In particular
$\|\widetilde{T}\|_{\mathcal{L}(H;X^*)}=\|T\|_{\mathcal{L}(X;H)}$
and since we assumed that the image of $T$ is dense in $H$ we deduce
that $\ker \widetilde{T}=\{0\}$ and so $\widetilde{T}$ is an
injective operator. So $\widetilde{T}$ is an inclusion of $H$ to
$X^*$ and the operator $I:=\widetilde{T}\circ T$ is an injective
inclusion of $X$ to $X^*$. Furthermore, clearly
\begin{equation}\label{tildethlhjhghjf}
\big<x,I\cdot z\big>_{X\times X^*}=\big<T\cdot x,T\cdot
z\big>_{H\times H}=\big<z,I\cdot x\big>_{X\times
X^*}\quad\quad\text{for every}\; x,z\in X\,.
\end{equation}
So $I\in \mathcal{L}(X,X^*)$ is self-adjoint operator. Moreover, $I$
is strictly positive, since
\begin{equation}\label{tildethlhjhghjffgfhfh}
\big<x,I\cdot x\big>_{X\times X^*}=\|T\cdot
x\|^2_H>0\quad\quad\forall x\neq 0\in X\,.
\end{equation}
\end{definition}
\begin{lemma}\label{hdfghdiogdiofg}
Let $X$ be a reflexive Banach space and $X^*$ the corresponding dual
space. Furthermore, let $I\in \mathcal{L}(X,X^*)$ be a self-adjoint
and strictly positive operator. i.e.
\begin{equation}\label{tildethlhjhghjfvvjhjhj}
\big<x,I\cdot z\big>_{X\times X^*}=\big<z,I\cdot x\big>_{X\times
X^*}\quad\quad\text{for every}\; x,z\in X\,,
\end{equation}
and
\begin{equation}\label{tildethlhjhghjffgfhfhhffkgh}
\big<x,I\cdot x\big>_{X\times X^*}>0\quad\quad\forall x\neq 0\in
X\,.
\end{equation}
Then there exists a Hilbert space $H$ and an injective operator
$T\in \mathcal{L}(X,H)$ (i.e. $\ker T=\{0\}$), whose image is dense
in $H$, and such that if we consider the operator $\widetilde{T}\in
\mathcal{L}(H;X^*)$, defined by the formula \er{tildet},
then we will have
\begin{equation}\label{tildethlhjhghjffgfhfhhffkghbjhjkhjjk}
(\widetilde{T}\circ T)\cdot x=I\cdot x\quad\quad\forall x\in X\,.
\end{equation}
I.e. $\{X,H,X^*\}$ is an evolution triple with the corresponding
inclusion operator $T\in \mathcal{L}(X;H)$, as it was defined in
Definition \ref{7bdf}, together with the corresponding operator
$\widetilde{T}\in \mathcal{L}(H;X^*)$, defined as in \er{tildet},
and $I\equiv \widetilde{T}\circ T$.
\end{lemma}
Next as a particular case of Lemma \ref{vbnhjjm} we have the
following Corollary.
\begin{corollary}\label{vbnhjjmcor}
Let $\{X,H,X^*\}$ be an evolution triple with the corresponding
inclusion operator $T\in \mathcal{L}(X;H)$ as it was defined in
Definition \ref{7bdf} together with the corresponding operator
$\widetilde{T}\in \mathcal{L}(H;X^*)$ defined as in \er{tildet} and
let $a,b\in\R$ s.t. $a<b$. Let $w(t)\in L^\infty(a,b;H)$ be such
that the function $v:[a,b]\to X^*$ defined by $v(t):=\widetilde
T\cdot \big(w(t)\big)$ belongs to $W^{1,q}(a,b;X^*)$ for some $q\geq
1$. Then we can redefine $w$ on a subset of $[a,b]$ of Lebesgue
measure zero, so that $w(t)$ will be $H$-weakly continuous in $t$ on
$[a,b]$ ( i.e. $w\in C_w^0(a,b;H)$ ). Moreover, for every $a\leq
\alpha<\beta\leq b$ and for every $\delta(t)\in
C^1\big([a,b];X\big)$ we will have
\begin{equation}\label{eqmultcor}
\int\limits_\alpha^\beta\bigg\{\Big<\delta(t), \frac{dv}{dt}(t)\Big>_{X\times X^*}+\Big<\frac{d\delta}{dt}(t), v(t)\Big>_{X\times X^*}\bigg\}dt=
\big<T\cdot\delta(\beta),w(\beta)\big>_{H\times H}-\big<T\cdot\delta(\alpha),w(\alpha)\big>_{H\times H}\,.
\end{equation}
\end{corollary}

\begin{lemma}\label{lem2}
Let $\{X,H,X^*\}$ be an evolution triple with the corresponding
inclusion operator $T\in \mathcal{L}(X;H)$ as it was defined in
Definition \ref{7bdf} together with the corresponding operator
$\widetilde{T}\in \mathcal{L}(H;X^*)$ defined as in \er{tildet} and
let $a,b\in\R$ s.t. $a<b$. Let $u(t)\in L^q(a,b;X)$ for some $q>1$
such that the function $v(t):[a,b]\to X^*$ defined by $v(t):=I\cdot
\big(u(t)\big)$ belongs to $W^{1,q^*}(a,b;X^*)$ for $q^*:=q/(q-1)$,
where we denote $I:=\widetilde T\circ T:\,X\to X^*$. Then the
function $w(t):[a,b]\to H$ defined by $w(t):=T\cdot \big(u(t)\big)$
belongs to $L^\infty(a,b;H)$ and for every subinterval
$[\alpha,\beta]\subset[a,b]$ we have
\begin{equation}\label{energyravenstvo}
\int_\alpha^\beta\Big<u(t),\frac{dv}{dt}(t)\Big>_{X\times X^*}dt=\frac{1}{2}\Big(\|w(\beta)\|_H^2
-\|w(\alpha)\|_H^2\Big)\,,
\end{equation}
up to a redefinition of $w(t)$ on a subset of $[a,b]$ of Lebesgue measure zero, such that $w$ is $H$-weakly continuous, as it was stated in
Corollary \ref{vbnhjjmcor}.
\end{lemma}

We will need in the sequel the following compactness results.
%
%
%
%
%
%
%
%
%
%
\begin{lemma}\label{ComTem1PP}
Let $X$, $Y$ $Z$
be three Banach spaces, such that $X$
is a reflexive space. Furthermore, let $T\in \mathcal{L}(X;Y)$ and
$S\in \mathcal{L}(X;Z)$
be bounded
linear operators. Moreover assume that $S$
is an injective inclusion (i.e. it satisfies $\ker S=\{0\}$)
and $T$ is a
compact operator. Assume that $a,b\in\R$ such that $a<b$, $1\leq
q<+\infty$ and $\{u_n(t)\}\subset L^q(a,b;X)$ is a bounded in
$L^q(a,b;X)$ sequence of functions, such that the functions
$v_n(t):(a,b)\to Z$, defined by $v_n(t):=S\cdot\big(u_n(t)\big)$,
belongs to $L^\infty(a,b;Z)$, the sequence $\{v_n(t)\}$ is bounded
in $L^\infty(a,b;Z)$ and for a.e. $t\in(a,b)$ we have
\begin{equation}\label{fghffhdpppplllkkkll}
v_n(t)\rightharpoonup v(t)\quad\text{weakly in}\;
Z\;\text{as}\;n\to+\infty\,.
\end{equation}
Then,
\begin{equation}\label{staeae1glllukjkjkojkl}
\big\{T\cdot\big(u_n(t)\big)\big\}\quad\text{converges strongly in }
L^q(a,b;Y)\,.
\end{equation}
\end{lemma}

\begin{lemma}\label{helplemlll}
Let $Z$ be a reflexive Banach space and let
$\big\{v_n(t)\big\}_{n=1}^{+\infty}\subset W^{1,1}(a,b;Z)$ be a
sequence of functions, bounded in $W^{1,1}(a,b;Z)$.
Then, $\big\{v_n(t)\big\}_{n=1}^{+\infty}$ is bounded in
$L^\infty(a,b;Z)$ and, up to a subsequence, we have
\begin{equation}\label{fghffhdpppplllkkkllkklkl}
v_n(t)\rightharpoonup v(t)\quad\text{weakly in}\;
Z\;\text{as}\;\;n\to+\infty\,,\quad\text{for a.e}\;\,t\in(a,b)\,.
\end{equation}
\end{lemma}
As a direct consequence of Lemma \ref{ComTem1PP} and Lemma
\ref{helplemlll} we have the following Lemma.
\begin{lemma}\label{ComTem1}
Let $X$, $Y$ and $Z$ be three Banach spaces, such that $X$ and $Z$
are reflexive.
Furthermore, let $T\in
\mathcal{L}(X;Y)$ and $S\in \mathcal{L}(X;Z)$ be bounded linear
operators. Moreover assume that $S$ is an injective inclusion (i.e.
it satisfies $\ker S=\{0\}$) and $T$
is a compact operator. Assume that $a,b\in\R$ such that $a<b$,
$1\leq q<+\infty$ and $\{u_n(t)\}\subset L^q(a,b;X)$ is a bounded in
$L^q(a,b;X)$ sequence of functions, such that the functions
$v_n(t):(a,b)\to Z$, defined by $v_n(t):=S\cdot\big(u_n(t)\big)$,
belongs to $W^{1,1}(a,b;Z)$ and the sequence
$\big\{\frac{dv_n}{dt}(t)\big\}$ is bounded in $L^1(a,b;Z)$. Then,
up to a subsequence,
\begin{equation}\label{staeae1gllluk}
\big\{T\cdot\big(u_n(t)\big)\big\}\quad\text{converges strongly in }
L^q(a,b;Y)\,.
\end{equation}
\end{lemma}

The following simple embedding result was proven in the Appendix of
\cite{PI}:
\begin{lemma}\label{hilbcomban}
Let $X$ be a separable Banach space. Then there exists a separable
Hilbert space $Y$ and a bounded linear inclusion operator $S\in
\mathcal{L}(Y;X)$ such that $S$ is injective (i.e. $\ker S=\{0\}$),
the image of $S$ is dense in $X$ and moreover, $S$ is a compact
operator.
\end{lemma}

We will need in the sequel the following simple well known result,
see the Appendix of \cite{PI}:
\begin{lemma}\label{Aplem1}
Let $X$, $Y$ and $Z$ be three Banach spaces, such that $X$
is a reflexive space. Furthermore, let $T\in \mathcal{L}(X;Y)$ and
$S\in \mathcal{L}(X;Z)$ be bounded linear operators. Moreover assume
that $S$ is an injective inclusion (i.e. it satisfies $\ker
S=\{0\}$) and $T$
is a compact operator. Then for each $\,\e>0$ there exists some
constant $c_\e>0$ depending on $\e$ (and on the spaces $X$, $Y$, $Z$
and on the operators $T$, $S$) such that
\begin{equation}\label{Eqetaenerap}
\big\|T\cdot h\big\|_Y\leq\e\big\|h\big\|_X+c_\e\big\|S\cdot
h\big\|_Z\quad\forall h\in X\,.
\end{equation}
\end{lemma}
%
%
%
%
%
%
In the future we also need the following simple Lemma (see
\cite{PI}):
\begin{lemma}\label{Legendre}
Let $X$ be a reflexive Banach space and let
$\Psi(x):X\to[0,+\infty)$ be a convex function which is G\^{a}teaux
differentiable on every $x\in X$, satisfies $\Psi(0)=0$ and
satisfies
\begin{equation}\label{rost}
0\leq \Psi(x)\leq C_0\,\|x\|_X^q+C_0\quad\forall x\in X\,,
\end{equation}
for some $q>1$ and $C_0>0$. Then for some $\bar C_0$, that depends
only on $C_0$ and $q$ from \er{rost}, we have
\begin{equation}\label{rostgrad}
\|D\Psi(x)\|_{X^*}\leq \bar C_0\|x\|_X^{q-1}+\bar C_0\quad\forall
x\in X\,.
\end{equation}
\end{lemma}
\section{The Existence results}
%
%
%
%
%
%
\begin{lemma}\label{thhypppggghhhjgggnew}
Let $X$ and $Z$ be reflexive Banach spaces and $X^*$ and $Z^*$ be
the corresponding dual spaces. Furthermore let $H$ be a Hilbert
space. Suppose that $Q\in \mathcal{L}(X,Z)$ is an injective
inclusion operator (i.e. it satisfies $\ker Q=\{0\}$) such that its
image is dense on $Z$. Furthermore, suppose that $P\in
\mathcal{L}(Z,H)$ is an injective inclusion operator such that its
image is dense on $H$. Let $T\in \mathcal{L}(X,H)$ be defined by
$T:=P\circ Q$. So $\{X,H,X^*\}$ is an evolution triple with the
corresponding inclusion operator $T\in \mathcal{L}(X;H)$ as it was
defined in Definition \ref{7bdf} together with the corresponding
operator $\widetilde{T}\in \mathcal{L}(H;X^*)$ defined as in
\er{tildet}.
Next let $a,b\in\R$ be such that $a<b$ and $q\geq 2$. Furthermore,
for every $t\in[a,b]$ let $\Psi_t(x):X\to[0,+\infty)$ be a convex
function which is G\^{a}teaux differentiable at every $x\in X$,
satisfies $\Psi_t(0)=0$ and satisfies the growth condition
\begin{equation}\label{roststglkagagsldnew}
(1/C)\,\|x\|_X^q-C\leq 0\leq\Psi_t(x)\leq C\,\|x\|_X^q+C\quad\forall
x\in X,\;\forall t\in[a,b]\,,
\end{equation}
for some $C>0$. We also assume that $\Psi_t(x)$ is Borel on the pair
of variables $(x,t)$. Next, for every $t\in[a,b]$ let
$\Lambda_t(z):Z\to X^*$ be a function which is G\^{a}teaux
differentiable on every $z\in Z$. Assume that
it satisfies the following bound
\begin{equation}\label{bghffgghgkoooojkvhgjgjgfffhnew}
\big\|\Lambda_t(z)\big\|_{X^*}\leq g\big(\|P\cdot
z\|_H\big)\cdot\Big(\|z\|_{Z}^{q-1}+\mu^{\frac{q-1}{q}}(t)\Big)\quad\quad\forall
z\in Z,\;\forall t\in[a,b]\,,
\end{equation}
where $g(s):[0,+\infty)\to(0,+\infty)$ is some nondecreasing
function and $\mu(t)\in L^1(a,b;\R)$ is some nonnegative function.
Moreover, assume that $\Lambda_t$ satisfies the following positivity
condition
\begin{multline}
\label{roststlambdglkFFyhuhtyui99999999newzzz} \Big<h,\Lambda_t
\big(Q\cdot h\big)\Big>_{X\times X^*}\geq \big(1/\bar
C\big)\big\|Q\cdot h\big\|^q_Z-\bar C\big\|L\cdot(Q\cdot
h)\big\|^2_V-\tilde \mu(t)
\Big(\big\|T\cdot h\big\|^{2}_H+1\Big)\quad\quad \forall h\in
X\,,\;\forall t\in[a,b]\,,
\end{multline}
where $V$ is a given Banach space, $L\in \mathcal{L}(Z,V)$ is a
given compact linear operator, $\bar C>0$ is some constant and
$\tilde\mu(t)\in L^1(a,b;\R)$ is some nonnegative function. We also
assume that $\Lambda_t(z)$ is strongly Borel on the pair of
variables $(z,t)$. Furthermore, let
$\{w^{(0)}_n\}_{n=1}^{\infty}\subset H$ be such that $w^{(0)}_n\to
w_0$ strongly in $H$ and let $\e_n>0$ be such that $\e_n\to 0$ as
$n\to+\infty$. Moreover, assume that $u_n(t)\in L^q(a,b;X)$ be such
that $v_n(t):=(\widetilde T\circ T)\cdot u_n(t)\in
W^{1,q^*}(a,b;X^*)$, where $q^*=q/(q-1)$, and $u_n(t)$ is a solution
to
\begin{equation}\label{uravnllllglkhghkgkghfsldnew}
\begin{cases}
\frac{d v_n}{dt}(t)+\Lambda_t\big(z_n(t)\big)+\e_n
D\Psi_t\big(u_n(t)\big)\quad \text{for a.e.}\;
t\in(a,b)
\\w_n(a)=w^{(0)}_n,
\end{cases}
\end{equation}
where $w_n(t):=T\cdot u_n(t)$, $z_n(t):=Q\cdot u_n(t)$ and we assume
that $w_n(t)$ is $H$-weakly continuous on $[a,b]$, as it was stated
in Corollary \ref{vbnhjjmcor}. Then there exists $z(t)\in
L^q(a,b;Z)$ and $\bar\Lambda(t)\in L^{q^*}(a,b;X^*)$ such that
$w(t):=P\cdot z(t)\in L^\infty(a,b;H)$, $v(t):=\widetilde T\cdot
w(t)\in W^{1,q^*}(a,b;X^*)$, $w(t)$ is $H$-weakly continuous on
$[a,b]$, up to a subsequence, we have
\begin{equation}\label{uravnllllgsnachlkagagsldnewgiukhoikljjlk}
\begin{cases}
z_n(t)\rightharpoonup z(t)\quad\text{weakly in}\;\;L^q(a,b;Z)\\
\frac{dv_n}{dt}(t)\rightharpoonup \frac{dv}{dt}(t)\quad\text{weakly in}\;\;L^{q^*}(a,b;X^*)\\
\Lambda_t \big(z_n(t)\big)\rightharpoonup \bar\Lambda(t)\quad\text{weakly in}\;\;L^{q^*}(a,b;X^*)\\
w_n(t)\rightharpoonup w(t)\quad\text{weakly in}\;\;H\quad\text{for
every fixed}\;\;t\in[a,b],\\
\big\{w_n(t)\big\}_{n=1}^{+\infty}\;\;\text{is bounded
in}\;\;L^\infty(a,b;H),
\end{cases}
\end{equation}
and $z(t)$ satisfies the following equation
\begin{equation}\label{uravnllllgsnachlklimhjhghnew}
\begin{cases}\frac{d v}{dt}(t)+\bar\Lambda(t)
=0\quad\text{for a.e.}\; t\in(a,b)\,,\\
w(a)=w_0\,.
\end{cases}
\end{equation}
Moreover,
\begin{equation}\label{vyifyurturfurfuyrfgyukkknew}
\frac{1}{2}\big\|w(t)\big\|^2_H+\limsup_{n\to+\infty}\bigg(\int_a^t\Big<u_n(s),\Lambda_s
\big(z_n(s)\big)\Big>_{X\times
X^*}ds\bigg)\leq\frac{1}{2}\big\|w_0\big\|^2_H\quad\forall
t\in[a,b].
\end{equation}
\end{lemma}
\begin{proof}
By Lemma \ref{Aplem1}, there exists a constant $K>0$ such that
$$\big\|L\cdot z\big\|^2_V\leq \frac{1}{2\bar C^2}\,\|z\|^2_Z+K\big\|P\cdot z\big\|^2_H
\quad\forall z\in Z.$$ Plugging it into
\er{roststlambdglkFFyhuhtyui99999999newzzz} we obtain
\begin{multline}\label{MonotonegnewhghghghghSLDhhh}
\Big<h,\Lambda_t \big(Q\cdot h\big)\Big>_{X\times X^*}\geq
\frac{1}{2\bar C}\Big(2\|Q\cdot h\|^q_Z -\|Q\cdot h\|^2_Z\Big)
-
\big(\tilde\mu(t)+\bar C K
\big)
\Big(\|T\cdot h\|^{2}_H+1\Big)\\
\geq\frac{1}{2\bar C}\,\|Q\cdot h\|^q_Z
-
\big(\tilde\mu(t)+\tilde K\big)
\Big(\|T\cdot h\|^{2}_H+1\Big)
\quad\forall h\in X,\;\forall t\in[a,b],
\end{multline}
where $\tilde K>0$ is a constant. Thus, denoting
$\bar\mu(t):=\big(\tilde\mu(t)+\tilde K\big)\in L^1(a,b;\R)$, we
obtain:
\begin{multline}
\label{roststlambdglkFFyhuhtyui99999999new} \Big<h,\Lambda_t
\big(Q\cdot h\big)\Big>_{X\times X^*}\geq \big(1/2\bar
C\big)\big\|Q\cdot h\big\|^q_Z-\bar \mu(t)
\Big(\big\|T\cdot h\big\|^{2}_H+1\Big)\quad\quad \forall h\in
X\,,\;\forall t\in[a,b].
\end{multline}
On the other hand, by \er{uravnllllglkhghkgkghfsldnew} we deduce
\begin{multline}\label{uravnllllgsnachlkhkkkksldnew}
\int_a^t\Big<u_n(s),\frac{d v_n}{dt}(s)\Big>_{X\times X^*}ds
+\int_a^t\bigg<u_n(s),\Lambda_s\big(z_n(s)\big)\bigg>_{X\times
X^*}ds\\+\e_n
\int_a^t\Big<u_n(s),D\Psi_t\big(u_n(s)\big)\Big>_{X\times
X^*}ds=0\quad\quad\forall t\in[a,b]\,.
\end{multline}
However, since by Lemma \ref{lem2} we have
\begin{equation*}
\int_a^t\Big<u_n(s),\frac{dv_n}{dt}(s)\Big>_{X\times
X^*}ds=\frac{1}{2}\Big(\big\|w_n(t)\big\|_H^2-\big\|w^{(0)}_n\big\|_H^2\Big)\,,
\end{equation*}
using \er{uravnllllgsnachlkhkkkksldnew} we obtain
\begin{multline}\label{uravnllllgsnachlkhkkkksledsldnew}
\frac{1}{2}\big\|w_n(t)\big\|_H^2
+\int_a^t\bigg<u_n(s),\Lambda_s\big(z_n(s)\big)\bigg>_{X\times
X^*}ds\\+\e_n
\int_a^t\Big<u_n(s),D\Psi_t\big(u_n(s)\big)\Big>_{X\times
X^*}ds=\frac{1}{2}\big\|w^{(0)}_n\big\|_H^2\quad\quad\forall
t\in[a,b]\,.
\end{multline}
However, since $\Psi_t(\cdot)$ is convex and since
$\Psi_t(\cdot)\geq 0$, $\Psi_t(0)=0$ and then also $D\Psi_t(0)=0$,
we have
\begin{equation}\label{dkgvhhhlhsldnew}
\Big<u_n(t),D\Psi_t\big(u_n(t)\big)\Big>_{X\times
X^*}\geq\Psi_t\big(u_n(t)\big)\geq 0\quad\quad\forall t\in(a,b)\,.
\end{equation}
Therefore, using \er{dkgvhhhlhsldnew},
from \er{uravnllllgsnachlkhkkkksledsldnew} we deduce
\begin{equation}
\label{neravenstvopolsldnewht}
\e_n\int_a^t\Psi_s\big(u_n(s)\big)ds+\frac{1}{2}\big\|w_n(t)\big\|_H^2+\int_a^t\bigg<u_n(s),\Lambda_s\big(z_n(s)\big)\bigg>_{X\times
X^*}ds\leq\frac{1}{2}\big\|w^{(0)}_n\big\|_H^2\quad\quad\forall
t\in[a,b]\,.
\end{equation}
and in particular,
\begin{equation}
\label{neravenstvopolsldnew}
\frac{1}{2}\big\|w_n(t)\big\|_H^2+\int_a^t\bigg<u_n(s),\Lambda_s\big(z_n(s)\big)\bigg>_{X\times
X^*}ds\leq\frac{1}{2}\big\|w^{(0)}_n\big\|_H^2\quad\quad\forall
t\in[a,b]\,.
\end{equation}
Thus, inserting \er{roststlambdglkFFyhuhtyui99999999new}  into
\er{neravenstvopolsldnewht} we deduce that:
\begin{equation}\label{neravenstvopoldffkllgfddsdrdtydtsldhkhkllnew}
\big\|w_n(t)\big\|_H^2+\e_n\int_a^t\Psi_s\big(u_n(s)\big)ds+\int_a^t\big\|z_n(s)\big\|_Z^q
ds\leq C_2\int_a^t \bar\mu(s)\big\|w_n(s)\big\|^2_Hds+C_2
\quad\forall t\in[a,b],
\end{equation}
where $C_2>0$ is a constant. In particular,
\begin{equation}\label{neravenstvopoldffkllgfddsdrdtydtsldnew}
\big\|w_n(t)\big\|_H^2\leq
C_2\int_a^t\bar\mu(s)\big\|w_n(s)\big\|^2_Hds+C_2\quad\quad\forall
t\in[a,b]
\end{equation}
Thus,
\begin{multline}\label{bjgjggg88888888889999dgvg99newsldnew}
\frac{d}{dt}\Bigg\{\exp{\bigg(-C_2\int_a^t
\bar\mu(s)ds\bigg)}\int_a^t\bar\mu(s)\|w_n(s)\|^{2}_H ds\Bigg\}\leq
C_2\bar\mu(t)\exp{\bigg(-C_2\int_a^t \bar\mu(s)ds\bigg)}\\
\leq C_2\bar\mu(t)\quad\quad\text{for a.e. } t\in[a,b]\quad\forall
n\in\mathbb{N}\,,
\end{multline}
and thus
\begin{multline}\label{bjgjggg88888888889999dgvg99mcv9999newsldnew}
\int_a^t\bar\mu(s)\|w_n(s)\|^{2}_H ds\leq  C_2\exp{\bigg(C_2\int_a^t
\bar\mu(s)ds\bigg)}\cdot\int_a^t\bar\mu(s)ds\leq\\
C_2\exp{\bigg(C_2\int_a^b
\bar\mu(s)ds\bigg)}\cdot\int_a^b\bar\mu(s)ds\quad\quad\forall
t\in[a,b]\quad\forall n\in\mathbb{N}\,.
\end{multline}
Then by \er{bjgjggg88888888889999dgvg99mcv9999newsldnew} from
\er{neravenstvopoldffkllgfddsdrdtydtsldnew} we obtain that the
sequence $\{w_n(t)\}$ is bounded in $L^\infty(a,b;H)$. Then by
\er{neravenstvopoldffkllgfddsdrdtydtsldhkhkllnew} we deduce that
that sequence $\{z_n(t)\}$ is bounded in $L^q(a,b;Z)$. Moreover, by
\er{bghffgghgkoooojkvhgjgjgfffhnew} we obtain that
$\Lambda_t\big(z_n(t)\big)$ is bounded in $L^{q^*}(a,b;X^*)$.
Therefore in particular, up to a subsequence we have
\begin{equation}\label{boundnes77889999hhhhhh999ghg78999sldnew}
\begin{cases}
z_n(t)\rightharpoonup z(t)\quad\text{weakly in}\;\;L^q(a,b;Z)\,,\\
w_n(t)\rightharpoonup w(t)\quad\text{weakly in}\;\;L^q(a,b;H)\,,\\
v_n(t)\rightharpoonup v(t)\quad\text{weakly in}\;\;L^q(a,b;X^*)\,,\\
\Lambda_t\big(z_n(t)\big)\rightharpoonup
\bar\Lambda(t)\quad\text{weakly in}\;\;L^{q^*}(a,b;X^*)\,,
\end{cases}
\end{equation}
where
$w(t):=P\cdot z(t)$, $v(t):=\widetilde{T}\cdot w(t)$. Next plugging
\er{boundnes77889999hhhhhh999ghg78999sldnew} into
\er{neravenstvopoldffkllgfddsdrdtydtsldhkhkllnew} and using
the fact that $\{w_n(t)\}$ is bounded in $L^\infty(a,b;H)$,
we deduce
\begin{equation}\label{ffghk9999999jjjjjj9999sldnew}
\e_n\int_a^t\Psi_s\big(u_n(s)\big)ds\leq C_4\,,
\end{equation}
where $C_4$ is a constant. Then using
\er{roststglkagagsldnew}
we deduce from \er{ffghk9999999jjjjjj9999sldnew},
\begin{equation}\label{ffghk9999999jjjjjj9999hhjh9999sldnew}
\e_n \int_a^b\big\|u_n(s)\big\|^q_X ds\leq C_5\,.
\end{equation}
Next by \er{rostgrad} in Lemma \ref{Legendre}, for some $\bar C>0$
we have
\begin{equation}\label{roststthigjkjkljkljkljkljklsdsldnew}
\Big\|D\Psi_t\big(u_n(t)\big)\Big\|_{X^*}\leq \bar
C\big\|u_n(t)\big\|^{q-1}_X+\bar C\quad\quad\forall t\in(a,b)\,,
\end{equation}
and then
\begin{equation}\label{roststthigjkjkljkljkljkljklsdnhjsldnew}
\Big\|D\Psi_t\big(u_n(t)\big)\Big\|^{q^*}_{X^*}\leq \bar
C_0\big\|u_n(t)\big\|^{q}_X+\bar C_0\quad\quad\forall t\in(a,b)\,,
\end{equation}
Thus plugging \er{roststthigjkjkljkljkljkljklsdnhjsldnew} into
\er{ffghk9999999jjjjjj9999hhjh9999sldnew} we deduce
\begin{equation}\label{fdfddsreaeuyuioi9999jojk99sldnew}
\int_a^b\Big\|\e_n D\Psi_t\big(u_n(s)\big)\Big\|^{q^*}_{X^*}ds\leq
\hat C\e^{1/(q-1)}_n\,.
\end{equation}
So,
\begin{equation}\label{bjhdvghfhkgthiojosa999999sldnew}
\lim\limits_{n\to+\infty}\Big\|\e_n
D\Psi_t\big(u_n(t)\big)\Big\|_{L^{q^*}(a,b;X^*)}=0\,.
\end{equation}
On the other hand by \er{uravnllllglkhghkgkghfsldnew} and by
Corollary \ref{vbnhjjmcor} for any $\beta\in[a,b]$ and every
$\delta(t)\in C^1\big([a,b];X\big)$ we have
\begin{multline}\label{eqmultcorjjjjkkkkklllllprodddkkksldnew}
\Big<T\cdot\delta(\beta),w_n(\beta)\Big>_{H\times
H}-\Big<T\cdot\delta(a),w^{(0)}_n\Big>_{H\times
H}-\int\limits_a^\beta\bigg<\frac{d\delta}{dt}(t),
v_n(t)\bigg>_{X\times X^*}dt+\\ \int\limits_a^\beta\bigg<\delta(t),
\e_n D\Psi_t\big(u_n(t)\big)\bigg>_{X\times X^*}dt
+\int\limits_a^\beta\bigg<\delta(t),
\Lambda_t\big(z_n(t)\big)
\bigg>_{X\times X^*}dt=0\,,
\end{multline}
Letting $n$ tend to $+\infty$ in
\er{eqmultcorjjjjkkkkklllllprodddkkksldnew} and using
\er{boundnes77889999hhhhhh999ghg78999sldnew},
\er{bjhdvghfhkgthiojosa999999sldnew} and the fact that $w^{(0)}_n\to
w_0$ in $H$ we obtain
\begin{multline}\label{eqmultcorjjjjkkkkklllllprodddhgggjgjsldnew}
\lim\limits_{n\to+\infty}\Big<T\cdot\delta(\beta),w_n(\beta)\Big>_{H\times
H}-\Big<T\cdot\delta(a),w_0\Big>_{H\times
H}\\-\int\limits_a^\beta\bigg<\frac{d\delta}{dt}(t),
v(t)\bigg>_{X\times X^*}dt +\int\limits_a^\beta\Big<\delta(t),
\bar\Lambda(t)
\Big>_{X\times X^*}dt=0\,,
\end{multline}
for every $\delta(t)\in C^1\big([a,b];X\big)$.
In particular, for every $\delta(t)\in C^1\big([a,b];X\big)$ such
that $\delta(b)=0$ we have
\begin{equation}\label{eqmultcorjjjjkkkkklllllprodddhgggjgjstsldnew}
-\Big<T\cdot\delta(a),w_0\Big>_{H\times
H}-\int\limits_a^b\bigg<\frac{d\delta}{dt}(t), v(t)\bigg>_{X\times
X^*}dt +\int\limits_a^b\Big<\delta(t),
\bar\Lambda(t)
\Big>_{X\times X^*}dt=0\,.
\end{equation}
Thus in particular $\frac{d v}{dt}(t)=-\bar\Lambda(t)\in
L^{q^*}(a,b;X^*)$ and so $v(t)\in W^{1,q^*}(a,b;X^*)$. Then, since
$\{w_n(t)\}$ is bounded in $L^\infty(a,b;H)$, we have $w(t)\in
L^\infty(a,b;H)$ and thus, as before, we can redefine $w$ on a
subset of $[a,b]$ of Lebesgue measure zero, so that $w(t)$ will be
$H$-weakly continuous in $t$ on $[a,b]$ and by
\er{eqmultcorjjjjkkkkklllllprodddhgggjgjstsldnew} we will have
$w(a)=w_0$.
So $w(t)$ is a solution to the following equation
\begin{equation}\label{uravnllllgsnachlklimhjhghghtynew}
\begin{cases}\frac{d v}{dt}(t)+
\bar\Lambda(t)=0\quad\text{for a.e.}\; t\in(a,b)\,.\\
w(a)=w_0\,,
\end{cases}
\end{equation}
Thus in particular for any $\beta\in[a,b]$ and every $\delta(t)\in
C^1\big([a,b];X\big)$ we have
\begin{multline}\label{eqmultcorjjjjkkkkklllllprodddhgggjgjsldkkhnew}
\Big<T\cdot\delta(\beta),w(\beta)\Big>_{H\times
H}-\Big<T\cdot\delta(a),w_0\Big>_{H\times
H}\\-\int\limits_a^\beta\bigg<\frac{d\delta}{dt}(t),
v(t)\bigg>_{X\times X^*}dt+\int\limits_a^\beta\Big<\delta(t),
\bar\Lambda(t)
\Big>_{X\times X^*}dt=0\,.
\end{multline}
Plugging \er{eqmultcorjjjjkkkkklllllprodddhgggjgjsldkkhnew} into
\er{eqmultcorjjjjkkkkklllllprodddhgggjgjsldnew} we deduce
\begin{equation}\label{eqmultcorjjjjkkkkklllllprodddhgggjgjsldukulnew}
\lim\limits_{n\to+\infty}\Big<T\cdot x,w_n(\beta)\Big>_{H\times
H}=\Big<T\cdot x,w(\beta)\Big>_{H\times H}\quad\forall x\in
X\;\forall\beta\in[a,b]\,.
\end{equation}
Therefore, since the image of $T$ has dense range in $H$ and
$\{w_n(t)\}$ is bounded in $L^\infty(a,b;H)$ we deduce that
\begin{equation}\label{tjhuyjkuiuliluipoiopnew}
w_n(t)\rightharpoonup w(t)\quad\text{weakly
in}\;\;H\quad\quad\forall t\in[a,b]\,.
\end{equation}
Next by \er{boundnes77889999hhhhhh999ghg78999sldnew},
\er{bjhdvghfhkgthiojosa999999sldnew},
\er{uravnllllglkhghkgkghfsldnew} and
\er{uravnllllgsnachlklimhjhghghtynew} we obtain
\begin{equation}\label{uravnllllgsnachlkagagsldnewgiukhoikljjlkgjkgkj}
\frac{dv_n}{dt}(t)\rightharpoonup \frac{dv}{dt}(t)\quad\text{weakly
in}\;\;L^{q^*}(a,b;X^*).
\end{equation}
So we established \er{uravnllllgsnachlkagagsldnewgiukhoikljjlk} and
\er{uravnllllgsnachlklimhjhghnew}. Finally, since $w^{(0)}_n\to w_0$
strongly in $H$, by plugging \er{tjhuyjkuiuliluipoiopnew} into
\er{neravenstvopolsldnew} we obtain
\er{vyifyurturfurfuyrfgyukkknew}.
\end{proof}
%
%
%
%
%
%
As a consequence of Lemma \ref{thhypppggghhhjgggnew} in some
particular case we have the following:
\begin{corollary}\label{thhypppggghhhjgggnewnnn}
Let $X$ and $Z$ be reflexive Banach spaces and $X^*$ and $Z^*$ be
the corresponding dual spaces. Furthermore let $H$ be a Hilbert
space. Suppose that $Q\in \mathcal{L}(X,Z)$ is an injective
inclusion operator (i.e. it satisfies $\ker Q=\{0\}$) such that its
image is dense on $Z$. Furthermore, suppose that $P\in
\mathcal{L}(Z,H)$ is an injective inclusion operator such that its
image is dense on $H$. Let $T\in \mathcal{L}(X,H)$ be defined by
$T:=P\circ Q$ and let $\widetilde P\in\mathcal{L}(H;Z^*)$ be defined
by
\begin{equation}\label{tildetPfhh}
\big<z,\widetilde{P}\cdot y\big>_{Z\times Z^*}:=\big<P\cdot
z,y\big>_{H\times H}\quad\quad\text{for every}\; y\in
H\;\text{and}\;z\in Z\,.
\end{equation} So $\{X,H,X^*\}$ is an evolution triple with the corresponding
inclusion operator $T\in \mathcal{L}(X;H)$ as it was defined in
Definition \ref{7bdf} together with the corresponding operator
$\widetilde{T}\in \mathcal{L}(H;X^*)$ defined as in \er{tildet}.
Moreover, $\{Z,H,Z^*\}$ is another evolution triple with the
corresponding inclusion operator $P\in \mathcal{L}(Z;H)$ together
with the corresponding operator $\widetilde{P}\in
\mathcal{L}(H;Z^*)$. Next let $a,b\in\R$ be such that $a<b$ and
$q\geq 2$. Furthermore, for every $t\in[a,b]$ let
$\Psi_t(x):X\to[0,+\infty)$ be a convex function which is
G\^{a}teaux differentiable at every $x\in X$, satisfies
$\Psi_t(0)=0$ and satisfies the growth condition
\begin{equation}\label{roststglkagagsldnewnnn}
(1/C)\,\|x\|_X^q-C\leq\Psi_t(x)\leq C\,\|x\|_X^q+C\quad\forall x\in
X,\;\forall t\in[a,b]\,,
\end{equation}
for some $C>0$. We also assume that $\Psi_t(x)$ is Borel on the pair
of variables $(x,t)$. Furthermore, for every $t\in[a,b]$ let
$\Lambda_t(z):Z\to Z^*$ be a function which is G\^{a}teaux
differentiable on every $z\in Z$. Assume that
it satisfies the following bound
\begin{equation}\label{bghffgghgkoooojkvhgjgjgfffhnewnnn}
\big\|\Lambda_t(z)\big\|_{Z^*}\leq g\big(\|P\cdot
z\|_H\big)\cdot\Big(\|z\|_{Z}^{q-1}+\mu^{\frac{q-1}{q}}(t)\Big)\quad\quad\forall
z\in Z,\;\forall t\in[a,b]\,,
\end{equation}
where $g(s):[0,+\infty)\to(0,+\infty)$ is some nondecreasing
function and $\mu(t)\in L^1(a,b;\R)$ is some nonnegative function.
Moreover, assume that $\Lambda_t$ satisfies the following positivity
condition
\begin{multline}
\label{roststlambdglkFFyhuhtyui99999999newnnn} \Big< h,\Lambda_t
\big(h\big)\Big>_{Z\times Z^*}\geq \big(1/\bar C\big)\big\|
h\big\|^q_Z-\bar C\big\|L\cdot h\big\|^2_V-\bar \mu(t)
\Big(\big\|P\cdot h\big\|^{2}_H+1\Big)\quad\quad \forall h\in
Z\,,\;\forall t\in[a,b]\,,
\end{multline}
where $V$ is a given Banach space, $L\in \mathcal{L}(Z,V)$ is a
given compact linear operator, $\bar C>0$ is some constant and
$\bar\mu(t)\in L^1(a,b;\R)$ is some nonnegative function. We also
assume that $\Lambda_t(z)$ is strongly Borel on the pair of
variables $(z,t)$. Moreover, assume the following compactness
property: for every sequence
$\big\{\sigma_n(t)\big\}_{n=1}^{+\infty}\subset L^q(a,b;Z)$, such
that $\big\{P\cdot \sigma_n(t)\big\}_{n=1}^{+\infty}\subset
L^\infty(a,b;H)$, $\sigma_n(t)\rightharpoonup \sigma(t)$ weakly in
$L^q(a,b;Z)$, $\big\{P\cdot \sigma_n(t)\big\}_{n=1}^{+\infty}$ is
bounded in $L^\infty(a,b;H)$ and $P\cdot \sigma_n(t)\rightharpoonup
P\cdot \sigma(t)$ weakly in $H$ for a.e. $t\in(a,b)$, the inequality
\begin{equation}
\label{ftguhhhhikkjhjhjkjkkkkkkkcbcbncvccghghjnnnn}
\liminf_{n\to+\infty}\int_a^b\Big<\sigma_n(t)-\sigma(t),\Lambda_t\big(\sigma_n(t)\big)\Big>_{Z\times
Z^*}dt\leq 0,
\end{equation}
necessarily implies that, up to a subsequence, we have
$\Lambda_t\big(\sigma_n(t)\big)\rightharpoonup\Lambda_t\big(\sigma(t)\big)$
weakly in $L^{q^*}(a,b;Z^*)$.
Next, let $\{w^{(0)}_n\}_{n=1}^{\infty}\subset H$ be such that
$w^{(0)}_n\to w_0$ strongly in $H$ and let $\e_n>0$ be such that
$\e_n\to 0$ as $n\to+\infty$. Moreover, assume that $u_n(t)\in
L^q(a,b;X)$ be such that $v_n(t):=(\widetilde T\circ T)\cdot
u_n(t)\in W^{1,q^*}(a,b;X^*)$, where $q^*=q/(q-1)$, and $u_n(t)$ is
a solution to
\begin{equation}\label{uravnllllglkhghkgkghfsldnewnnn}
\begin{cases}
\frac{d v_n}{dt}(t)+Q^*\cdot\Lambda_t\big(z_n(t)\big)+\e_n
D\Psi_t\big(u_n(t)\big)\quad \text{for a.e.}\; t\in(a,b)
\\w_n(a)=w^{(0)}_n,
\end{cases}
\end{equation}
where $Q^*\in \mathcal{L}(Z^*;X^*)$ is the adjoint to $Q$ operator,
$w_n(t):=T\cdot u_n(t)$, $z_n(t):=Q\cdot u_n(t)$ and we assume that
$w_n(t)$ is $H$-weakly continuous on $[a,b]$, as it was stated in
Corollary \ref{vbnhjjmcor}. Then, there exists $z(t)\in L^q(a,b;Z)$,
such that $w(t):=P\cdot
z(t)\in L^\infty(a,b;H)$, $\zeta(t):=\widetilde P\cdot w(t)\in
W^{1,q^*}(a,b;Z^*)$, $v(t):=\widetilde T\cdot w(t)\in
W^{1,q^*}(a,b;X^*)$, $w(t)$ is $H$-weakly continuous on $[a,b]$, up
to a subsequence, we have
\begin{equation}\label{uravnllllgsnachlkagagsldnewgiukhoikljjlknnnzzz}
\begin{cases}
z_n(t)\rightharpoonup z(t)\quad\text{weakly in}\;\;L^q(a,b;Z)\\
\frac{dv_n}{dt}(t)\rightharpoonup \frac{dv}{dt}(t)\quad\text{weakly in}\;\;L^{q^*}(a,b;X^*)\\
\Lambda_t \big(z_n(t)\big)\rightharpoonup \Lambda_t\big(z(t)\big)\quad\text{weakly in}\;\;L^{q^*}(a,b;Z^*)\\
w_n(t)\rightharpoonup w(t)\quad\text{weakly in}\;\;H\quad\text{for
every fixed}\;\;t\in[a,b],\\
\big\{w_n(t)\big\}_{n=1}^{+\infty}\;\;\text{is bounded
in}\;\;L^\infty(a,b;H),
\end{cases}
\end{equation}
and $z(t)$ satisfies the following equation
\begin{equation}\label{uravnllllgsnachlklimhjhghnewnnn}
\begin{cases}\frac{d \zeta}{dt}(t)+\Lambda_t\big(z(t)\big)
=0\quad\text{for a.e.}\; t\in(a,b)\,,\\
w(a)=w_0\,.
\end{cases}
\end{equation}
Moreover,
\begin{equation}\label{vyifyurturfurfuyrfgyukkknewnnn}
\frac{1}{2}\big\|w(t)\big\|^2_H+\int_a^t\Big<z(s),\Lambda_s
\big(z(s)\big)\Big>_{Z\times
Z^*}ds=\frac{1}{2}\big\|w_0\big\|^2_H\quad\forall t\in[a,b].
\end{equation}
\end{corollary}
\begin{proof}
Using Lemma \ref{thhypppggghhhjgggnew}, we deduce that there exists
$z(t)\in L^q(a,b;Z)$ and $\bar\Lambda(t)\in L^{q^*}(a,b;Z^*)$ such
that $w(t):=P\cdot z(t)\in L^\infty(a,b;H)$, $v(t):=\widetilde
T\cdot w(t)\in W^{1,q^*}(a,b;X^*)$, $w(t)$ is $H$-weakly continuous
on $[a,b]$, up to a subsequence, we have
\begin{equation}\label{uravnllllgsnachlkagagsldnewgiukhoikljjlknnn}
\begin{cases}
z_n(t)\rightharpoonup z(t)\quad\text{weakly in}\;\;L^q(a,b;Z)\\
\frac{dv_n}{dt}(t)\rightharpoonup \frac{dv}{dt}(t)\quad\text{weakly in}\;\;L^{q^*}(a,b;X^*)\\
\Lambda_t \big(z_n(t)\big)\rightharpoonup \bar\Lambda(t)\quad\text{weakly in}\;\;L^{q^*}(a,b;Z^*)\\
w_n(t)\rightharpoonup w(t)\quad\text{weakly in}\;\;H\quad\text{for
every fixed}\;\;t\in[a,b],\\
\big\{w_n(t)\big\}_{n=1}^{+\infty}\;\;\text{is bounded
in}\;\;L^\infty(a,b;H),
\end{cases}
\end{equation}
and $z(t)$ satisfies the following equation
\begin{equation}\label{uravnllllgsnachlklimhjhghnewnnnkkk}
\begin{cases}\frac{d v}{dt}(t)+Q^*\cdot\bar\Lambda(t)
=0\quad\text{for a.e.}\; t\in(a,b)\,,\\
w(a)=w_0\,.
\end{cases}
\end{equation}
Moreover,
\begin{equation}\label{vyifyurturfurfuyrfgyukkknewnnnkkk}
\frac{1}{2}\big\|w(t)\big\|^2_H+\limsup_{n\to+\infty}\bigg(\int_a^t\Big<z_n(s),\Lambda_s
\big(z_n(s)\big)\Big>_{Z\times
Z^*}ds\bigg)\leq\frac{1}{2}\big\|w_0\big\|^2_H\quad\forall
t\in[a,b].
\end{equation}
Next using \er{uravnllllgsnachlklimhjhghnewnnnkkk} with Lemma
\ref{vlozhenie} we deduce that  $\zeta(t):=\widetilde P\cdot w(t)\in
W^{1,q^*}(a,b;Z^*)$. Moreover, by Lemma \ref{lem2} we have
\begin{equation}\label{vyifyurturfurfuyrfgyukkknewnnnkkkhlhhlhhggiu}
\frac{1}{2}\big\|w(t)\big\|^2_H+\int_a^t\Big<z(s),\bar\Lambda(s)\Big>_{Z\times
Z^*}ds=\frac{1}{2}\big\|w_0\big\|^2_H\quad\forall t\in[a,b].
\end{equation}
Thus plugging \er{vyifyurturfurfuyrfgyukkknewnnnkkkhlhhlhhggiu} into
\er{vyifyurturfurfuyrfgyukkknewnnnkkk} and using
\er{uravnllllgsnachlkagagsldnewgiukhoikljjlknnn} gives
\begin{multline}\label{vyifyurturfurfuyrfgyukkknewnnnkkkjhjk}
\limsup_{n\to+\infty}\bigg(\int_a^b\Big<z_n(t),\Lambda_t
\big(z_n(t)\big)\Big>_{Z\times
Z^*}dt\bigg)\leq\int_a^b\Big<z(t),\bar\Lambda(t)\Big>_{Z\times
Z^*}dt\\=\lim_{n\to+\infty}\bigg(\int_a^b\Big<z(t),\Lambda_t
\big(z_n(t)\big)\Big>_{Z\times Z^*}dt\bigg).
\end{multline}
So
\begin{equation*}
\limsup_{n\to+\infty}\int_a^b\Big<z_n(t)-z(t),\Lambda_t\big(z_n(t)\big)\Big>_{Z\times
Z^*}dt\leq 0,
\end{equation*}
that implies $\bar\Lambda(t)=\Lambda_t\big(z(t)\big)$. This
completes the proof.
\end{proof}

\begin{definition}\label{hfguigiugyuyhkjjhjjjk}
Let $\{X,H,X^*\}$ be an evolution triple with the corresponding
inclusion operator $T\in \mathcal{L}(X;H)$ as it was defined in
Definition \ref{7bdf}. Furthermore let $(a,b)$ be a real interval,
$q>1$ and $q^*:=q/(q-1)$.
We say that the mapping $\Gamma(u):\big\{u\in L^q(a,b;X):\;T\cdot
u\in L^\infty(a,b;H)\big\}\to L^{q^*}(a,b;X^*)\equiv
\big\{L^q(a,b;X)\}^*$ is weakly pseudo-monotone if for every
sequence $\big\{u_n(t)\big\}_{n=1}^{+\infty}\subset L^q(a,b;X)$,
such that $\big\{T\cdot u_n(t)\big\}_{n=1}^{+\infty}\subset
L^\infty(a,b;H)$, $u_n(t)\rightharpoonup u(t)$ weakly in
$L^q(a,b;X)$, $\big\{T\cdot u_n(t)\big\}_{n=1}^{+\infty}$ is bounded
in $L^\infty(a,b;H)$ and for a.e. $t\in(a,b)$ $T\cdot
u_n(t)\rightharpoonup T\cdot u(t)$ weakly in $H$, the following
conditions are satisfied:
\begin{itemize}
\item
\begin{equation}
\label{ftguhhhhikkjhjhjkjkkkkkkkcbcbncvccghghj}
\liminf_{n\to+\infty}\Big<u_n-u,\Gamma(u_n)\Big>_{L^q(a,b;X)\times
L^{q^*}(a,b;X^*)}\geq 0\,.
\end{equation}
\item If we have
\begin{equation}
\label{ftguhhhhikkjhjhjkjkkkkkkkjjjjhgghhhghgggh}
\lim_{n\to+\infty}\Big<u_n-u,\Gamma(u_n)\Big>_{L^q(a,b;X)\times
L^{q^*}(a,b;X^*)}=0\,,
\end{equation}
then $\Gamma(u_n)\rightharpoonup \Gamma(u)$ weakly in
$L^{q^*}(a,b;X^*)$.
\end{itemize}
\end{definition}
%
%
%
%
\begin{remark}\label{fyyjfjyhfgjggh}
By Lemmas \ref{hhhhhhhhhhhhhhhhhhiogfydtdtyd} and
\ref{hhhhhhhhhhhhhhhhhhiogfydtdtydjkgkgk} we know that if the
mapping $\Gamma(u):L^q(a,b;X)\to L^{q^*}(a,b;X^*)$ is
pseudo-monotone, then $\Gamma(u)$ is weakly pseudo-monotone.
\end{remark}
\begin{remark}\label{fyyjfjyhfgjgghgjkgkjgkgggfhfh}
It is trivially follows from the definition that if
$\Gamma_1(u),\Gamma_2(u):\big\{u\in L^q(a,b;X):\;T\cdot u\in
L^\infty(a,b;H)\big\}\to L^{q^*}(a,b;X^*)$ are two weakly
pseudo-monotone mappings, then the sum of them,
$\Gamma_1(u)+\Gamma_2(u)$ is also a weakly pseudo-monotone mapping.
\end{remark}
%
%
%
%
\begin{lemma}\label{hkjghiohioujpohgkgk}
Let $\{X,H,X^*\}$ be an evolution triple with the corresponding
inclusion operator $T\in \mathcal{L}(X;H)$ as it was defined in
Definition \ref{7bdf} together with the corresponding operator
$\widetilde{T}\in \mathcal{L}(H;X^*)$ defined as in \er{tildet}.
Furthermore, let $q\geq 2$ and for every $t\in[a,b]$ let
$\Theta_t(x):X\to X^*$ be a function which
satisfies the growth condition:
\begin{equation}\label{roststlambdglkagagfgfffhkjlh}
\|\Theta_t(x)\|_{X^*}\leq g\big(\|T\cdot
x\|_H\big)\,\big(\|x\|_X^{q-1}+\mu^{\frac{q-1}{q}}(t)\big)\quad\forall
x\in X,\;\forall t\in[a,b]\,,
\end{equation}
for some nondecreasing function $g(s):[0,+\infty)\to(0,+\infty)$ and
some nonnegative function $\mu(t)\in L^{1}\big(a,b;\R\big)$.
We
also assume that $\Theta_t(x)$
is strongly Borel on the pair of variables $(x,t)$ and satisfies the
following monotonicity
condition
\begin{multline}\label{Monotonegnewagagkjljhjhghjfjfghfhhjghgj}
\Big<x,\Theta_t(x)\Big>_{X\times X^*}\geq\frac{1}{\hat C}\,\|x\|_X^q
-
\Big(\|x\|^p_X+\tilde\mu^{\frac{p}{2}}(t)\Big)\cdot\tilde\mu^{\frac{2-p}{2}}(t)\Big(\|T\cdot
x\|^{(2-p)}_H+1\Big)
\quad\forall x\in X,\;\forall
t\in[a,b],
\end{multline}
where $p\in[0,2)$ and $\hat C>0$ are some constants
and $\tilde\mu(t)\in L^1\big(a,b;\R\big)$ is some nonnegative
function. Finally, assume that for a.e. fixed $t\in(a,b)$ the
function $\Theta_t(x):X\to X^*$ is pseudo-monotone (see Definition
\ref{hfguigiugyuyhkjjh}). Then, the mapping $\Gamma(u):\big\{u\in
L^q(a,b;X):\;T\cdot u\in L^\infty(a,b;H)\big\}\to L^{q^*}(a,b;X^*)$,
defined by
\begin{multline}\label{fuduftfdftffyvjhvhvghghhjghjg}
\bigg<h(t),\Gamma\big(u(t)\big)\bigg>_{L^q(a,b;X)\times
L^{q^*}(a,b;X^*)}:=\int_a^b\Big<h(t),\Theta_t\big(u(t)\big)\Big>_{X\times
X^*}dt\\ \forall\, u(t)\in \big\{\bar u(t)\in L^q(a,b;X):\;T\cdot
\bar u(t)\in L^\infty(a,b;H)\big\},\;\forall\, h(t)\in
L^{q}(a,b;X)\,,
\end{multline}
is weakly pseudo-monotone (see Definition
\ref{hfguigiugyuyhkjjhjjjk}).
\end{lemma}
\begin{proof}
Consider a sequence $\big\{u_n(t)\big\}_{n=1}^{+\infty}\subset
L^q(a,b;X)$, such that $\big\{T\cdot
u_n(t)\big\}_{n=1}^{+\infty}\subset L^\infty(a,b;H)$,
$u_n(t)\rightharpoonup u(t)$ weakly in $L^q(a,b;X)$, $\big\{T\cdot
u_n(t)\big\}_{n=1}^{+\infty}$ is bounded in $L^\infty(a,b;H)$ and
$T\cdot u_n(t)\rightharpoonup T\cdot u(t)$ weakly in $H$ for a.e.
$t\in(a,b)$. Then, by \er{roststlambdglkagagfgfffhkjlh} and
\er{Monotonegnewagagkjljhjhghjfjfghfhhjghgj}, for every $h(t)\in
L^{q}(a,b;X)$ there exists $\eta_h(t)\in L^1\big(a,b;\R\big)$, such
that
\begin{equation}\label{fuduftfdftffyvjhvhvghghhjghjgjkkjhkj}
\Big<u_n(t)-h(t),\Theta_t\big(u_n(t)\big)\Big>_{X\times X^*}\geq
\frac{1}{2\hat C}\,\big\|u_n(t)\big\|_X^q+\eta_h(t)\quad\quad\forall
t\in[a,b].
\end{equation}
Therefore, by Fatou's Lemma
\begin{multline}\label{fuduftfdftffyvjhvhvghghhjghjgjkkjhkjiuitjjf}
\liminf_{n\to+\infty}\int_a^b\Big<u_n(t)-h(t),\Theta_t\big(u_n(t)\big)\Big>_{X\times
X^*}dt\geq\\
\int_a^b\bigg(\liminf_{n\to+\infty}\Big<u_n(t)-h(t),\Theta_t\big(u_n(t)\big)\Big>_{X\times
X^*}\bigg)dt
\quad\quad\forall\,h(t)\in L^{q}(a,b;X).
\end{multline}
Then,  assuming
$\liminf_{n\to+\infty}\int_a^b\Big<u_n(t)-u(t),\Theta_t\big(u_n(t)\big)\Big>_{X\times
X^*}dt<+\infty$, by taking $h(t)=u(t)$ in
\er{fuduftfdftffyvjhvhvghghhjghjgjkkjhkjiuitjjf}, we deduce:
\begin{multline}\label{fuduftfdftffyvjhvhvghghhjghjgjkkjhkjiuitjjfjkkjggkkggg}
\int_a^b\bigg(\liminf_{n\to+\infty}\Big<u_n(t)-u(t),\Theta_t\big(u_n(t)\big)\Big>_{X\times
X^*}\bigg)dt\leq\\
\liminf_{n\to+\infty}\int_a^b\Big<u_n(t)-u(t),\Theta_t\big(u_n(t)\big)\Big>_{X\times
X^*}dt<+\infty.
\end{multline}
In particular, for a.e. $t\in(a,b)$ there exists a strictly
increasing subsequence
$\big\{n^{(t)}_k\}_{k=1}^{+\infty}\subset\mathbb{N}$ such that we
have
\begin{equation}\label{fuduftfdftffyvjhvhvghghhjghjgjkkjhkjjhhffjhfjfjfj}
\lim_{k\to+\infty}\Big<u_{n^{(t)}_k}(t)-u(t),\Theta_t\big(u_{n^{(t)}_k}(t)\big)\Big>_{X\times
X^*}=\liminf_{n\to+\infty}\Big<u_n(t)-u(t),\Theta_t\big(u_n(t)\big)\Big>_{X\times
X^*}<+\infty.
\end{equation}
Therefore, by \er{fuduftfdftffyvjhvhvghghhjghjgjkkjhkj}, for a.e.
fixed $t\in(a,b)$ the sequence
$\{u_{n^{(t)}_k}(t)\}_{k=1}^{+\infty}$ is bounded in $X$. On the
other hand, $T\cdot u_n(t)\rightharpoonup T\cdot u(t)$ weakly in $H$
for a.e. $t\in(a,b)$. Thus, since $T$ is an injective operator, we
obtain that, for a.e. fixed $t\in(a,b)$
$u_{n^{(t)}_k}(t)\rightharpoonup u(t)$ weakly in $X$. Therefore,
since for a.e. fixed $t\in(a,b)$ the function $\Theta_t(x):X\to X^*$
is pseudo-monotone, using
\er{fuduftfdftffyvjhvhvghghhjghjgjkkjhkjjhhffjhfjfjfj} and Lemma
\ref{hhhhhhhhhhhhhhhhhhiogfydtdtyd}, for a.e. $t\in(a,b)$ we deduce:
\begin{equation}\label{fuduftfdftffyvjhvhvghghhjghjgjkkjhkjjhhffjhfjfjfjkghggghhhf}
\liminf_{n\to+\infty}\Big<u_n(t)-u(t),\Theta_t\big(u_n(t)\big)\Big>_{X\times
X^*}=\lim_{k\to+\infty}\Big<u_{n^{(t)}_k}(t)-u(t),\Theta_t\big(u_{n^{(t)}_k}(t)\big)\Big>_{X\times
X^*}\geq 0.
\end{equation}
Plugging it into
\er{fuduftfdftffyvjhvhvghghhjghjgjkkjhkjiuitjjfjkkjggkkggg}, we
infer
\begin{multline}\label{fuduftfdftffyvjhvhvghghhjghjgjkkjhkjiuitjjfjkkjggkkggghklhlhhjhk}
\liminf_{n\to+\infty}\int_a^b\Big<u_n(t)-u(t),\Theta_t\big(u_n(t)\big)\Big>_{X\times
X^*}dt\geq\\
\int_a^b\bigg(\liminf_{n\to+\infty}\Big<u_n(t)-u(t),\Theta_t\big(u_n(t)\big)\Big>_{X\times
X^*}\bigg)dt\geq 0.
\end{multline}
Moreover, obviously in the case that
$\liminf_{n\to+\infty}\int_a^b\Big<u_n(t)-u(t),\Theta_t\big(u_n(t)\big)\Big>_{X\times
X^*}dt=+\infty$, the first inequality in
\er{fuduftfdftffyvjhvhvghghhjghjgjkkjhkjiuitjjfjkkjggkkggghklhlhhjhk}
still holds. So
\begin{equation}
\label{ftguhhhhikkjhjhjkjkkkkkkkcbcbncvccghghjghkjjh}
\liminf_{n\to+\infty}\Big<u_n-u,\Gamma(u_n)\Big>_{L^q(a,b;X)\times
L^{q^*}(a,b;X^*)}\geq 0\,.
\end{equation}

 Next assume that,
\begin{equation*}
\lim_{n\to+\infty}\Big<u_n-u,\Gamma(u_n)\Big>_{L^q(a,b;X)\times
L^{q^*}(a,b;X^*)}=0\,.
\end{equation*}
Then, plugging it into
\er{fuduftfdftffyvjhvhvghghhjghjgjkkjhkjiuitjjfjkkjggkkggghklhlhhjhk}
we deduce,
\begin{multline}\label{gythjfjfvjhffhj}
\int_a^b\bigg(\liminf_{n\to+\infty}\Big<u_n(t)-u(t),\Theta_t\big(u_n(t)\big)\Big>_{X\times
X^*}\bigg)dt=\lim_{n\to+\infty}\int_a^b\Big<u_n(t)-u(t),\Theta_t\big(u_n(t)\big)\Big>_{X\times
X^*}dt=0\,.
\end{multline}
On the other hand, plugging \er{gythjfjfvjhffhj} into
\er{fuduftfdftffyvjhvhvghghhjghjgjkkjhkjjhhffjhfjfjfjkghggghhhf}, we
deduce:
\begin{equation}\label{fuduftfdftffyvjhvhvghghhjghjgjkkjhkjjhhffjhfjfjfjkghggghhhfjgggj}
\liminf_{n\to+\infty}\Big<u_n(t)-u(t),\Theta_t\big(u_n(t)\big)\Big>_{X\times
X^*}=0\quad\text{for a.e.}\;\,t\in(a,b).
\end{equation}
Therefore,
\begin{equation}\label{fuduftfdftffyvjhvhvghghhjghjgjkkjhkjjhhffjhfjfjfjkghggghhhfjgggjjfff}
\lim_{n\to+\infty}\Bigg(\min\bigg\{\;0\;,\;\Big<u_n(t)-u(t),\Theta_t\big(u_n(t)\big)\Big>_{X\times
X^*}\bigg\}\Bigg)=0\quad\text{for a.e.}\;\,t\in(a,b).
\end{equation}
Then using \er{fuduftfdftffyvjhvhvghghhjghjgjkkjhkj} and the
dominated convergence theorem, by
\er{fuduftfdftffyvjhvhvghghhjghjgjkkjhkjjhhffjhfjfjfjkghggghhhfjgggjjfff}
we deduce
\begin{equation}\label{fuduftfdftffyvjhvhvghghhjghjgjkkjhkjjhhffjhfjfjfjkghggghhhfjgggjjfffhdhdddghjffjfj}
\lim_{n\to+\infty}\int_a^b\Bigg(\min\bigg\{\;0\;,\;\Big<u_n(t)-u(t),\Theta_t\big(u_n(t)\big)\Big>_{X\times
X^*}\bigg\}\Bigg)dt=0.
\end{equation}
Thus plugging
\er{fuduftfdftffyvjhvhvghghhjghjgjkkjhkjjhhffjhfjfjfjkghggghhhfjgggjjfffhdhdddghjffjfj}
into \er{gythjfjfvjhffhj} we obtain
\begin{equation}\label{fuduftfdftffyvjhvhvghghhjghjgjkkjhkjjhhffjhfjfjfjkghggghhhfjgggjjfffhdhdddghjffjfjfjfj}
\lim_{n\to+\infty}\int_a^b\Bigg(\max\bigg\{\;0\;,\;\Big<u_n(t)-u(t),\Theta_t\big(u_n(t)\big)\Big>_{X\times
X^*}\bigg\}\Bigg)dt=0.
\end{equation}
So by
\er{fuduftfdftffyvjhvhvghghhjghjgjkkjhkjjhhffjhfjfjfjkghggghhhfjgggjjfffhdhdddghjffjfjfjfj}
and
\er{fuduftfdftffyvjhvhvghghhjghjgjkkjhkjjhhffjhfjfjfjkghggghhhfjgggjjfffhdhdddghjffjfj}
we deduce
\begin{equation}\label{fuduftfdftffyvjhvhvghghhjghjgjkkjhkjjhhffjhfjfjfjkghggghhhfjgggjjfffhdhdddghjffjfjfjfjlkhkl}
\lim_{n\to+\infty}\int_a^b\bigg|\Big<u_n(t)-u(t),\Theta_t\big(u_n(t)\big)\Big>_{X\times
X^*}\bigg|dt=0.
\end{equation}
Therefore, up to a subsequence, we have
\begin{equation}\label{fuduftfdftffyvjhvhvghghhjghjgjkkjhkjjhhffjhfjfjfjkghggghhhfjgggjjfffgghghgkhggkj}
\lim_{n\to+\infty}\Big<u_n(t)-u(t),\Theta_t\big(u_n(t)\big)\Big>_{X\times
X^*}=0\quad\text{for a.e.}\;\,t\in(a,b).
\end{equation}
Furthermore, using the fact that $u_n(t)\rightharpoonup u(t)$ weakly
in $L^q(a,b;X)$ and \er{roststlambdglkagagfgfffhkjlh}, we obtain
that there exists $\widetilde\Theta(t)\in L^{q^*}\big(a,b;X^*\big)$,
such that up to a further subsequence,
$\Theta_t\big(u_n(t)\big)\rightharpoonup \widetilde\Theta(t)$ weakly
in $L^{q^*}\big(a,b;X^*\big)$. Using this fact and
\er{fuduftfdftffyvjhvhvghghhjghjgjkkjhkjjhhffjhfjfjfjkghggghhhfjgggjjfffhdhdddghjffjfjfjfjlkhkl},
we deduce that for every $h(t)\in L^{q}(a,b;X)$, that we now fix, we
have:
\begin{equation}\label{fuduftfdftffyvjhvhvghghhjghjgjkkjhkjjhhffjhfjfjfjkghggghhhfjgggjjfffhdhdddghjffjfjfjfjlkhklghjjhhk}
\int_a^b\Big<h(t),\widetilde\Theta(t)\Big>_{X\times
X^*}dt=\lim_{n\to+\infty}\int_a^b\Big<u_n(t)-u(t)+h(t),\Theta_t\big(u_n(t)\big)\Big>_{X\times
X^*}dt.
\end{equation}
Thus, using \er{fuduftfdftffyvjhvhvghghhjghjgjkkjhkj} and Fatou's
Lemma, by
\er{fuduftfdftffyvjhvhvghghhjghjgjkkjhkjjhhffjhfjfjfjkghggghhhfjgggjjfffhdhdddghjffjfjfjfjlkhklghjjhhk}
and
\er{fuduftfdftffyvjhvhvghghhjghjgjkkjhkjjhhffjhfjfjfjkghggghhhfjgggjjfffgghghgkhggkj}
we infer
\begin{multline}\label{fuduftfdftffyvjhvhvghghhjghjgjkkjhkjiuitjjfkggkjggigjkljgkjgifyuf}
\int_a^b\Big<h(t),\widetilde\Theta(t)\Big>_{X\times X^*}dt\geq
\int_a^b\bigg(\liminf_{n\to+\infty}\Big<u_n(t)-u(t)+h(t),\Theta_t\big(u_n(t)\big)\Big>_{X\times
X^*}\bigg)dt\\=\int_a^b\bigg(\liminf_{n\to+\infty}\Big<h(t),\Theta_t\big(u_n(t)\big)\Big>_{X\times
X^*}\bigg)dt.
\end{multline}
On the other hand, by
\er{fuduftfdftffyvjhvhvghghhjghjgjkkjhkjjhhffjhfjfjfjkghggghhhfjgggjjfffgghghgkhggkj}
for a.e. $t\in(a,b)$ there exists a strictly increasing subsequence
$\big\{\bar n^{(t)}_k\}_{k=1}^{+\infty}\subset\mathbb{N}$ such that
we have
\begin{multline}\label{fuduftfdftffyvjhvhvghghhjghjgjkkjhkjjhhffjhfjfjfjgjgjgjyui}
\lim_{k\to+\infty}\Big<u_{\bar
n^{(t)}_k}(t)-u(t)+h(t),\Theta_t\big(u_{\bar
n^{(t)}_k}(t)\big)\Big>_{X\times X^*}=\\
\lim_{k\to+\infty}\Big<h(t),\Theta_t\big(u_{\bar
n^{(t)}_k}(t)\big)\Big>_{X\times
X^*}=\liminf_{n\to+\infty}\Big<h(t),\Theta_t\big(u_n(t)\big)\Big>_{X\times
X^*}<+\infty.
\end{multline}
Therefore, by \er{fuduftfdftffyvjhvhvghghhjghjgjkkjhkj}, for a.e.
fixed $t\in(a,b)$ the sequence $\{u_{\bar
n^{(t)}_k}(t)\}_{k=1}^{+\infty}$ is bounded in $X$. On the other
hand, $T\cdot u_n(t)\rightharpoonup T\cdot u(t)$ weakly in $H$ for
a.e. $t\in(a,b)$. Thus, since, $T$ is injective operator, we obtain
that, for a.e. fixed $t\in(a,b)$ $u_{\bar
n^{(t)}_k}(t)\rightharpoonup u(t)$ weakly in $X$. Therefore, since
for a.e. fixed $t\in(a,b)$ the function $\Theta_t(x):X\to X^*$ is
pseudo-monotone, using
\er{fuduftfdftffyvjhvhvghghhjghjgjkkjhkjjhhffjhfjfjfjkghggghhhfjgggjjfffgghghgkhggkj}
and Lemma \ref{hhhhhhhhhhhhhhhhhhiogfydtdtyd}, for a.e. $t\in(a,b)$
we deduce:
\begin{equation}\label{fuduftfdftffyvjhvhvghghhjghjgjkkjhkjjhhffjhfjfjfjkghggghhhfdfhdhgjuy}
\Theta_t\big(u_{\bar
n^{(t)}_k}(t)\big)\rightharpoonup\Theta_t\big(u(t)\big)\quad\text{weakly
in}\;\, X^*.
\end{equation}
Plugging it into
\er{fuduftfdftffyvjhvhvghghhjghjgjkkjhkjjhhffjhfjfjfjgjgjgjyui}, we
deduce
\begin{equation}\label{fuduftfdftffyvjhvhvghghhjghjgjkkjhkjjhhffjhfjfjfjgjgjgjyuiyyhuikkk}
\liminf_{n\to+\infty}\Big<h(t),\Theta_t\big(u_n(t)\big)\Big>_{X\times
X^*}=\Big<h(t),\Theta_t\big(u(t)\big)\Big>_{X\times
X^*}\quad\text{for a.e.}\;\,t\in(a,b).
\end{equation}
Thus, plugging
\er{fuduftfdftffyvjhvhvghghhjghjgjkkjhkjjhhffjhfjfjfjgjgjgjyuiyyhuikkk}
into
\er{fuduftfdftffyvjhvhvghghhjghjgjkkjhkjiuitjjfkggkjggigjkljgkjgifyuf}
gives
\begin{equation}\label{fuduftfdftffyvjhvhvghghhjghjgjkkjhkjiuitjjfkggkjggigjkljgkjgifyufjhg}
\int_a^b\Big<h(t),\widetilde\Theta(t)\Big>_{X\times X^*}dt\geq
\int_a^b\Big<h(t),\Theta_t\big(u(t)\big)\Big>_{X\times X^*}dt.
\end{equation}
Thus, since $h(t)\in L^{q}(a,b;X)$ was chosen arbitrary,
interchanging between the roles of $h(t)$ and $-h(t)$ gives
\begin{equation}\label{fuduftfdftffyvjhvhvghghhjghjgjkkjhkjiuitjjfkggkjggigjkljgkjgifyufkjgkjkj}
\int_a^b\Big<h(t),\Theta_t\big(u(t)\big)\Big>_{X\times
X^*}dt\leq\int_a^b\Big<h(t),\widetilde\Theta(t)\Big>_{X\times
X^*}dt.
\end{equation}
Therefore, plugging
\er{fuduftfdftffyvjhvhvghghhjghjgjkkjhkjiuitjjfkggkjggigjkljgkjgifyufjhg}
and
\er{fuduftfdftffyvjhvhvghghhjghjgjkkjhkjiuitjjfkggkjggigjkljgkjgifyufkjgkjkj}
gives
\begin{equation}\label{fuduftfdftffyvjhvhvghghhjghjgjkkjhkjiuitjjfkggkjggigjkljgkjgifyufkjgkjkjjfffjfjf}
\int_a^b\Big<h(t),\Theta_t\big(u(t)\big)\Big>_{X\times
X^*}dt=\int_a^b\Big<h(t),\widetilde\Theta(t)\Big>_{X\times X^*}dt,
\end{equation}
and, since $h(t)\in L^{q}(a,b;X)$ was arbitrarily chosen, we deduce
$\Theta_t\big(u(t)\big)=\widetilde\Theta(t)$ for a.e. $t\in(a,b)$.
So $\Theta_t\big(u_n(t)\big)\rightharpoonup \Theta\big(u(t)\big)$
weakly in $L^{q^*}\big(a,b;X^*\big)$. This completes the proof.
\end{proof}

\begin{theorem}\label{defHkkkkglkjjjgkjgkjgggk}
Let $\{X,H,X^*\}$ be an evolution triple with the corresponding
inclusion operator $T\in \mathcal{L}(X;H)$ as it was defined in
Definition \ref{7bdf} together with the corresponding operator
$\widetilde{T}\in \mathcal{L}(H;X^*)$ defined as in \er{tildet}.
Assume also that the Banach space $X$ is separable.
Furthermore, let $a,b,q\in\R$ s.t. $a<b$ and $q\geq 2$.  Next, for
every $t\in[a,b]$ let $\Phi_t(x):X\to[0,+\infty)$ be a convex
function which is G\^{a}teaux differentiable at every $x\in X$,
satisfies $\Phi_t(0)=0$ and satisfies the growth condition
\begin{equation}\label{roststglkjjjuhuh}
0\leq \Phi_t(x)\leq C\,\|x\|_X^q+C\quad\forall
x\in X,\;\forall t\in[a,b]\,,
\end{equation}
for some $C>0$. We also assume that $\Phi_t(x)$
is Borel on the pair of variables $(x,t)$.
Furthermore, for every $t\in[a,b]$ let $\Lambda_t(x):X\to X^*$ be a
function which is G\^{a}teaux differentiable at every $x\in X$,
$\Lambda_t(0)\in L^{q^*}(a,b;X^*)$ and the derivative of $\Lambda_t$
satisfies the growth condition
\begin{equation}\label{roststlambdglkjjjuhuh}
\|D\Lambda_t(x)\|_{\mathcal{L}(X;X^*)}\leq g\big(\|T\cdot
x\|_H\big)\,\big(\|x\|_X^{q-2}+1\big)\quad\forall x\in X,\;\forall
t\in[a,b]\,,
\end{equation}
for some nondecreasing function $g(s):[0,+\infty)\to(0,+\infty)$. We
also assume that $\Lambda_t(x)$
is Borel on the pair of variables $(x,t)$ (see Definition
\ref{fdfjlkjjkkkkkllllkkkjjjhhhkkk}). Assume also that $\Lambda_t$
and $\Phi_t$ satisfy the following monotonicity
condition:
\begin{multline}\label{Monotoneglkjjjuhuhgutuityiyikkk}
\Big<x,D\Phi_t(x)+\Lambda_t(x)\Big>_{X\times X^*}\geq\frac{1}{\hat
C}\,\|x\|_X^q
-
\Big(\|x\|^p_X+\mu^{\frac{p}{2}}(t)\Big)\Bigg(\hat C\|L\cdot
x\|^{(2-p)}_V+\mu^{\frac{2-p}{2}}(t)\Big(\|T\cdot
x\|^{(2-p)}_H+1\Big)\Bigg)\\
\quad\forall x\in
X,\;\forall t\in[a,b],
\end{multline}
where $V$ is a given Banach space, $L\in\mathcal{L}(X,V)$ is a given
compact operator, $p\in[0,2)$,
$\mu(t)\in L^1(a,b;\R)$ is a nonnegative
function and $\hat C>0$ is a constant.
%
%
%
%
%
%
%
%
%
%
%
%
%
%
%
%
Finally, assume that for every $t\in[a,b]$ the mapping
$\big(D\Phi_t+\Lambda_t\big)(x):X\to X^*$ is pseudo-monotone (see
Definition \ref{hfguigiugyuyhkjjh}).
%
%
%
%
%
%
%
%
Then for every $w_0\in H$ there exists $u(t)\in L^q(a,b;X)$, such
that $w(t):=T\cdot\big(u(t)\big)\in L^\infty(a,b;H)$,
$v(t):=\widetilde T\cdot\big(w(t)\big)=\widetilde T\circ
T\big(u(t)\big)\in W^{1,q^*}(a,b; X^*)$
and $u(t)$ is a solution to
\begin{equation}\label{uravnllllgsnachlklimjjjuhuh}
\begin{cases}\frac{d v}{dt}(t)+\Lambda_t\big(u(t)\big)+D\Phi_t\big(u(t)\big)=0\quad\text{for a.e.}\; t\in(a,b)\,,\\
w(a)=w_0\,,
\end{cases}
\end{equation}
where we assume that $w(t)$ is $H$-weakly continuous on $[a,b]$, as
it was stated in Corollary \ref{vbnhjjmcor}. Moreover, if
$\Lambda_t$ and $\Phi_t$ satisfy the following monotonicity
condition
\begin{multline}\label{roststghhh77889lkhjhfvffuhuh}
\Big<h,\big\{D\Phi_t(x+h)-D\Phi_t(x)\big\}+D\Lambda_t(x)\cdot
h\Big>_{X\times X^*}\geq\\
\frac{k_0\big|f(h,t)\big|^2}{\hat g(\|T\cdot x\|_H)} -\hat
g\big(\|T\cdot
x\|_H\big)\cdot\Big(\|x\|_X^q+\mu(t)\Big)^{(2-p)/2}\cdot\big|f(h,t)\big|^p\cdot\|T\cdot
h\|^{(2-p)}_H\quad\forall x\in X,\,\forall h\in X\;\forall
t\in[a,b],
\end{multline}
%
for some constant $k_0\geq 0$ such that $k_0\neq 0$ if $p>0$, some
function $f(h,t):X\times[a,b]\to\R$ and some nondecreasing function
$\hat g(s):[0,+\infty)\to(0,+\infty)$, then such a solution to
\er{uravnllllgsnachlklimjjjuhuh} is unique.
\end{theorem}
\begin{proof}
\noindent\underline{Step 1:} Existence of the solution. Since the
Banach space is separable, by Lemma \ref{hilbcomban}
we deduce that there exists a separable Hilbert space $Y$ and a
bounded linear inclusion operator $S\in \mathcal{L}(Y;X)$ such that
$S$ is injective, the image of $S$ is dense in $X$ and moreover, $S$
is a compact operator. Moreover, let $S^*\in \mathcal{L}(X^*;Y^*)$
be the corresponding adjoint operator, which satisfies
\begin{equation}\label{tildetdallhjhjh778889kkk}
\big<y,S^*\cdot x^*\big>_{Y\times Y^*}:=\big<S\cdot
 y,x^*\big>_{X\times X^*}\quad\quad\text{for every}\; x^*\in
X^*\;\text{and}\;y\in Y\,.
\end{equation}
Set $P\in \mathcal{L}(Y;H)$, defined by $P:=T\circ S$ and
$\widetilde{P}\in \mathcal{L}(H;Y^*)$, defined by
$\widetilde{P}:=S^*\circ \widetilde{T}$. Then it is clear that
$\{Y,H,Y^*\}$ is another evolution triple with the corresponding
inclusion operator $P\in \mathcal{L}(Y;H)$ as it was defined in
Definition \ref{7bdf} together with the corresponding adjoint
operator $\widetilde{P}\in \mathcal{L}(H;Y^*)$ defined as in
\er{tildet}.
%
%
%

 Furthermore, let $\psi(t)\in L^q(a,b;Y)$ be such that
the function $\varphi(t):(a,b)\to X^*$ defined by
$\varphi(t):=I_Y\cdot \big(\psi(t)\big)$ belongs to
$W^{1,q^*}(a,b;Y^*)$,
where $I_Y:=\widetilde
P\circ P:\,Y\to Y^*$. Denote the set of all such functions $\psi$ by
$\mathcal{R}_{Y,q}(a,b)$.
As before, by Lemma \ref{lem2}, for every
$\psi(t)\in\mathcal{R}_{q}(a,b)$ the function $w(t):[a,b]\to H$
defined by $w(t):=P\cdot \big(\psi(t)\big)$ belongs to
$L^\infty(a,b;H)$ and, up to a redefinition of $w(t)$ on a subset of
$[a,b]$ of Lebesgue measure zero,  $w$ is $H$-weakly continuous, as
it was stated in Corollary \ref{vbnhjjmcor}.

Next let
$\Psi(y):Y\to[0,+\infty)$ be a function defined by
\begin{equation}\label{hjfxsphiscalddduhuh}
\Psi(y):=
\|y\|_Y^q+\|y\|_Y^2\quad\quad\forall y\in Y.
\end{equation}
Then $\Psi(y)$ is a convex function which is G\^{a}teaux
differentiable on every $y\in Y$, satisfies $\Psi(0)=0$ and
satisfies the growth condition
\begin{equation}\label{roststglkkljjjuhuh}
(1/C_0)\,\|y\|_Y^q-C_0\leq \Psi(y)\leq
C_0\,\|y\|_Y^q+C_0\quad\forall y\in Y\,,
\end{equation}
and the uniform convexity condition:
\begin{equation}\label{roststghhh77889lkjjjuhuh}
\Big<h,D\Psi(y+h)-D\Psi(y)\Big>_{Y\times Y^*}\geq
\frac{1}{C_0}\big(\| y\|^{q-2}_Y+1\big)\cdot\|h\|_Y^2\quad\forall
y,h\in Y
\end{equation}
for some $C_0>0$.

Next let $w_0\in H$. Then, since the image of the operator $T\circ
S$ is dense in $H$, there exists a sequence $\{\psi^{(0)}_n\}\subset
Y$ such that $w^{(0)}_n:=(T\circ S)\cdot\psi^{(0)}_n\to w_0$
strongly in $H$ as $n\to +\infty$. Furthermore, let $\e_n\to 0^+$ as
$n\to +\infty$. By Theorem \ref{THSldInt}, for every $n$ there
exists $\psi_n(t)\in\mathcal{R}_{Y,q}(a,b)$, such that
\begin{multline}\label{uravnllllgsnachlkagagnew}
\frac{d\varphi_n}{dt}(t)+S^*\cdot
\Big(\Lambda_t\big(u_n(t)\big)+D\Phi_t\big(u_n(t)\big)\Big)+\e_n
D\Psi\big(\psi_n(t)\big)=0\;\; \text{for a.e.}\;
t\in(a,b)\quad\text{and}\;\; w_n(a)=w^{(0)}_n,
\end{multline}
where $u_n(t):=S\cdot\big(\psi_n(t)\big)$,
$w_n(t):=P\cdot\big(\psi_n(t)\big)$, $\varphi_n(t):=
\widetilde P\cdot \big(w_n(t)\big)$
and we assume that $w_n(t)$ is $H$-weakly continuous on $[a,b]$, as
it was stated in Corollary \ref{vbnhjjmcor}.

%
%
%
%
%

 On the other hand, by the trivial inequality $(p/2)\,a^2+\big((2-p)/2\big)\,
b^2\geq a^p \,b^{2-p}$ using \er{Monotoneglkjjjuhuhgutuityiyikkk} be
deduce:
\begin{multline}\label{Monotoneglkjjjuhuhgutuityiyikkkllk}
\Big<x,D\Phi_t(x)+\Lambda_t(x)\Big>_{X\times X^*}\geq\frac{1}{
C_1}\,\|x\|_X^q -C_1\|L\cdot x\|^{2}_V-C_1\mu(t)\Big(\|T\cdot
x\|^{2}_H+1\Big)\quad\forall x\in X,\;\forall t\in[a,b],
\end{multline}
for some constant $C_1>0$. Then, as before in
\er{roststlambdglkFFyhuhtyui99999999new}, we obtain:
\begin{multline}\label{Monotoneglkjjjuhuhgutuityiyikkkllkjjlioiujh}
\Big<x,D\Phi_t(x)+\Lambda_t(x)\Big>_{X\times X^*}\geq\frac{1}{
K}\,\|x\|_X^q -\tilde\mu(t)\Big(\|T\cdot
x\|^{2}_H+1\Big)\quad\forall x\in X,\;\forall t\in[a,b],
\end{multline}
for some constant $K>0$ and $\tilde\mu(t)\in L^1(a,b;\R)$. Thus,
since for every $t\in[a,b]$ the mapping
$\big(D\Phi_t+\Lambda_t\big)(x):X\to X^*$ is pseudo-monotone,
applying Lemma \ref{hkjghiohioujpohgkgk} implies that the mapping
$\Gamma\big(x(t)\big):\big\{\bar x(t)\in L^q(a,b;X):\;T\cdot \bar
x(t)\in L^\infty(a,b;H)\big\}\to L^{q^*}(a,b;X^*)$,
defined by
\begin{multline}\label{fuduftfdftffyvjhvhvghghh}
\bigg<h(t),\Gamma\big(x(t)\big)\bigg>_{L^q(a,b;X)\times
L^{q^*}(a,b;X^*)}:=\int_a^b\Big<h(t),\Lambda_t\big(x(t)\big)+D\Phi_t\big(x(t)\big)\Big>_{X\times
X^*}dt\\ \forall\, x(t)\in \big\{\bar x(t)\in L^q(a,b;X):\;T\cdot
\bar x(t)\in L^\infty(a,b;H)\big\},\;\forall\, h(t)\in
L^{q}(a,b;X)\,,
\end{multline}
is weakly pseudo-monotone with respect to the evolution triple
$\{X,H,X^*\}$ (see Definition \ref{hfguigiugyuyhkjjhjjjk}).

So all the conditions of Corollary \ref{thhypppggghhhjgggnew}
satisfied and therefore, by this Corollary, up to a subsequence, we
have $u_n(t)\rightharpoonup u(t)$ weakly in $L^q(a,b;X)$ where
$u(t)\in L^q(a,b;X)$ is such that $w(t):=T\cdot\big(u(t)\big)\in
L^\infty(a,b;H)$, $v(t):=\widetilde T\cdot\big(w(t)\big)=\widetilde
T\circ T\big(u(t)\big)\in W^{1,q^*}(a,b; X^*)$
and $u(t)$ is a solution to \er{uravnllllgsnachlklimjjjuhuh},
where we assume that $w(t)$ is $H$-weakly continuous on $[a,b]$, as
it was stated in Corollary \ref{vbnhjjmcor}.

 \noindent\underline{Step 2:} Uniqueness of the solution. Assume
that $\Phi_t$ satisfies \er{roststghhh77889lkhjhfvffuhuh}. Then
applying Theorem \ref{EulerLagrangeInt} completes the proof.
\end{proof}
\begin{remark}\label{hjhuiguigu}
By Lemma \ref{lem2} the solution to \er{uravnllllgsnachlklimjjjuhuh}
from Theorem \ref{defHkkkkglkjjjgkjgkjgggk}
satisfies the following energy equality:
\begin{equation}\label{uravnllllgsnachlklimjjjuhuhhgghghfjfggjkghhgkg}
\frac{1}{2}\|w(t)\|^2_H+\int_a^t
\Big<u(s),\Lambda_s\big(u(s)\big)+D\Phi_s\big(u(s)\big)\Big>_{X\times
X^*}ds=\frac{1}{2}\|w_0\|^2_H\quad\quad\forall t\in[a,b]\,.
\end{equation}
\end{remark}
As a particular case, of Theorem \ref{defHkkkkglkjjjgkjgkjgggk} we
have the following Theorem:
\begin{theorem}\label{defHkkkkglkjjj}
Let $\{X,H,X^*\}$ be an evolution triple with the corresponding
inclusion operator $T\in \mathcal{L}(X;H)$ as it was defined in
Definition \ref{7bdf} together with the corresponding operator
$\widetilde{T}\in \mathcal{L}(H;X^*)$ defined as in \er{tildet}.
Assume also that the Banach space $X$ is separable.
Furthermore, let $a,b,q\in\R$ s.t. $a<b$ and $q\geq 2$.  Next, for
every $t\in[a,b]$ let $\Phi_t(x):X\to[0,+\infty)$ be a convex
function which is G\^{a}teaux differentiable at every $x\in X$,
satisfies $\Phi_t(0)=0$ and satisfies the growth condition
\begin{equation}\label{roststglkjjj}
0\leq \Phi_t(x)\leq C\,\|x\|_X^q+C\quad\forall
x\in X,\;\forall t\in[a,b]\,,
\end{equation}
for some $C>0$. We also assume that $\Phi_t(x)$
is Borel on the pair of variables $(x,t)$.
Furthermore, for every $t\in[a,b]$ let $\Lambda_t(x):X\to X^*$ be a
function which is G\^{a}teaux differentiable at every $x\in X$,
$\Lambda_t(0)\in L^{q^*}(a,b;X^*)$ and the derivative of $\Lambda_t$
satisfies the growth condition
\begin{equation}\label{roststlambdglkjjj}
\|D\Lambda_t(x)\|_{\mathcal{L}(X;X^*)}\leq g\big(\|T\cdot
x\|_H\big)\,\big(\|x\|_X^{q-2}+ 1\big)\quad\forall x\in X,\;\forall
t\in[a,b]\,,
\end{equation}
for some nondecreasing function $g(s):[0+\infty)\to(0+\infty)$. We
also assume that $\Lambda_t(x)$
is
Borel on the pair of variables $(x,t)$ (see Definition
\ref{fdfjlkjjkkkkkllllkkkjjjhhhkkk}). Assume also that $\Lambda_t$
satisfies the following monotonicity conditions
\begin{equation}\label{Monotoneglkjjj}
\Big<h,D\Lambda_t(x)\cdot h\Big>_{X\times X^*}\geq 0
\quad\quad\forall x,h\in X,\;\forall t\in[a,b]\,.
\end{equation}
%
%
%
%
%
%
%
%
Finally let $F_t(x):X\to X^*$ be a function which is G\^{a}teaux
differentiable on every $x\in X$, $F_t(0)\in L^{q^*}(a,b;X^*)$ and
the derivative of $F_t$ satisfies the condition
\begin{equation}\label{roststlambdglkFFjjj}
\|DF_t(x)\|_{\mathcal{L}(X;X^*)}\leq g\big(\|T\cdot
x\|_H\big)\big(\|x\|^{q-2}_X+1\big)\quad\forall x\in X,\;\forall
t\in[a,b]\,,
\end{equation}
for some nondecreasing function $g(s):[0+\infty)\to(0+\infty)$. We
also assume that $F_t(x)$ is
Borel on the pair of variables $(x,t)$.
Next assume that
\begin{multline}\label{Monotoneglkjjjuhuhgutuityiyikkkjggghghjgh}
\Big<x,D\Phi_t(x)+\Lambda_t(x)
+F_t(x)\Big>_{X\times X^*}\geq\\ \frac{1}{\hat C}\,\|x\|_X^q
-\hat C\big(\|x\|_X+1\big)\Big(\|L\cdot x\|_V+\|T\cdot
x\|_H+1\Big)-\mu(t)\quad\quad\forall x\in X,\;\forall t\in[a,b],
\end{multline}
where $V$ is a given Banach space, $L\in\mathcal{L}(X,V)$ is a given
compact operator, $\hat C>0$ is some constant and $\mu(t)\in
L^1(a,b;\R)$ is some nonnegative function. Finally, assume that
$F_t(x)$ is weak-to-strong continuous, i.e. for every fixed
$t\in[a,b]$ and every $x_n\rightharpoonup x$ weakly in $X$, we have
$F_t(x_n)\to F_t(x)$ strongly in $X^*$.
%
%
%
%
%
%
%
%
Then for every $w_0\in H$ there exists $u(t)\in L^q(a,b;X)$, such
that $w(t):=T\cdot\big(u(t)\big)\in L^\infty(a,b;H)$,
$v(t):=\widetilde T\cdot\big(w(t)\big)=\widetilde T\circ
T\big(u(t)\big)\in W^{1,q^*}(a,b; X^*)$
and
$u(t)$ is a solution to
\begin{equation}\label{uravnllllgsnachlklimjjj}
\begin{cases}\frac{d v}{dt}(t)+F_t\big(u(t)\big)+\Lambda_t\big(u(t)\big)
+D\Phi_t\big(u(t)\big)=0\quad\text{for a.e.}\; t\in(a,b)\,,\\
w(a)=w_0\,,
\end{cases}
\end{equation}
where we assume that $w(t)$ is $H$-weakly continuous on $[a,b]$, as
it was stated in Corollary \ref{vbnhjjmcor}.
\end{theorem}
\begin{proof}
Since $F_t(x):X\to X^*$ is weak to strong continuous, it is
pseudo-monotone on $X$. Moreover, for every $t\in[a,b]$
$D\Phi_t(x):X\to X^*$ and $\Lambda_t(x):X\to X^*$ are monotone
mappings. Therefore, since $\Lambda_t$ is G\^{a}teaux differentiable
and $\Phi_t$ is convex, using Lemma
\ref{hhhhhhhhhhhhhhhhhhiogfydtdtyd} and Remark
\ref{fyyjfjyhfgjgghgjkgkjgkgggfhfhkkk} we deduce that
$(D\Phi_t+\Lambda_t+F_t)(x):X\to X^*$ is pseudo monotone.
Thus, applying
Theorem
\ref{defHkkkkglkjjjgkjgkjgggk} with $\Lambda_t+F_t$ instead of
$\Lambda_t$, gives the desired result.
\end{proof}

\begin{theorem}\label{thhypppggghhhjggg}
Let $X$ and $Z$ be reflexive Banach spaces and $X^*$ and $Z^*$ be
the corresponding dual spaces. Furthermore let $H$ be a Hilbert
space. Suppose that $Q\in \mathcal{L}(X,Z)$ is an injective
inclusion operator such that its image is
dense on $Z$. Furthermore, suppose that $P\in \mathcal{L}(Z,H)$ is
an injective inclusion operator such that its image is dense on $H$.
Let $T\in \mathcal{L}(X,H)$ be defined by $T:=P\circ Q$.
So $\{X,H,X^*\}$ is an evolution triple with the corresponding
inclusion operator $T\in \mathcal{L}(X;H)$ as it was defined in
Definition \ref{7bdf} together with the corresponding operator
$\widetilde{T}\in \mathcal{L}(H;X^*)$ defined as in \er{tildet}.
Assume also that the Banach space $X$ is separable. Next let
$a,b\in\R$ s.t. $a<b$ and $q\geq 2$. Furthermore, for every
$t\in[a,b]$ let $\Lambda_t(z):Z\to X^*$ and $A_t(z):Z\to X^*$  be
functions which are G\^{a}teaux differentiable at every $z\in Z$ and
$\Lambda_t(0),A_t(0)\in L^{q^*}\big(a,b;X^*\big)$. Assume that for
every $t\in[a,b]$, they satisfy the following bounds
\begin{equation}\label{bghffgghgkoooojkvhgjgjgfffh}
\big\|D\Lambda_t(z)\big\|_{\mathcal{L}(Z;X^*)}\leq g\big(\|P\cdot
z\|_H\big)\cdot\Big(\|z\|_{Z}^{q-2}+1\Big)\quad\quad\forall z\in
Z,\;\forall t\in[a,b]\,,
\end{equation}
\begin{equation}\label{bghffgghgkoooojkvhgj}
\big\|\Lambda_t(z)\big\|_{X^*}\leq g\big(\|P\cdot
z\|_H\big)\cdot\Big(\|L_0\cdot
z\|_{V_0}^{q-1}+\tilde\mu^{\frac{q-1}{q}}(t)\Big)\quad\quad\forall
z\in Z,\;\forall t\in[a,b]\,,
\end{equation}
and
\begin{equation}\label{bghffgghgkoooojkvhgjgjgfffhfgbfbh}
\big\|DA_t(z)\big\|_{\mathcal{L}(Z;X^*)}\leq g\big(\|P\cdot
z\|_H\big)\cdot\Big(\|L_0\cdot z\|_{
V_0}^{q-2}+1\Big)\quad\quad\forall z\in Z,\;\forall t\in[a,b]\,,
\end{equation}
where $\tilde\mu(t)\in L^1(a,b;\R)$ is some nonnegative function,
$g(s):[0,+\infty)\to(0,+\infty)$ is some nondecreasing function,
$V_0$ is some
Banach space and
$L_0\in\mathcal{L}(Z;V_0)$ is some compact linear operator.
Moreover, assume that $\Lambda_t$ and $A_t$ satisfy the following
monotonicity
condition:
\begin{multline}
\label{roststlambdglkFFyhuhtyui99999999} \Big<h,A_t(Q\cdot
h\big)+\Lambda_t \big(Q\cdot h\big)\Big>_{X\times X^*}\geq
\big(1/\bar C\big)\big\|Q\cdot h\big\|^q_Z\\-
\Big(\big\|Q\cdot h\big\|^p_Z+\mu^{\frac{p}{2}}(t)\Big)\Bigg(\bar
C\big\|L\cdot(Q\cdot
h)\big\|^{(2-p)}_V+\mu^{\frac{2-p}{2}}(t)\Big(\big\|T\cdot
h\big\|^{(2-p)}_H+1\Big)\Bigg)
\quad \forall h\in X\,,\;\forall t\in[a,b],
\end{multline}
where $V$ is a given Banach space, $L\in\mathcal{L}(Z,V)$ is a given
compact operator, $p\in[0,2)$,
$\mu(t)\in L^1(a,b;\R)$ is some nonnegative
function and $\bar C>0$ is some constant. We also assume that
$\Lambda_t(z)$ $A_t(z)$ are
Borel on the pair of variables $(z,t)$.
Finally assume that there exists a family of
Banach spaces
$\{V_j\}_{j=1}^{+\infty}$ and a family  of compact bounded linear
operators $\{L_j\}_{j=1}^{+\infty}$, where
$L_j\in\mathcal{L}(Z,V_j)$, which satisfy the following condition:
\begin{itemize}
\item
If $\{h_n\}_{n=1}^{+\infty}\subset Z$ is a sequence and $h_0\in Z$,
are such that for every fixed $j$ $\lim_{n\to+\infty}L_j\cdot
h_n=L_j\cdot h_0$ strongly in $V_j$ and $P\cdot h_n\rightharpoonup
P\cdot h_0$ weakly in $H$, then for every fixed $t\in(a,b)$ we have
$\Lambda_t(h_n)\rightharpoonup \Lambda_t(h_0)$ weakly in $X^*$ and
$D A_t(h_n)\to D A_t(h_0)$ strongly in $\mathcal{L}(Z,X^*)$.
\end{itemize}
Then for every $w_0\in H$ there exists $z(t)\in L^q(a,b;Z)$ such
that $w(t):=P\cdot z(t)\in L^\infty(a,b;H)$, $v(t):=\widetilde
T\cdot \big(w(t)\big)\in W^{1,q^*}(a,b;X^*)$ and $z(t)$ satisfies
the following equation
\begin{equation}\label{uravnllllgsnachlklimhjhgh}
\begin{cases}\frac{d v}{dt}(t)+A_t\big(z(t)\big) +\Lambda_t\big(z(t)\big)
=0\quad\text{for a.e.}\; t\in(a,b)\,,\\
w(a)=w_0\,,
\end{cases}
\end{equation}
where we assume that $w(t)$ is $H$-weakly continuous on $[a,b]$, as
it was stated in Corollary \ref{vbnhjjmcor}. Moreover, if in
addition we assume that there exist a Banach space $V$, a compact
operator $L\in\mathcal{L}(Z,V)$, a nondecreasing function
$\widetilde g(s):[0,+\infty)\to(0,+\infty)$ and for every
$t\in[a,b]$ a convex G\^{a}teaux differentiable functions
$\Phi_t:Z\to\R$, Borel measurable on $(z,t)$, and a G\^{a}teaux
differentiable mapping $F_t(\sigma):V\to Z^*$, Borel measurable on
$(\sigma,t)$, satisfying $F_t(0)\in L^{q^*}(a,b;Z^*)$ and such that
\begin{equation}\label{roststlambdglkFFhjhjhhjjkjkjkjkjkllhllhlhtgythtyjuui}
0\leq \Phi_t(z)\leq \widetilde g\big(\|P\cdot
z\|_H\big)\cdot\big(\|z\|^q_Z+1\big)\quad\forall z\in Z,\;\forall
t\in[a,b]\,,
\end{equation}
\begin{equation}\label{roststlambdglkFFhjhjhhjjkjkjkjkjkllhllhlhtgytht}
\big\|DF_t(L\cdot z)\big\|_{\mathcal{L}(V;Z^*)}\leq \widetilde
g\big(\|P\cdot z\|_H\big)\cdot\big(\|L\cdot
z\|^{q-2}_V+1\big)\quad\forall z\in Z,\;\forall t\in[a,b]\,,
\end{equation}
and
\begin{multline}
\label{roststlambdglkFFyhuhtyui99999999ghghghh} \Big<h,A_t(Q\cdot
h\big)+\Lambda_t \big(Q\cdot h\big)\Big>_{X\times X^*}\geq
\Phi_t\big(Q\cdot h\big)+\Big<Q\cdot h,F_t\big((L\circ Q)\cdot
h\big)\Big>_{Z\times Z^*}\quad\quad \forall h\in X\,,\;\forall
t\in[a,b]\,,
\end{multline}
then the function $z(t)$, as above, satisfies the following energy
inequality
\begin{equation}
\label{Monotonegkkkkdhtdtjgkgjfgdfxfdfdnf}
\frac{1}{2}\big\|w(t)\big\|^2_H+\int_a^t\bigg(\Phi_s\big(z(s)\big)+\Big<z(s),F_s\big(L\cdot
z(s)\big)\Big>_{Z\times Z^*}\bigg)ds\leq
\frac{1}{2}\big\|w_0\big\|^2_H\quad\quad \forall t\in[a,b]\,.
\end{equation}
\end{theorem}
\begin{proof}
Since the Banach space is separable, as before, by Lemma
\ref{hilbcomban} we deduce that there exists a separable Hilbert
space $Y$ and a bounded linear inclusion operator $S\in
\mathcal{L}(Y;X)$ such that $S$ is injective, the image of $S$ is
dense in $X$ and moreover, $S$ is a compact operator. Moreover, let
$S^*\in \mathcal{L}(X^*;Y^*)$ be the corresponding adjoint operator,
which satisfy
\begin{equation}\label{tildetdallhjhjh778889kkknnn}
\big<y,S^*\cdot x^*\big>_{Y\times Y^*}:=\big<S\cdot
 y,x^*\big>_{X\times X^*}\quad\quad\text{for every}\; x^*\in
X^*\;\text{and}\;y\in Y\,.
\end{equation}
Set $P_0\in \mathcal{L}(Y;H)$, defined by $P_0:=T\circ S$ and
$\widetilde{P}_0\in \mathcal{L}(H;Y^*)$, defined by
$\widetilde{P}_0:=S^*\circ \widetilde{T}$. Then it is clear that
$\{Y,H,Y^*\}$ is another evolution triple with the corresponding
inclusion operator $P_0\in \mathcal{L}(Y;H)$ as it was defined in
Definition \ref{7bdf} together with the corresponding adjoint
operator $\widetilde{P}_0\in \mathcal{L}(H;Y^*)$ defined as in
\er{tildet}.
%
%
%

 Furthermore, let $\psi(t)\in L^q(a,b;Y)$ be such that the function
$\varphi(t):(a,b)\to Y^*$ defined by $\varphi(t):=I_Y\cdot
\big(\psi(t)\big)$ belongs to $W^{1,q^*}(a,b;Y^*)$,
where $I_Y:=\widetilde P_0\circ P_0:\,Y\to Y^*$.
Denote the set of all such functions $\psi$ by
$\mathcal{R}_{Y,q}(a,b)$.
As before, by Lemma \ref{lem2}, for every
$\psi(t)\in\mathcal{R}_{q}(a,b)$ the function $w(t):[a,b]\to H$
defined by $w(t):=P_0\cdot \big(\psi(t)\big)$ belongs to
$L^\infty(a,b;H)$ and, up to a redefinition of $w(t)$ on a subset of
$[a,b]$ of Lebesgue measure zero,  $w$ is $H$-weakly continuous, as
it was stated in Corollary \ref{vbnhjjmcor}.

Next let $\Psi(y):Y\to[0,+\infty)$ be a function defined by
\begin{equation}\label{hjfxsphiscalddduhuhnnn}
\Psi(y):=
\|y\|_Y^q+\|y\|_Y^2\quad\quad\forall y\in Y.
\end{equation}
Then $\Psi(y)$ is a convex function which is G\^{a}teaux
differentiable on every $y\in Y$, satisfies $\Psi(0)=0$ and
satisfies the growth condition
\begin{equation}\label{hjfxsphiscaldddjjhhhh888899999}
(1/C_0)\,\|y\|_Y^q-C_0\leq \Psi(y)\leq
C_0\,\|y\|_Y^q+C_0\quad\forall y\in Y\,,
\end{equation}
and the uniform convexity condition:
\begin{equation*}
\Big<h,D\Psi(y+h)-D\Psi(y)\Big>_{Y\times Y^*}\geq
\frac{1}{C_0}\big(\| y\|^{q-2}_Y+1\big)\cdot\|h\|_Y^2\quad\forall
y,h\in Y
\end{equation*}
for some $C_0>0$.

Next let $w_0\in H$. Then, since the image of the operator $T\circ
S$ is dense in $H$, there exists a sequence $\{\psi^{(0)}_n\}\subset
Y$ such that $w^{(0)}_n:=(T\circ S)\cdot\psi^{(0)}_n\to w_0$
strongly in $H$ as $n\to +\infty$. Furthermore, let $\e_n\to 0^+$ as
$n\to +\infty$. By Theorem \ref{THSldInt}, for every $n$ there
exists $\psi_n(t)\in\mathcal{R}_{Y,q}(a,b)$, such that
\begin{equation}\label{uravnllllglkhghkgkghfsld}
\begin{cases}\frac{d\varphi_n}{dt}(t)+S^*\cdot
\Big(A_t\big(z_n(t)\big)+\Lambda_t\big(z_n(t)\big)\Big)
+\e_n D\Psi\big(\psi_n(t)\big)=0\;\;\text{for}\; t\in(a,b)\,,\\
w_n(a)=(T\circ S)\cdot\psi^{(0)}_n\,,\end{cases}
\end{equation}
where $u_n(t):=S\cdot \big(\psi_n(t)\big)$, $z_n(t):=(Q\circ S)\cdot
\big(\psi_n(t)\big)=Q\cdot \big(u_n(t)\big)$, $w_n(t):=(T\circ
S)\cdot\big(\psi_n(t)\big)=P\cdot\big(z_n(t)\big)$,
$\varphi_n(t):=(S^*\circ\widetilde T\circ T\circ S)\cdot
\big(\psi_n(t)\big)=(S^*\circ\widetilde T)\cdot \big(w_n(t)\big)$
and we assume that $w_n(t)$ is $H$-weakly continuous on $[a,b]$.
%
%
%
%
%
Thus all the conditions of Lemma \ref{thhypppggghhhjgggnew}
satisfied, and by Lemma \ref{thhypppggghhhjgggnew} and Lemma
\ref{vlozhenie}, there exists $z(t)\in L^q(a,b;Z)$ and
$\bar\Lambda(t),\bar A(t)\in L^{q^*}(a,b;X^*)$ such that
$w(t):=P\cdot z(t)\in L^\infty(a,b;H)$, $v(t):=\widetilde T\cdot
w(t)\in W^{1,q^*}(a,b;X^*)$, $w(t)$ is $H$-weakly continuous on
$[a,b]$, up to a subsequence, we have
\begin{equation}\label{uravnllllgsnachlkagagsldnewgiukhoikljjlkjgjgg}
\begin{cases}
z_n(t)\rightharpoonup z(t)\quad\text{weakly in}\;\;L^q(a,b;Z)\\
\frac{d\varphi_n}{dt}(t)\rightharpoonup \frac{d\varphi}{dt}(t)\quad\text{weakly in}\;\;L^{q^*}(a,b;Y^*)\\
\Lambda_t \big(z_n(t)\big)\rightharpoonup \bar\Lambda(t)\quad\text{weakly in}\;\;L^{q^*}(a,b;X^*)\\
A_t \big(z_n(t)\big)\rightharpoonup \bar A(t)\quad\text{weakly in}\;\;L^{q^*}(a,b;X^*)\\
w_n(t)\rightharpoonup w(t)\quad\text{weakly in}\;\;H\quad\text{for
every fixed}\;\;t\in[a,b],\\
\big\{w_n(t)\big\}_{n=1}^{+\infty}\;\;\text{is bounded
in}\;\;L^\infty(a,b;H),
\end{cases}
\end{equation}
where $\varphi(t)=S^*\cdot v(t)$, and $z(t)$ satisfies the following
equation
\begin{equation}\label{uravnllllgsnachlklimhjhghghty}
\begin{cases}\frac{d v}{dt}(t)+\bar A(t)+\bar\Lambda(t)
=0\quad\text{for a.e.}\; t\in(a,b)\,,\\
w(a)=w_0\,.
\end{cases}
\end{equation}
Moreover,
\begin{equation}\label{vyifyurturfurfuyrfgyukkknewlhhhlljkh}
\frac{1}{2}\big\|w(t)\big\|^2_H+\limsup_{n\to+\infty}\bigg(\int_a^t\Big<u_n(s),A_s
\big(z_n(s)\big)+\Lambda_s \big(z_n(s)\big)\Big>_{X\times
X^*}ds\bigg)\leq\frac{1}{2}\big\|w_0\big\|^2_H\quad\forall
t\in[a,b].
\end{equation}
Next
there exists a family of reflexive Banach spaces
$\{V_j\}_{j=1}^{+\infty}$ and a family  of compact bounded linear
operators $\{L_j\}_{j=1}^{+\infty}$, where
$L_j\in\mathcal{L}(Z,V_j)$, which satisfy the following condition:
\begin{itemize}
\item
If $\{h_n\}_{n=1}^{+\infty}\subset Z$ is a sequence and $h_0\in Z$,
are such that for every fixed $j$ $\lim_{n\to+\infty}L_j\cdot
h_n=L_j\cdot h_0$ strongly in $V_j$ and $P\cdot h_n\rightharpoonup
P\cdot h_0$ weakly in $H$, then for every fixed $t\in(a,b)$
$\Lambda_t(h_n)\rightharpoonup \Lambda_t(h_0)$ weakly in $X^*$ and
$DA_t(h_n)\to DA_t(h_0)$ strongly in $\mathcal{L}(Z,X^*)$.
\end{itemize}
On the other hand, using
\er{uravnllllgsnachlkagagsldnewgiukhoikljjlkjgjgg}, and Lemma
\ref{ComTem1}, we deduce that for every $j$ we have $L_j\cdot
z_n(t)\to L_j \cdot z(t)$ strongly in $L^{q}(a,b;V_j)$ as $n\to
+\infty$. By the same way we obtain $L_0\cdot z_n(t)\to L_0 \cdot
z(t)$ strongly in $L^{q}(a,b;V_0)$ as $n\to +\infty$. Thus, up to a
further subsequence we will have $L_j\cdot z_n(t)\to L_j \cdot z(t)$
strongly in $V_j$ for almost every $t\in(a,b)$ and every $j$.
Therefore, by \er{uravnllllgsnachlkagagsldnewgiukhoikljjlkjgjgg} and
the above condition, we must have
$\Lambda_t\big(z_n(t)\big)\rightharpoonup \Lambda_t\big(z(t)\big)$
weakly in $X^*$ and $DA_t\big(sz_n(t)+(1-s)z(t)\big)\to
DA_t\big(z(t)\big)$ strongly in $\mathcal{L}(Z,X^*)$ for almost
every $t\in(a,b)$ and for every $s\in[0,1]$. Therefore, using
\er{bghffgghgkoooojkvhgj}, the fact that $\{w_n(t)\}$ is bounded in
$L^\infty(a,b;H)$ and the fact that $L_0\cdot z_n(t)\to L_0 \cdot
z(t)$ strongly in $L^{q}(a,b;V_0)$, we deduce that
$$\int_a^b\Big<h(t),\Lambda_t\big(z_n(t)\big)\Big>_{X\times X^*}dt\to \int_a^b\Big<h(t),\Lambda_t\big(z(t)\big)\Big>_{X\times X^*}dt\quad\quad\forall h(t)\in L^{q}(a,b,X)\,.$$
Thus
\begin{equation}\label{boundnes77889999hhhhhh999ghg78999hjhjhjh99kkkk}
\Lambda_t\big(z_n(t)\big)\rightharpoonup
\Lambda_t\big(z(t)\big)\quad\text{weakly in}\;\;L^{q^*}(a,b;X^*)\,.
\end{equation}
In the similar way, by \er{bghffgghgkoooojkvhgjgjgfffhfgbfbh}, the
fact that $\{w_n(t)\}$ is bounded in $L^\infty(a,b;H)$ and the fact
that $L_0\cdot z_n(t)\to L_0 \cdot z(t)$ strongly in
$L^{q}(a,b;V_0)$, we deduce that, for $q=2$ we have
\begin{multline}\label{boundnes77889999hhhhhh999ghg78999hjhjhjh99kkkkgjjfbvbvbghhgfjffddhgdh}
DA_t\big(sz_n(t)+(1-s)z(t)\big)\to
DA_t\big(z(t)\big)\quad\text{strongly
in}\;\mathcal{L}(Z,X^*)\;\text{for a.e.}\;t\in(a,b)\;\forall s\in[0,1]\,\\
\text{and}\quad DA_t\big(sz_n(t)+(1-s)z(t)\big)\;\text{is bounded
in}\; L^{\infty}\big(a,b;\mathcal{L}(Z,X^*)\big)\;\text{uniformly
by}\;s\,.
\end{multline}
and for $q>2$ we have
\begin{equation}\label{boundnes77889999hhhhhh999ghg78999hjhjhjh99kkkkgjjfbvbvb}
DA_t\big(sz_n(t)+(1-s)z(t)\big)\to
DA_t\big(z(t)\big)\;\;\text{strongly
in}\;\;L^{q/(q-2)}\big(a,b;\mathcal{L}(Z,X^*)\big)\;\;\forall
s\in[0,1]\,.
\end{equation}
In both cases
\begin{multline}\label{boundnes77889999hhhhhh999ghg78999hjhjhjh99kkkkgjjfbvbvbfgdddhdffgf}
\Big\{DA_t\big(sz_n(t)+(1-s)z(t)\big)\Big\}^*\cdot h(t)\to
\Big\{DA_t\big(z(t)\big)\Big\}^*\cdot h(t)\;\;\text{strongly
in}\;\;L^{q^*}(a,b,Z)\\ \forall h(t)\in L^{q}(a,b,X)\;\;\forall
s\in[0,1]\,,
\end{multline}
where $\big\{D A_t(\cdot)\big\}^*\in \mathcal{L}(X,Z^*)$ is the
adjoint to $D A_t(\cdot)\in \mathcal{L}(Z,X^*)$ operator. Thus, by
\er{bghffgghgkoooojkvhgjgjgfffhfgbfbh}, the fact that $\{w_n(t)\}$
is bounded in $L^\infty(a,b;H)$ and the fact that $L_0\cdot
z_n(t)\to L_0 \cdot z(t)$ strongly in $L^{q}(a,b;V_0)$, together
with
\er{boundnes77889999hhhhhh999ghg78999hjhjhjh99kkkkgjjfbvbvbfgdddhdffgf}
and
\er{uravnllllgsnachlkagagsldnewgiukhoikljjlkjgjgg}
we obtain
\begin{multline*}
\int_a^b\bigg<h(t),A_t\big(z_n(t)\big)-A_t\big(z(t)\big)\Big>_{X\times
X^*}dt=\\
\int_0^1\int_a^b\Big<h(t),DA_t\Big(sz_n(t)+(1-s)z(t)\Big)\cdot\Big(z_n(t)-z(t)\Big)\bigg>_{X\times
X^*}dtds=\\
\int_0^1\int_a^b\Big<\Big(z_n(t)-z(t)\Big),
\Big\{DA_t\big(sz_n(t)+(1-s)z(t)\big)\Big\}^*\cdot
h(t)\bigg>_{Z\times Z^*}dtds\to 0 \quad\quad\forall h(t)\in
L^{q}(a,b,X)\,.
\end{multline*}
So, by \er{uravnllllgsnachlkagagsldnewgiukhoikljjlkjgjgg},
and \er{boundnes77889999hhhhhh999ghg78999hjhjhjh99kkkk} we have
$\bar \Lambda(t)=\Lambda_t\big(z(t)\big)$ and $\bar
A(t)=A_t\big(z(t)\big)$,
and thus using \er{uravnllllgsnachlklimhjhghghty} we finally deduce
that $z(t)$ is a solution to \er{uravnllllgsnachlklimhjhgh}.

 Finally, assume that there exist a reflexive Banach space $V$, a
compact operator $L\in\mathcal{L}(Z,V)$, and for every $t\in[a,b]$ a
convex G\^{a}teaux differentiable functions $\Phi_t:Z\to\R$ and a
G\^{a}teaux differentiable mapping $F_t(\sigma):V\to Z^*$ satisfying
\er{roststlambdglkFFhjhjhhjjkjkjkjkjkllhllhlhtgythtyjuui},
\er{roststlambdglkFFhjhjhhjjkjkjkjkjkllhllhlhtgytht} and
\er{roststlambdglkFFyhuhtyui99999999ghghghh}. Then, since, as
before, we have $L\cdot z_n(t)\to L\cdot z(t)$ strongly in
$L^q(a,b;V)$, we deduce that, up to a subsequence, $F_t\big(L\cdot
z_n(t)\big)\to F_t\big(L\cdot z(t)\big)$ strongly in
$L^{q^*}(a,b;Z^*)$. On the other hand by
\er{roststlambdglkFFyhuhtyui99999999ghghghh} and
\er{vyifyurturfurfuyrfgyukkknewlhhhlljkh}
we
infer
\begin{equation}
\label{neravenstvopolsldjhjkgghkjjhg}
\frac{1}{2}\big\|w(t)\big\|_H^2+\limsup\limits_{n\to+\infty}\Bigg\{\int_a^t\bigg(\Phi_s\big(z_n(s)\big)+\Big<z_n(s),F_s\big(L\cdot
z_n(s)\big)\Big>_{Z\times
Z^*}\bigg)ds\Bigg\}\leq\frac{1}{2}\big\|w_0\big\|_H^2\quad\quad\forall
t\in[a,b]\,.
\end{equation}
Therefore, letting $n$ tend to $+\infty$ in
\er{neravenstvopolsldjhjkgghkjjhg} and using
\er{uravnllllgsnachlkagagsldnewgiukhoikljjlkjgjgg}
and the convexity of $\Phi_t$ we finally obtain
\er{Monotonegkkkkdhtdtjgkgjfgdfxfdfdnf}.
\end{proof}


As a particular case of Theorem \ref{thhypppggghhhjggg} we have the
following Theorem.
\begin{theorem}\label{thhypppggghhhj}
Let $X$ and $Z$ be reflexive Banach spaces and $X^*$ and $Z^*$ be
the corresponding dual spaces. Furthermore let $H$ be a Hilbert
space. Suppose that $Q\in \mathcal{L}(X,Z)$ is an injective
inclusion operator such that its image is
dense on $Z$. Furthermore, suppose that $P\in \mathcal{L}(Z,H)$ is
an injective inclusion operator such that its image is dense on $H$.
Let $T\in \mathcal{L}(X,H)$ be defined by $T:=P\circ Q$. So
$\{X,H,X^*\}$ is an evolution triple with the corresponding
inclusion operator $T\in \mathcal{L}(X;H)$ as it was defined in
Definition \ref{7bdf} together with the corresponding operator
$\widetilde{T}\in \mathcal{L}(H;X^*)$ defined as in \er{tildet}.
Assume also that the Banach space $X$ is separable. Next let
$a,b\in\R$ s.t. $a<b$. Furthermore, for every $t\in[a,b]$ let
$\Lambda_t\in L^\infty\big(a,b;\mathcal{L}(Z,X^*)\big)$.
%
%
%
Next let $F_t(z):Z\to X^*$ be a function which is G\^{a}teaux
differentiable at every $z\in Z$ for every $t\in[a,b]$, and
satisfies $F_t(0)\in L^{2}(a,b;X^*)$ and the Lipshitz condition
\begin{equation}\label{bghffgghgkoooojkvhgjggg}
\big\|DF_t(z)\big\|_{\mathcal{L}(Z;X^*)}\leq g\big(\|P\cdot
z\|_H\big)\quad\quad\forall z\in Z,\;\forall t\in[a,b]\,,
\end{equation}
for some nondecreasing function $g(s):[0,+\infty)\to(0,+\infty)$. We
also assume that $F_t(z)$ is
Borel on the pair of variables $(z,t)$.
Moreover, suppose that $\Lambda_t$
and $F_t$ satisfy the
following lower bound condition
\begin{multline}
\label{roststlambdglkFFyhuhtyui99999999ggg} \bigg<h,\Lambda_t\cdot
\big(Q\cdot h\big)
+F_t\big(Q\cdot h\big)\bigg>_{X\times X^*}\geq\\ \big(1/\bar
C\big)\big\|Q\cdot h\big\|^2_Z-
\Big(\big\|Q\cdot h\big\|^p_Z+\mu^{\frac{p}{2}}(t)\Big)\Bigg(\bar
C\big\|L\cdot(Q\cdot
h)\big\|^{(2-p)}_V+\mu^{\frac{2-p}{2}}(t)\Big(\big\|T\cdot
h\big\|^{(2-p)}_H+1\Big)\Bigg)
\quad\forall h\in X,\;\forall t\in[a,b],
\end{multline}
where $V$ is a given Banach space, $L\in\mathcal{L}(Z,V)$ is a given
compact operator, $p\in[0,2)$ and $\bar C>0$ are some constants and
 $\mu(t)\in L^1(a,b;\R)$ is a nonnegative function. Finally assume that
there exists a family of reflexive Banach spaces
$\{V_j\}_{j=1}^{+\infty}$ and a family of compact bounded linear
operators $\{L_j\}_{j=1}^{+\infty}$, where
$L_j\in\mathcal{L}(Z,V_j)$, which satisfy the following condition:
\begin{itemize}
\item
If $\{h_n\}_{n=1}^{+\infty}\subset Z$ is a sequence such that for
all fixed $j$ $\lim_{n\to+\infty}L_j\cdot h_n=L_j\cdot h_0$ strongly
in $V_j$ and $P\cdot h_n\rightharpoonup P\cdot h_0$ weakly in $H$,
then for every fixed $t\in(a,b)$ $F_t(h_n)\rightharpoonup F_t(h_0)$
weakly in $X^*$.
\end{itemize}
Then for every $w_0\in H$ there exists $z(t)\in L^2(a,b;Z)$ such
that $w(t):=P\cdot z(t)\in L^\infty(a,b;H)$, $v(t):=\widetilde
T\cdot \big(w(t)\big)\in W^{1,2}(a,b;X^*)$ and $z(t)$ satisfies the
following equation
\begin{equation}\label{uravnllllgsnachlklimhjhghggg}
\begin{cases}\frac{d v}{dt}(t)+\Lambda_t\cdot \big(z(t)\big)
+
F_t\big(z(t)\big)=0\quad\text{for a.e.}\; t\in(a,b)\,,\\
w(a)=w_0\,,
\end{cases}
\end{equation}
where we assume that $w(t)$ is $H$-weakly continuous on $[a,b]$, as
it was stated in Corollary \ref{vbnhjjmcor}. Moreover if we assume
in addition that there exist a reflexive Banach space $E$, a compact
operator $L_0\in\mathcal{L}(Z,E)$, and for every $t\in[a,b]$ a
G\^{a}teaux differentiable mapping $H_t(\zeta):E\to Z^*$, measurable
on $(\zeta,t)$, such that $H_t(0)\in L^2(a,b;Z^*)$ and satisfying
\begin{equation}\label{roststlambdglkFFhjhjhhjjkjkjkjkjkllhllhlhtgythtjhjhjhuytity}
\big\|DH_t(L_0\cdot z)\big\|_{\mathcal{L}(E;Z^*)}\leq \widetilde
g\big(\|P\cdot z\|_H\big)\quad\forall z\in Z,\;\forall t\in[a,b]
\end{equation}
for some nondecreasing function $\widetilde
g(s):[0,+\infty)\to(0,+\infty)$, and satisfying
\begin{equation}\label{Monotonegkkkkdhtdtggggjhfghgdfgh}
\Big<h,\Lambda_t\cdot (Q\cdot h)
+F_t (Q\cdot h)\Big>_{X\times X^*}\geq\Big<Q\cdot h,A_t\cdot (Q\cdot
h)+H_t\big((L_0\circ Q)\cdot h\big)\Big>_{Z\times
Z^*}\quad\quad\forall h\in X\,,\;\forall t\in[a,b]\,,
\end{equation}
where $A_t\in L^\infty\big(a,b;\mathcal{L}(Z,Z^*)\big)$ is such that
$\big<z,A_t\cdot z\big>_{Z\times Z^*}\geq 0$ $\forall z\in Z$,
then the function $z(t)$, as above, satisfies the following energy
inequality
\begin{equation}
\label{Monotonegkkkkdhtdtjgkgjfgdfxfdfdnfhjkjhjk}
\frac{1}{2}\big\|w(t)\big\|^2_H+\int_a^t\Big<z(s),A_s\cdot\big(z(s)\big)
+H_s\big(L_0\cdot z(s)\big)\Big>_{Z\times Z^*}ds\leq
\frac{1}{2}\big\|w_0\big\|^2_H\quad\quad \;\forall t\in[a,b]\,.
\end{equation}
\end{theorem}

As a particular case of Theorem \ref{thhypppggghhhj}, where $Z=H$,
we have the following statement, which is useful in the study of
Hyperbolic systems.
\begin{corollary}\label{thhypppggg}
Let $\{X,H,X^*\}$ be an evolution triple with the corresponding
inclusion operator $T\in \mathcal{L}(X;H)$ as it was defined in
Definition \ref{7bdf} together with the corresponding operator
$\widetilde{T}\in \mathcal{L}(H;X^*)$ defined as in \er{tildet}.
Assume also that the Banach space $X$ is separable. Next let
$a,b\in\R$ s.t. $a<b$. Furthermore, for every $t\in[a,b]$ let
$\Lambda_t\in L^\infty\big(a,b;\mathcal{L}(H,X^*)\big)$.
%
%
%
Next let $F_t(w):H\to X^*$ be a function which is G\^{a}teaux
differentiable on every $w\in H$ for every $t\in[a,b]$, and
satisfies $F_t(0)\in L^{2}(a,b;X^*)$ and the Lipshitz condition
\begin{equation}\label{bghffgghgkoooojkvhgjjjk}
\big\|DF_t(w)\big\|_{\mathcal{L}(H;X^*)}\leq
g\big(\|w\|_H\big)\quad\quad\forall w\in H,\;\forall t\in[a,b]\,,
\end{equation}
for some nondecreasing function $g(s):[0,+\infty)\to(0,+\infty)$. We
also assume that $F_t(w)$ is
Borel on the pair of variables $(w,t)$ (see Definition
\ref{fdfjlkjjkkkkkllllkkkjjjhhhkkk}).
Moreover, assume that $F_t$ is weak to weak continuous from $H$ to
$X^*$ for every fixed $t$ i.e.
for every sequence $\{h_n\}_{n=1}^{+\infty}\subset H$ such that
$h_n\rightharpoonup h_0$ weakly in $H$ and for every $t\in[a,b]$, we
have $F_t(h_n)\rightharpoonup F_t(h_0)$ weakly in $X^*$. Finally
suppose that $\Lambda_t$ and $F_t$ satisfy the following lower bound
condition
\begin{equation}
\label{roststlambdglkFFyhuhtyui99999999jjk} \Big<h,\Lambda_t\cdot
(T\cdot h)+F_t(T\cdot h)\Big>_{X\times X^*}\geq
-\mu(t)\Big(\big\|T\cdot h\big\|^2_H+1\Big)\quad\quad\forall h\in
X\,,\;\forall t\in[a,b]\,,
\end{equation}
for nonnegative function $\mu(t)\in L^1(a,b;\R)$.
Then for every $w_0\in H$ there exists $w(t)\in L^\infty(a,b;H)$,
such that $v(t):=\widetilde T\cdot \big(w(t)\big)\in
W^{1,2}(a,b;X^*)$ and $w(t)$ satisfies the following equation
\begin{equation}\label{uravnllllgsnachlklimhjhghjjk}
\begin{cases}\frac{d v}{dt}(t)+\Lambda_t\cdot \big(w(t)\big)+
F_t\big(w(t)\big)=0\quad\text{for a.e.}\; t\in(a,b)\,,\\
w(a)=w_0\,,
\end{cases}
\end{equation}
where we assume that $w(t)$ is $H$-weakly continuous on $[a,b]$, as
it was stated in Corollary \ref{vbnhjjmcor}.
\end{corollary}

\section{Applications}\label{dkgfkghfhkljl}
\subsection{Notations in the present section}
For a $p\times q$ matrix $A$ with $ij$-th entry $a_{ij}$ we denote
by $|A|=\bigl(\Sigma_{i=1}^{p}\Sigma_{j=1}^{q}a_{ij}^2\bigr)^{1/2}$
the
Frobenius norm of $A$.\\
For two matrices $A,B\in\R^{p\times q}$  with $ij$-th entries
$a_{ij}$ and $b_{ij}$ respectively, we write
$A:B\,:=\,\sum\limits_{i=1}^{p}\sum\limits_{j=1}^{q}a_{ij}b_{ij}$.\\
Given a vector valued function
$f(x)=\big(f_1(x),\ldots,f_k(x)\big):\O\to\R^k$ ($\O\subset\R^N$) we
denote by $\nabla_x f$ the $k\times N$ matrix with
$ij$-th entry $\frac{\partial f_i}{\partial x_j}$.\\
For a matrix valued function $F(x):=\{F_{ij}(x)\}:\R^N\to\R^{k\times
N}$ we denote by $div\,F$ the $\R^k$-valued vector field defined by
$div\,F:=(l_1,\ldots,l_k)$ where
$l_i=\sum\limits_{j=1}^{N}\frac{\partial F_{ij}}{\partial x_j}$.\\
For $u=(u_1,\ldots,u_p)\in\R^p$ and $v=(v_1,\ldots,v_q)\in\R^q$ we
denote by $\vec u\otimes\vec v$ the $p\times q$ matrix with $ij$-th
entry $u_i v_j$.
\subsection{A general parabolic system in a divergent form}
Let $\Psi(A,x,t):\R_A^{k\times N}\times\R_x^N\times\R_t\to\R$ be a
nonnegative measurable function. Moreover assume that $\Psi(A,x,t)$
is $C^1$ as a function of the first argument $A$ when $(x,t)$ are
fixed, which satisfies $\Psi(0,x,t)=0$ and it is convex by the first
argument $A$  when $(x,t)$ are fixed, i.e.
$$\Psi\big(\alpha A_1+(1-\alpha)A_2,x,t\big)\leq\alpha\Psi(A_1,x,t)+(1-\alpha)\Psi(A_2,x,t)$$
for every $\alpha\in[0,1]$, $A_1,A_2\in\R^{k\times N}$, $x\in\R^N$
and $t\in\R$. Moreover, we assume that $\Psi$ satisfies the
following growth condition
\begin{multline}\label{roststglkjjjaplneabst} (1/C)|A|^q-|g_0(x)|\leq
\Psi(A,x,t)\leq C|A|^q+|g_0(x)| \quad\quad\forall A\in\R^{k\times
N},\;\forall x\in\R^N,\;\forall t\in\R\,,
\end{multline}
where $C>0$ is some constant, $g_0(x)\in L^1(\R^N,\R)$
and $q\in [2,+\infty)$. Next let $\Gamma(A,x,t):\R_A^{k\times
N}\times\R_x^N\times\R_t\to\R^{k\times N}$ be a measurable function.
Moreover assume that $\Gamma(A,x,t)$ is $C^1$ as a function of the
first argument $A$ when $(x,t)$ are fixed, which satisfies,
\begin{equation}\label{jeheefkeplpl}
\Gamma(0,x,t)\in L^{q^*}\big(\R;L^{2}(\R^N,\R^{k\times N})\big)\,,
\end{equation}
the following monotonicity condition
\begin{equation}\label{Monotoneglkjjjaplneabst}
\sum\limits_{1\leq j,n\leq N}\sum\limits_{1\leq i,m\leq k}
H_{ij}H_{mn}\frac{\partial \Gamma_{mn}}{\partial A_{ij}}(A,x,t)\geq
0 \quad\quad\forall H,A\in \R^{k\times N},\;\forall
x\in\R^N,\;\forall t\in\R\,,
\end{equation}
and the following growth condition
\begin{multline}\label{roststlambdglkjjjaplneabst}
\bigg|\frac{\partial \Gamma}{\partial A_{ij}}(A,x,t)\bigg|\leq
C\,|A|^{q-2}+ C\\ \forall A\in \R^{k\times N},\;\forall x\in
\R^N,\;\forall t\in\R,\quad\forall i\in\{1,\ldots,k\},\;\forall
j\in\{1,\ldots,N\}\,,
\end{multline}
where $C>0$ is some constant. Finally let
$\Xi(B,x,t):\R_B^k\times\R_x^N\times\R_t\to\R^{k\times N}$ and
$\Theta(B,x,t):\R_B^k\times\R_x^N\times\R_t\to\R^{k}$ be two
measurable functions. Moreover, assume that $\Xi(B,x,t)$ and
$\Theta(B,x,t)$ are $C^1$ as a functions of the first argument $B$
when $(x,t)$ are fixed. We also assume that $\Xi(B,x,t)$ and
$\Theta(B,x,t)$ are globally Lipschitz by the first argument $B$ and
satisfy
\begin{equation}\label{jeheefkeplplgryrtu}
\Xi(0,x,t)\in L^{q^*}\big(\R;L^{2}(\R^N,\R^{k\times
N})\big),\;\;\Theta(0,x,t)\in
L^{q^*}\big(\R;L^{2}(\R^N,\R^k)\big)\,.
\end{equation}
\begin{proposition}\label{divform}
Let $\Psi,\Gamma,\Xi,\Theta$ be as above and let $\Omega\subset\R^N$
be a bounded open set, $2\leq q<+\infty$ and $T_0>0$.
Then for every $w_0(x)\in L^2(\O,\R^k)$ there exists $u(x,t)\in
L^q\big(0,T_0;W_0^{1,q}(\O,\R^k)\big)$, such that $u(x,t)\in
L^\infty\big(0,T_0;L^2(\O,\R^k)\big)\cap W^{1,q^*}\big(0,T_0;
W^{-1,q^*}(\O,\R^k)\big)$, where $q^*:=q/(q-1)$, $u(x,t)$ is
$L^2(\O,\R^k)$-weakly continuous on $[0,T_0]$, $u(x,0)=w_0(x)$ and
$u(x,t)$ is a solution to
\begin{multline}\label{uravnllllgsnachlklimjjjaplneabst}
\frac{d u}{dt}(x,t)=\Theta\big(u(x,t),x,t\big)+div_x
\Big(\Xi\big(u(x,t),x,t\big)\Big)+\\ div_x\Big(\Gamma\big(\nabla_x
u(x,t),x,t\big)\Big)+ div_x \Big(D_A\Psi\big(\nabla_x
u(x,t),x,t\big)\Big)\quad\quad\text{in}\;\;\O\times(0,T_0)\,,
\end{multline}
where
$$D_A\Psi(A,x,t):=\bigg\{\frac{\partial \Psi}{\partial A_{ij}}(A,x,t)\bigg\}_{1\leq i\leq k,1\leq j\leq N}\in\R^{k\times N}\,.$$
Moreover if $\Psi(A,x,t)$ is a uniformly convex function by the
first argument $A$ then such a solution $u$ is unique.
\end{proposition}
\begin{proof}
Let $X:=W_0^{1,q}(\O,\R^k)$
(a separable reflexive Banach space), $H:=L^2(\O,\R^k)$ (a Hilbert
space) and $T\in \mathcal{L}(X;H)$ be a usual embedding operator
from $W_0^{1,q}(\O,\R^k)$ into $L^2(\O,\R^k)$. Then $T$ is an
injective inclusion with dense image. Furthermore,
$X^*=W^{-1,q^*}(\O,\R^k)$ where $q^*=q/(q-1)$ and the corresponding
operator $\widetilde{T}\in \mathcal{L}(H;X^*)$, defined as in
\er{tildet}, is a usual inclusion of $L^2(\O,\R^k)$ into
$W^{-1,q^*}(\O,\R^k)$.
Then $\{X,H,X^*\}$ is an evolution triple with the
corresponding inclusion operators $T\in \mathcal{L}(X;H)$ and
$\widetilde{T}\in \mathcal{L}(H;X^*)$, as it was defined in
Definition \ref{7bdf}. Moreover, by the Theorem about the compact
embedding in Sobolev Spaces it is well known that
$T$ is a compact operator.

Next, for every $t\in[0,T_0]$ let $\Phi_t(x):X\to[0,+\infty)$ be
defined by
\begin{equation*}
\Phi_t(u):=\int_\O\Psi\big(\nabla u(x),x,t\big)dx
+\frac{k_\O}{2}\int_\O|u(x)|^2dx \quad\forall u\in
W^{1,q}(\O,\R^k)\equiv X\,.
\end{equation*}
where
\begin{equation}\label{ghkgjgiooioioppiljkjk}
k_\O:=
\begin{cases}
0\quad\quad\text{if}\;\;\O\;\;\text{is bounded}\,,\\
1\quad\quad\text{if}\;\;\O\;\;\text{is unbounded}\,.
\end{cases}
\end{equation}
Then
$\Phi_t(x)$ is G\^{a}teaux differentiable at every $x\in X$, satisfy
$\Phi_t(0)=0$ and by \er{roststglkjjjaplneabst} it satisfies the
growth condition
\begin{equation*}
(1/C)\,\|x\|_X^q-C\leq \Phi_t(x)\leq C\,\|x\|_X^q+C\quad\forall x\in
X,\;\forall t\in[0,T]\,,
\end{equation*}
Furthermore, for every $t\in[0,T_0]$ let $\Lambda_t(x):X\to X^*$ be
defined by
\begin{equation*}
\Big<\delta,\Lambda_t(u)\Big>_{X\times X^*}:=\int_\O
\Gamma\big(\nabla u(x),x,t\big):\nabla\delta(x)\,dx\quad\forall
u,\delta\in W^{1,q}(\O,\R^k)\equiv X\,.
\end{equation*}
Then
$\Lambda_t(x):X\to X^*$ is G\^{a}teaux differentiable at every $x\in
X$, and, by \er{roststlambdglkjjjaplneabst} its derivative satisfies
the growth condition
\begin{equation*}
\|D\Lambda_t(x)\|_{\mathcal{L}(X;X^*)}\leq C\,\|x\|_X^{q-2}+
C\quad\forall x\in X,\;\forall t\in[0,T_0]\,,
\end{equation*}
for some $C>0$. Moreover, by \er{Monotoneglkjjjaplneabst},
$\Lambda_t$ satisfy the following monotonicity conditions
\begin{equation*}
\Big<h,D\Lambda_t(x)\cdot h\Big>_{X\times X^*}\geq 0
\quad\quad\forall x,h\in X\;\forall t\in[0,T_0]\,.
\end{equation*}
Finally for every $t\in[0,T_0]$ let $F_t(w):H\to X^*$ be defined by
\begin{multline}\label{jhfjggkjkjhhkhkloopll}
\Big<\delta,F_t(w)\Big>_{X\times X^*}:=\int_\O
\bigg\{\Xi\big(w(x),x,t\big):\nabla\delta(x)- \Big(k_\O w(x)+
\Theta\big(w(x),x,t\big) \Big)
\cdot\delta(x)\bigg\}dx\\
\forall w\in L^2(\O,\R^k)\equiv H,\; \forall\delta\in
W^{1,q}(\O,\R^k)\equiv X\,.
\end{multline}
Then
$F_t(w)$ is G\^{a}teaux differentiable at every $w\in H$, and, since
$\Xi$ and $\Theta$ are Lipshitz functions, the derivative of
$F_t(w)$ satisfy the Lipschitz condition
\begin{equation}
\label{roststlambdglkFFjjjapljjjkhkl}
\|DF_t(w)\|_{\mathcal{L}(H;X^*)}\leq C\quad\forall w\in H,\;\forall
t\in[0,T_0]\,,
\end{equation}
for some $C>0$.
%
%
%
%
%
%
Thus
all the conditions of Theorem \ref{defHkkkkglkjjj} are satisfied.
Applying this Theorem completes the proof.
\end{proof}
\begin{remark}\label{vhhgyhgjfgghfggh}
If in the framework of Proposition \ref{divform} we suppose $q=2$
and that $D_A\Psi(A,x,t)$ and $\Gamma(A,x,t)$ are linear by the
first argument $A$, however we assume that $\O$ is unbounded, we
obtain the similar existence result as in Proposition \ref{divform},
as a consequence of Theorem \ref{thhypppggghhhj} with $Z=X$.

 Indeed in the case of unbounded $\O$, let $V_j=L^2(\O\cap
B_{R_j}(0),\R^k)$ for some sequence $R_j\to +\infty$ and set $L_j\in
\mathcal{L}(H,V_j)$ by
$$L_j\cdot \big(h(x)\big):=h(x)\llcorner\O\cap B_{R_j}(0)\in L^2(\O\cap B_{R_j}(0),\R^k)=V_j\quad\quad\forall h(x)\in L^2(\O,\R^k)=H\,.$$
Then by the standard embedding theorems in the Sobolev Spaces the
operator $L_j\circ T\in \mathcal{L}(X,V_j)$ is compact for every
$j$. Moreover, if $\{h_n\}\subset H$ is a sequence such that
$h_n\rightharpoonup h_0$ weakly in $H$ and $L_j\cdot h_n\to L_j\cdot
h_0$ strongly in $V_j$ as $n\to +\infty$ for every $j$, then we have
$h_n\to h_0$ strongly in $L^2_{loc}(\O,\R^k)$ and thus, by
\er{jhfjggkjkjhhkhkloopll} and \er{roststlambdglkFFjjjapljjjkhkl} we
must have $F_t(h_n)\rightharpoonup F_t(h_0)$ weakly in $X^*$.
\end{remark}

\subsection{Parabolic systems in a non-divergent form}
Let $\Psi(L,x,t):\R_L^{k}\times\R_x^N\times\R_t\to\R$ be a
nonnegative measurable function. Moreover, assume that $\Psi(L,x,t)$
is $C^1$ as a function of the first argument $L$ when $(x,t)$ are
fixed, which satisfies $\Psi(0,x,t)=0$ and it is convex by the first
argument $L$ when $(x,t)$ are fixed, i.e.
$$\Psi\big(\alpha L_1+(1-\alpha)L_2,x,t\big)\leq\alpha\Psi(L_1,x,t)+(1-\alpha)\Psi(L_2,x,t)$$
for every $\alpha\in[0,1]$, $L_1,L_2\in\R^{k}$, $x\in\R^N$ and
$t\in\R$. Moreover, we assume that $\Psi$ satisfies the following
growth condition
\begin{equation}\label{roststglkjjjaplneabstnedv} (1/C)|L|^q-C\leq
\Psi(L,x,t)\leq C|L|^q+C \quad\quad\forall L\in\R^{k},\;\forall
x\in\R^N,\;\forall t\in\R\,,
\end{equation}
where $C>0$ is some constant and $q\in [2,+\infty)$. Next let
$\Gamma(L,x,t):\R_L^{k}\times\R_x^N\times\R_t\to\R^{k}$ be a
measurable function. Moreover, assume that $\Gamma(L,x,t)$ is $C^1$
as a function of the first argument $L$ when $(x,t)$ are fixed,
which satisfies
\begin{equation}\label{jeheefkeplplkjjkgj}
\Gamma(0,x,t)\in L^{q^*}\big(\R;L^{2}(\R^N,\R^k)\big)\,,
\end{equation}
the following monotonicity condition
\begin{equation}\label{Monotoneglkjjjaplneabstnedv}
\sum\limits_{1\leq i,j\leq k} h_{i}h_{j}\frac{\partial
\Gamma_{i}}{\partial L_{j}}(L,x,t)\geq 0 \quad\quad\forall\, h,L\in
\R^{k},\;\forall x\in\R^N,\;\forall t\in\R\,,
\end{equation}
and the following growth condition
\begin{multline}\label{roststlambdglkjjjaplneabstnedv}
\bigg|\frac{\partial \Gamma}{\partial L_{j}}(L,x,t)\bigg|\leq
C\,|L|^{q-2}+ C\quad\quad \forall L\in \R^{k},\;\forall x\in
\R^N,\;\forall t\in\R,\quad\forall j\in\{1,\ldots,k\}\,.
\end{multline}
Finally let $\Theta(A,L,x,t):\R_A^{k\times
N}\times\R_L^k\times\R_x^N\times\R_t\to\R^{k}$ be a measurable
function. Moreover, assume that $\Theta(A,L,x,t)$ is $C^1$ as a
function of the first two arguments $A$ and $L$ when $(x,t)$ are
fixed. We also assume that $\Theta(A,L,x,t)$ is globally Lipschitz
by the first two arguments $A$ and $L$ and
\begin{equation}\label{jeheefkeplpllhkhlkhllh}
\Theta(0,0,x,t)\in L^{q^*}\big(\R;L^2(\R^N,\R^k)\big)\,.
\end{equation}
\begin{proposition}\label{divformnedv}
Let $\Psi,\Gamma,\Theta$ be as above and let $\Omega\subset\R^N$ be
a bounded open set,
$2\leq q<+\infty$ and $T_0>0$.
Then for every $w_0(x)\in W^{1,2}_0(\O,\R^k)$ there exists
$u(x,t)\in L^q\big(0,T_0;W^{2,q}_{loc}(\O,\R^k)\big)$, such that
$\Delta_x u(x,t)\in L^q\big(0,T_0;L^q(\O,\R^k)\big)$, $u(x,t)\in
L^\infty\big(0,T_0;W^{1,2}_0(\O,\R^k)\big)\cap W^{1,q^*}\big(0,T_0;
L^{q^*}(\O,\R^k)\big)$ ,where $q^*:=q/(q-1)$, $u(x,t)$ is
$W^{1,2}_0(\O,\R^k)$-weakly continuous on $[0,T_0]$, $u(x,0)=w_0(x)$
and $u(x,t)$ is a solution to
\begin{multline}\label{uravnllllgsnachlklimjjjaplneabstnedv}
\frac{d u}{dt}(x,t)=\Theta\big(\nabla_x u(x,t),u(x,t),x,t\big)+
\Gamma\big(\Delta_x u(x,t),x,t\big)+ \nabla_L\Psi\big(\Delta_x
u(x,t),x,t\big)\quad\quad\text{in}\;\;\O\times(0,T_0)\,,
\end{multline}
where $\nabla_L\Psi(L,x,t)$ is a partial gradient by the first
variable $L$.
Moreover if $\Psi(L,x,t)$ is uniformly convex function by the first
argument $L$ then such a solution $u$ is unique.
\end{proposition}
\begin{proof}
Let
\begin{equation}\label{hfhfvgfibihfbhfhbfhbh}
X:=\Big\{u(x)\in W_0^{1,2}(\O,\R^k):\;\Delta u(x)\in
L^q(\O,\R^k)\Big\}\,,
\end{equation}
for $2\leq q<+\infty$ endowed with the norm
\begin{equation}\label{nrmpoprgbjuyh}
\|u\|_X:=\|\Delta u\|_{L^q(\O,\R^k)}+\|\nabla
u\|_{L^2(\O,\R^{k\times N})}
\quad\quad\forall u\in X\subset W_0^{1,2}(\O,\R^k)\,.
\end{equation}
Thus $X$ is a separable reflexive Banach space. Next let
$H:=W^{1,2}_0(\O,\R^k)$ endowed with the standard scalar product
$<\phi_1,\phi_2>_{H\times
H}=\int_\O\nabla\phi_1(x):\nabla\phi_2(x)\,dx$ (a Hilbert space) and
$T\in \mathcal{L}(X;H)$ be a trivial embedding operator from
$X\subset W_0^{1,2}(\O,\R^k)$ into $H=W_0^{1,2}(\O,\R^k)$. Then $T$
is an injective inclusion with dense image. Moreover,
$T$ is a compact operator. In order to follow the definitions above
we identify the dual space $H^*$ with $H$. So in our notations
$\big\{W^{1,2}_0(\O,\R^k)\big\}^*=W^{1,2}_0(\O,\R^k)$ (although in
the usual notations $\big\{W^{1,2}_0(\O,\R^k)\big\}^*$ identified
with the isomorphic space $W^{-1,2}(\O,\R^k)$). Next define
$S\in\mathcal{L}\big(L^{q*}(\O,\R^k), X^*\big)$ by the formula
\begin{equation}\label{funcoprrrthhj}
\Big<\delta,S\cdot h\Big>_{X\times X^*}=-\int_\O
h(x)\cdot\Delta\delta(x)\,dx\quad\quad\forall \delta\in X
,\;\forall h\in L^{q^*}(\O,\R^k)\,.
\end{equation}
Then, since for every $\phi\in L^q(\O,\R^k)$ there exists unique
$\delta_\phi\in X$ such that $\Delta\delta_\phi=\phi$ we deduce
that $S$ is an injective inclusion (i.e. $\ker S=0$).
For the corresponding operator $\widetilde{T}\in
\mathcal{L}(H;X^*)$, by \er{tildet} and \er{funcoprrrthhj} we must
have
\begin{multline}\label{tildetfddgfjhnhj}
\big<u,\widetilde{T}\cdot w\big>_{X\times X^*}:=\big<T\cdot
u,w\big>_{H\times H}=\int_\O\nabla u(x):\nabla w(x)\,dx=-\int_\O
w(x)\cdot\Delta u(x)\,dx=\Big<u,S\cdot(L\cdot w)\Big>_{X\times X^*}\\
\text{for every}\; w\in H\;\text{and}\;u\in X\,,
\end{multline}
where $L$ is a trivial inclusion of $W^{1,2}_0(\O,\R^k)$ into
$L^{q^*}(\O,\R^k)$ ($q^*\leq 2$). So
\begin{equation}\label{gbgyfyffftfdfy687t8k}
\widetilde{T}=S\circ L\,.
\end{equation}
Then $\{X,H,X^*\}$ is an evolution triple with the corresponding
inclusion operators $T\in \mathcal{L}(X;H)$ and $\widetilde{T}\in
\mathcal{L}(H;X^*)$, as it was defined in Definition \ref{7bdf}.

 Next, for every $t\in[0,T_0]$ let $\Phi_t(x):X\to[0,+\infty)$ be
defined by
\begin{equation*}
\Phi_t(u):=\int_\O \Big( \Psi\big(\Delta u(x),x,t\big)
+\frac{1}{2}\big|\nabla u(x)\big|^2\Big) dx\quad\forall u\in
X\,.
\end{equation*}
Then
$\Phi_t(x)$ is G\^{a}teaux differentiable at every $x\in X$,
satisfies $\Phi_t(0)=0$ and
it satisfies the
growth condition
\begin{equation*}
(1/C)\,\|x\|_X^q-C\leq \Phi_t(x)\leq C\,\|x\|_X^q+C\quad\forall x\in
X,\;\forall t\in[0,T_0]\,.
\end{equation*}
Furthermore, for every $t\in[0,T_0]$ let $\Lambda_t(x):X\to X^*$ be
defined by
\begin{equation*}
\Big<\delta,\Lambda_t(u)\Big>_{X\times X^*}:=\int_\O
\Gamma\big(\Delta u(x),x,t\big)\cdot\Delta\delta(x)\,dx\quad\forall
u,\delta\in
X\,,
\end{equation*}
i.e.
\begin{equation}\label{fhdhvhdfhivhiihf}
\Lambda_t(u)=-S\cdot\Big(\Gamma\big(\Delta
u(x),x,t\big)\Big)\quad\forall u\in X\,.
\end{equation}
Then
$\Lambda_t(x):X\to X^*$ is G\^{a}teaux differentiable at every $x\in
X$, and, by \er{roststlambdglkjjjaplneabst} its derivative satisfies
the growth condition
\begin{equation*}
\|D\Lambda_t(x)\|_{\mathcal{L}(X;X^*)}\leq C\,\|x\|_X^{q-2}+
C\quad\forall x\in X,\;\forall t\in[0,T_0]\,,
\end{equation*}
for some $C>0$. Moreover, by \er{Monotoneglkjjjaplneabst},
$\Lambda_t$ satisfies the following monotonicity conditions
\begin{equation*}
\Big<h,D\Lambda_t(x)\cdot h\Big>_{X\times X^*}\geq 0
\quad\quad\forall x,h\in X\;\forall t\in[0,T_0]\,.
\end{equation*}
Finally for every $t\in[0,T_0]$ let $F_t(w):H\to X^*$ be defined by
\begin{multline*}
\Big<\delta,F_t(w)\Big>_{X\times X^*}:=\int_\O \Big(
\Theta\big(\nabla w(x),w(x),x,t\big) +w(x)\Big) \cdot\Delta\delta(x)
dx\quad \forall w\in W^{1,2}_0(\O,\R^k)\equiv H,\; \forall\delta\in
X\,,
\end{multline*}
i.e.
\begin{equation}\label{fhdhvhdfhivhiihfhugugugulll}
F_t(w)=-S\cdot\Big( \Theta\big(\nabla w(x),w(x),x,t\big) +w(x) \Big)
\quad\forall w\in H\,.
\end{equation}
Then
$F_t(w)$ is G\^{a}teaux differentiable at every $w\in H$, and, since
$\Theta$ is a Lipshitz function, the derivative of $F_t(w)$
satisfies Lipschitz condition
\begin{equation}
\label{roststlambdglkFFjjjapljjjhhhhh}
\|DF_t(w)\|_{\mathcal{L}(H;X^*)}\leq C\quad\forall w\in H,\;\forall
t\in[0,T_0]\,.
\end{equation}
%
%
%

 Thus all the conditions of Theorem \ref{defHkkkkglkjjj} are
satisfied. Applying this Theorem and \er{funcoprrrthhj}, we obtain
that
for every $w_0(x)\in W^{1,2}_0(\O,\R^k)$ there exists $u(x,t)\in
L^q\big(0,T_0;W^{2,q}_{loc}(\O,\R^k)\big)$, such that $u(x,t)\in
L^\infty\big(0,T_0;W^{1,2}_0(\O,\R^k)\big)$ ,where $q^*:=q/(q-1)$,
$u(x,t)$ is $W^{1,2}_0(\O,\R^k)$-weakly continuous on $[0,T_0]$,
$u(x,0)=w_0(x)$ and $u(x,t)$ is a solution to
\begin{equation}\label{uravnllllgsnachlklimjjjapltggtjgffyyfiyllll}
\frac{d
v}{dt}(t)+\Lambda_t\big(u(t)\big)+F_t\big(w(t)\big)+D\Phi_t\big(u(t)\big)=0\quad\text{for
a.e.}\; t\in(0,T_0)\,.
\end{equation}
Thus, by \er{uravnllllgsnachlklimjjjapltggtjgffyyfiyllll},
\er{funcoprrrthhj}, \er{gbgyfyffftfdfy687t8k},
\er{fhdhvhdfhivhiihf}, \er{fhdhvhdfhivhiihfhugugugulll} and Lemma
\ref{vlozhenie} we infer that $u(x,t)\in W^{1,q^*}\big(0,T_0;
L^{q^*}(\O,\R^k)\big)$ and
\begin{multline}\label{uravnllllgsnachlklimjjjaplneabstnedvggg}
\int\limits_\O\Bigg\{-\frac{d u}{dt}(x,t)+\Theta\big(\nabla_x
u(x,t),u(x,t),x,t\big)\\+ \Gamma\big(\Delta_x u(x,t),x,t\big)+
\nabla_L\Psi\big(\Delta_x
u(x,t),x,t\big)\Bigg\}\cdot\Delta\delta(x)\,dx=0\\ \forall
t\in(0,T_0),\;\forall \delta(x)\in X\,.
\end{multline}
Therefore
\begin{multline}\label{uravnllllgsnachlklimjjjaplneabstnedvgggjkjhi}
\frac{d u}{dt}(x,t)=\Theta\big(\nabla_x u(x,t),u(x,t),x,t\big)+
\Gamma\big(\Delta_x u(x,t),x,t\big)+ \nabla_L\Psi\big(\Delta_x
u(x,t),x,t\big)\\ \forall (x,t)\in\O\times(0,T_0)\,,
\end{multline}
and the result follows.
\end{proof}

\subsection{Hyperbolic systems of second order}
%
%
%
%
%
%
%
%
%
%
\begin{proposition}\label{divformnedvkkk}
Let $\Omega\subset\R^N$ be an open set and $T_0>0$. Furthermore, let
$\Xi(L,x,t):\R_L^k\times\R_x^N\times\R_t\to\R^{k\times N}$,
$\Upsilon(L,x,t):\R_L^k\times\R_x^N\times\R_t\to\R^{k}$ and
$\Theta(L,x,t):\R_L^k\times\R_x^N\times\R_t\to\R^{k}$ be measurable
functions. Moreover, assume that $\Xi(L,x,t)$, $\Upsilon(L,x,t)$ and
$\Theta(L,x,t)$ are $C^1$ as a functions of the first argument $L$
when $(x,t)$ are fixed. We also assume that $\Upsilon(L,x,t)$
$\nabla_x\Upsilon(L,x,t)$, $\Theta(L,x,t)$, $\Xi(L,x,t)$ and
$\nabla_x\Xi(L,x,t)$ are globally Lipschitz by the first argument
$L$, $\Upsilon(L,x,t)$ is globally Lipschitz by the last argument
$t$, $\Theta(0,x,t)\in L^2\big(\R;L^2(\R^N,\R^k)\big)$,
$\Xi(0,x,t)\in L^2\big(\R;W^{1,2}(\R^N,\R^{k\times N})\big)$
and $\Upsilon(0,x,t)\in L^2\big(\R;W^{1,2}_0(\O,\R^k)\big)$.
Then for every $w_0(x)\in W^{1,2}_0(\O,\R^k)$ and $h_0(x)\in
L^2(\O,\R^k)$ there exists
$\,u(x,t)\in L^\infty\big(0,T_0;W^{1,2}_0(\O,\R^k)\big)$ such that
$\frac{d u}{dt}(x,t)\in L^\infty\big(0,T_0;L^2(\O,\R^k)\big)\cap
W^{1,2}\big(0,T_0; W^{-1,2}(\O,\R^k)\big)$,
$u(x,t)$ is $W^{1,2}_0(\O,\R^k)$-weakly continuous on $[0,T_0]$,
$\frac{d u}{dt}(x,t)$ is $L^2(\O,\R^k)$-weakly continuous on
$[0,T_0]$, $u(x,0)=w_0(x)$, $\frac{d u}{dt}(x,0)=h_0(x)$ and
$u(x,t)$ is a solution to
\begin{multline}\label{uravnllllgsnachlklimjjjaplneabstnedvkkk}
\frac{d^2 u}{dt^2}(x,t)-\Delta_x
u(x,t)+\partial_t\big\{\Upsilon\big(u(x,t),x,t\big)\big\}+
div_x\big\{\Xi\big(u(x,t),x,t\big)\big\}+\Theta\big(u(x,t),x,t\big)=0\\
\text{in}\;\;\O\times(0,T_0)\,.
\end{multline}
\end{proposition}
\begin{proof}
Let $X_0:=\big\{\varphi\in W_0^{1,2}(\O,\R^k)\cap
W^{2,2}_{loc}(\O,\R^k):\Delta\varphi\in L^2(\O,\R^k) \big\}$ endowed
with the norm
\begin{equation}\label{nrmpoprgbjuyhkkk}
\|\varphi\|_{X_0}:=\sqrt{\|\Delta
\varphi\|^2_{L^2(\O,\R^k)}+\|\nabla \varphi\|^2_{L^2(\O,\R^{k\times
N})}+\|\varphi\|^2_{L^2(\O,\R^{k})}} \quad\quad\forall \varphi\in
X_0\subset W^{2,2}_{loc}(\O,\R^k)\cap W_0^{1,2}(\O,\R^k)\,,
\end{equation}
Then
$X_0$ is a separable reflexive Banach space. Next let
$H_0:=W^{1,2}_0(\O,\R^k)$ endowed with the standard scalar product
$<\phi_1,\phi_2>_{H\times
H}=\int_\O\big(\nabla\phi_1(x):\nabla\phi_2(x)+\phi_1(x)\cdot\phi_2(x)\big)\,dx$
(a Hilbert space) and $\mathcal{T}_0\in \mathcal{L}(X_0;H_0)$ be a
trivial embedding operator from $X_0\subset W_0^{1,2}(\O,\R^k)$ into
$H_0=W_0^{1,2}(\O,\R^k)$. Then $\mathcal{T}_0$ is an injective
inclusion with dense image. As before, in out notations,
$\big\{W^{1,2}_0(\O,\R^k)\big\}^*=W^{1,2}_0(\O,\R^k)$ (although in
the usual notations $\big\{W^{1,2}_0(\O,\R^k)\big\}^*$ identified
with the isomorphic space $W^{-1,2}(\O,\R^k)$ ). Next,
define $S_0\in\mathcal{L}\big(L^2(\O,\R^k),X_0^*\big)$ by
\begin{equation}
\label{funcoprrrthhjjjkkkk} \big<\delta,S_0\cdot h\big>_{X_0\times
X_0^*}=\int_{\O}\Big(\delta(x)-\Delta\delta(x)\Big)\cdot
h(x)\,dx\quad\quad \forall \delta\in X_0
,\;\forall h\in L^2(\O,\R^k)\,.
\end{equation}
Then, since for every $\phi\in L^2(\O,\R^k)$ there exists unique
$\delta_\phi\in X_0$ such that
$(\Delta\delta_\phi-\delta_\phi)=\phi$ we deduce that $S_0$ is
injective inclusion (i.e. $\ker S_0=0$). As before,
$\{X_0,H_0,X_0^*\}$ is an evolution triple with the corresponding
inclusion operators $\mathcal{T}_0\in \mathcal{L}(X_0;H_0)$ and
$\widetilde{\mathcal{T}}_0\in \mathcal{L}(H_0;X_0^*)$, as it was
defined in Definition \ref{7bdf} by
\begin{equation}\label{tildetkkkklllllkkkkk}
\big<\delta,\widetilde{\mathcal{T}}_0\cdot \varphi\big>_{X_0\times
X_0^*}:=\big<\mathcal{T}_0\cdot \delta,\varphi\big>_{H_0\times
H_0}\quad\quad\text{for every}\; \varphi\in
H_0\;\text{and}\;\delta\in X_0\,.
\end{equation}
However,
\begin{multline}\label{tildetkkkklllllkklklklklkkkkkk}
\big<\mathcal{T}_0\cdot \delta,\varphi\big>_{H_0\times
H_0}=\int_\O\Big(\nabla
\delta(x):\nabla\varphi(x)+\delta(x)\cdot\varphi(x)\Big)dx=\\
\int_\O\Big(\delta(x)-\Delta
\delta(x)\Big)\cdot\varphi(x)dx=\big<\delta,(S_0\circ L)\cdot
\varphi\big>_{X_0\times X_0^*}\quad\quad \text{for every}\;
\varphi\in H_0\;\text{and}\;\delta\in X_0\,,
\end{multline}
where $L\in\mathcal{L}\big(W^{1,2}_0(\O,\R^k),L^2(\O,\R^k)\big)$ is
a trivial inclusion of $W^{1,2}_0(\O,\R^k)$ into $L^2(\O,\R^k)$.
Thus plugging \er{tildetkkkklllllkklklklklkkkkkk} into
\er{tildetkkkklllllkkkkk} we obtain
\begin{equation}\label{tildetkkkklllllkkklklkklklklkkjjkkkkkkkk}
\widetilde{\mathcal{T}}_0\cdot \varphi=S_0\cdot (L\cdot
\varphi)\quad\quad\text{for every}\; \varphi\in H_0\,.
\end{equation}
Next, as in the proof of Proposition \ref{divform}, let
$X_1:=W_0^{1,2}(\O,\R^k)$, $H_1:=L^2(\O,\R^k)$ and $T_1\in
\mathcal{L}(X_1;H_1)$ be a usual embedding operator from
$W_0^{1,2}(\O,\R^k)$ into $L^2(\O,\R^k)$. Then $T_1$ is an injective
inclusion with dense image. Furthermore, $X_1^*=W^{-1,2}(\O,\R^k)$
and the corresponding operator $\widetilde{T}_1\in
\mathcal{L}(H_1;X_1^*)$, defined as in \er{tildet}, is a usual
inclusion of $L^2(\O,\R^k)$ into $W^{-1,2}(\O,\R^k)$. Thus
$\{X_1,H_1,X_1^*\}$ is another evolution triple with the
corresponding inclusion operators $T_1\in \mathcal{L}(X_1;H_1)$ and
$\widetilde{T}_1\in \mathcal{L}(H_1;X_1^*)$, as it was defined in
Definition \ref{7bdf}.
Finally set
\begin{multline}\label{obshspkkkl}
X:=\Big\{\big(u(x),v(x)\big):\;u(x):\O\to\R^k,\;v(x):\O\to\R^k \\
u(x)\in X_0\subset W^{2,2}_{loc}(\O,\R^k)\cap
W_0^{1,2}(\O,\R^k),\;v(x)\in X_1\equiv W_0^{1,2}(\O,\R^k)\Big\}\,.
\end{multline}
In this space we consider the norm
\begin{equation}\label{nrmojiihguf}
\|z\|_X:=\sqrt{\|u\|^2_{X_0}+\|v\|^2_{X_1}}=\sqrt{\|\Delta
u\|^2_{L^2(\O,\R^k)}+\|u\|^2_{W^{1,2}_0(\O,\R^k)}+\|v\|^2_{W^{1,2}_0(\O,\R^k)}}\quad\quad\forall
z=(u,v)\in X\,.
\end{equation}
Thus $X$ is a separable reflexive Banach space. Next set
\begin{multline}\label{obshspkkklHHH}
H:=\Big\{\big(u(x),v(x)\big):\;u(x):\O\to\R^k,\;v(x):\O\to\R^k \\
u(x)\in H_0\equiv  W_0^{1,2}(\O,\R^k),\;v(x)\in H_1\equiv
L^2(\O,\R^k)\Big\}\,.
\end{multline}
In this space we consider the scalar product
\begin{multline}\label{nrmojiihgufHHH}
<z_1,z_2>_{H\times H}:=<u_1,u_2>_{H_0\times
H_0}+<v_1,v_2>_{H_1\times H_1}\\=\int_\O\big\{\nabla u_1(x):\nabla
u_2(x)+u_1(x)\cdot u_2(x)+v_1(x)\cdot
v_2(x)\big\}dx\quad\quad\forall z_1=(u_1,v_1),z_2=(u_2,v_2)\in H\,.
\end{multline}
Then $H$ is a Hilbert space. Furthermore, consider
$T\in\mathcal{L}(X,H)$ by
\begin{equation}\label{nrmojiihgufoper}
T\cdot z=\big(\mathcal{T}_0\cdot u,T_1\cdot v\big)\quad\quad\forall
z=(u,v)\in X\,.
\end{equation}
Thus $T$ is an injective inclusion with dense image. Furthermore,
\begin{equation}
\label{obshspkkklsopr} X^*:=\Big\{\big(u,v\big):\;
u\in X^*_0,\;v\in X^*_1\equiv W^{-1,2}(\O,\R^k)\Big\}\,,
\end{equation}
where
\begin{equation}\label{funcoprrrthhjkkkhhhbjhjhjhyh}
\big<\delta,h\big>_{X\times X^*}=\big<\delta_0,h_0\big>_{X_0\times
X_0^*}+\big<\delta_1,h_1\big>_{X_1\times X_1^*}\quad\quad\forall
\delta=(\delta_0,\delta_1)\in X,\;\forall h=(h_0,h_1)\in X^*\,,
\end{equation}
and
\begin{equation}\label{nrmojiihgufdghsopr}
\|z\|_{X^*}:=\Big(\|u\|^2_{X^*_0}+\|v\|^2_{X^*_1}\Big)^{1/2}\quad\quad\forall
z=(u,v)\in X^*\,.
\end{equation}
Moreover, the corresponding operator $\widetilde{T}\in
\mathcal{L}(H;X^*)$, defined as in \er{tildet}, is defined by
\begin{equation}\label{nrmojiihgufoperhutild}
\widetilde{T}\cdot z=\big(\widetilde {\mathcal{T}_0}\cdot
u,\widetilde{T}_1\cdot v\big)\quad\quad\forall z=(u,v)\in H\,.
\end{equation}
Thus $\{X,H,X^*\}$ is an evolution triple with the corresponding
inclusion operators $T\in \mathcal{L}(X;H)$ and $\widetilde{T}\in
\mathcal{L}(H;X^*)$, as it was defined in Definition \ref{7bdf}.

 Next let $\Lambda\in \mathcal{L}(H,X^*)$ be defined by by
\begin{equation}\label{dvbgfghtfgjgyhjkghliuj}
\Lambda\cdot z:=(S_0\cdot v,\Delta u-u)\quad\quad\forall z=(u,v)\in
H\;\;(\text{i.e.}\;u\in W^{1,2}_0(\O,\R^k),\,v\in L^2(\O,\R^k))
\end{equation}
Then using \er{funcoprrrthhjkkkhhhbjhjhjhyh} and
\er{funcoprrrthhjjjkkkk}
we deduce
\begin{multline}\label{rgtrutyiuibcgsgtuikik}
\Big<h,\Lambda\cdot (T\cdot h)\Big>_{X\times
X^*}=\Big<u,S_0\cdot(T_1\cdot v)\Big>_{X_0\times
X_0^*}+\Big<v,\Delta (\mathcal{T}_0\cdot u)-\mathcal{T}_0\cdot
u\Big>_{X_1\times X_1^*}\\=\int_\O v(x)\cdot\Big(u(x)-\Delta
u(x)\Big)dx-\int_\O \Big(\nabla v(x):\nabla u(x)+v(x)\cdot
u(x)\Big)\,dx=0\quad\quad\forall h=(u,v)\in X\,.
\end{multline}
%
%
%
%
%
%
%
%
%
%
Furthermore, for $t\in[0,T_0]$ let $F_t(z):H\to H$ be a function
defined by
\begin{multline}\label{jhgdueowuwyhfyklkl}
F_t(z):=
\Big(\Upsilon\big(u(x),x,t\big),u(x)-\Theta\big(u(x),x,t\big)-div_x\Xi\big(u(x),x,t\big)\Big)\quad\quad
\forall z=(u,v)\in H\,,
\end{multline}
(We have $\Upsilon(u(x),x,t)\in W^{1,2}_0(\O,\R^k)$ for a.e. $t$),
i.e.
\begin{multline}\label{jhgdueowuwyhfy}
\Big<F_t(z),z_0\Big>_{H\times H}=
\int_\O\Big(\nabla_x\big\{\Upsilon\big(u(x),x,t\big)\big\}:\nabla
u_0(x)+\Upsilon\big(u(x),x,t\big)\cdot
u_0(x)\Big)dx+\\\int_\O\bigg\{u(x)-\Theta\big(u(x),x,t\big)-div_x\Xi\big(u(x),x,t\big)\bigg\}
\cdot v_0(x)dx\\
\forall z=(u,v)\in H,\;\forall z_0=(u_0,v_0)\in H\,.
\end{multline}
Then it satisfies the following
conditions
\begin{equation}\label{roststlambdglkFFyhuhtyui99999999aplvjhf}
\|F_t(z)\|_{H}\leq C\|z\|_H+f(t)\quad\quad\forall z\in H,\;\forall
t\in[0,T_0]\,,
\end{equation}
and
\begin{equation}\label{roststlambdglkFFyhuhtyui99999999aplvjhfjkjkkljjk}
\|\widetilde{T}\circ DF_t(z)\|_{\mathcal{L}(H;X^*)}\leq
C\quad\quad\forall z\in H,\;\forall t\in[0,T_0]\,,
\end{equation}
for some $C>0$ and some $f(t)\in L^2(0,T_0;\R)$. Moreover, for
bounded $\O$, since the embedding of $W^{1,2}_0(\O,\R^k)$ into
$L^2(\O,\R^k)$ is compact we obtain that $F_t$ is weak to weak
continuous on $H$. If we assume $\O$ to be unbounded then, for every
$\O'\subset\subset\O$, $F_t$ is weak to weak continuous, as a
mapping defined on $H$ with the valued functions, restricted to the
smaller set $\O'$. Therefore, since $\O'$ is arbitrary, using
\er{roststlambdglkFFyhuhtyui99999999aplvjhf} we deduce that in any
case $F_t$ is weak to weak continuous on $H$.
Then all the conditions of Corollary \ref{thhypppggg} satisfied and
by this Corollary for every $w_0\in W^{1,2}_0(\O,\R^k)$ and every
$h_0\in L^2(\O,\R^k)$ there exists $\zeta(t)\in L^\infty(0,T_0;H)$,
such that $\xi(t):=\widetilde T\cdot \big(\zeta(t)\big)\in
W^{1,2}(0,T_0;X^*)$ and $\zeta(t)$ satisfies the following equation
\begin{equation}\label{uravnllllgsnachlklimhjhghaplhkgujgukyuiioy8utr}
\begin{cases}\frac{d \xi}{dt}(t)+\Lambda\cdot \big(\zeta(t)\big)+
\widetilde{T}\cdot F_t\big(\zeta(t)\big)=0\quad\text{for a.e.}\; t\in(0,T_0)\,,\\
\zeta(0)=\Big(w_0(x),\;-h_0(x)-\Upsilon\big(w_0(x),x,0\big)\Big)\,,
\end{cases}
\end{equation}
where we assume that $\zeta(t)$ is $H$-weakly continuous on
$[0,T_0]$, as it was stated in Corollary \ref{vbnhjjmcor}. We can
rewrite \er{uravnllllgsnachlklimhjhghaplhkgujgukyuiioy8utr} as
follows. Let $\big(u(x,t),v(x,t)\big)=\zeta(t)$. Then by
\er{uravnllllgsnachlklimhjhghaplhkgujgukyuiioy8utr},
\er{nrmojiihgufoper}, \er{dvbgfghtfgjgyhjkghliuj},
\er{jhgdueowuwyhfy}, \er{tildetkkkklllllkkklklkklklklkkjjkkkkkkkk}
and Lemma \ref{vlozhenie}, $u(x,t)\in
L^\infty\big(0,T_0;W^{1,2}_0(\O,\R^k)\big)\cap
W^{1,2}\big(0,T_0;L^2(\O,\R^k)\big)$, $v(x,t)\in
L^\infty\big(0,T_0;L^2(\O,\R^k)\big)\cap W^{1,2}\big(0,T_0;
W^{-1,2}(\O,\R^k)\big)$, $u(x,t)$ is $W^{1,2}_0(\O,\R^k)$-weakly
continuous on $[0,T_0]$, $v(x,t)$ is $L^2(\O,\R^k)$-weakly
continuous on $[0,T_0]$, $u(x,0)=w_0(x)$,
$v(x,0)=-h_0(x)-\Upsilon\big(w_0(x),x,0\big)$ and in
$\O\times(0,T_0)$ $\big(u(x,t), v(x,t)\big)$ solves
\begin{equation}\label{uravnllllgsnachlklimjjjaplneabstnedvgggkkk}
\begin{cases}\frac{du}{dt}(x,t)+v(x,t)+\Upsilon\big(u(x,t),x,t)=0\,,\\
\frac{dv}{dt}(x,t)+\Delta_x
u(x,t)-\Theta\big(u(x,t),x,t)-div_x\Xi\big(u(x,t),x,t)=0\,.
\end{cases}
\end{equation}
Thus in particular $\frac{du}{dt}(x,t)\in
L^\infty\big(0,T_0;L^2(\O,\R^k)\big)\cap W^{1,2}\big(0,T_0;
W^{-1,2}(\O,\R^k)\big)$ and $\frac{du}{dt}(x,0)=h_0(x)$. Moreover,
differentiating the equality
$v(x,t)=-\frac{du}{dt}(x,t)-\Upsilon\big(u(x,t),x,t)$ by the
argument $t$ and inserting it into the second equation in
\er{uravnllllgsnachlklimjjjaplneabstnedvgggkkk} we finally deduce
\er{uravnllllgsnachlklimjjjaplneabstnedvkkk}.
\end{proof}

\subsection{Schr\"{o}dinger type nonlinear systems}
\begin{proposition}\label{divformnedvkkkljl}
Let $\O\subset\R^N$ be an open set and $T_0>0$. Furthermore, let
$\Theta(a,b,x,t):\R_a^k\times\R_b^k\times\R_x^N\times\R_t\to\R^{k}$
and
$\;\Xi(a,b,x,t):\R_a^k\times\R_b^k\times\R_x^N\times\R_t\to\R^{k}$
be measurable functions. Moreover, assume that $\Theta(a,b,x,t)$ and
$\Xi(a,b,x,t)$ are $C^1$ as a functions of the first two arguments
$a$ and $b$ when $(x,t)$ are fixed. We also assume that
$\Theta(a,b,x,t)$, $\nabla_x\Theta(a,b,x,t)$, $\Xi(a,b,x,t)$ and
$\nabla_x\Xi(a,b,x,t)$ are globally Lipschitz by the first two
arguments $a$ and $b$, and $\Theta(0,0,x,t)\in
L^2\big(\R;W^{1,2}_0(\O,\R^k)\big)$ and
$\Xi(0,0,x,t)\in L^2\big(\R;W^{1,2}_0(\O,\R^k)\big)$.
Then for every $w_0(x)\in W^{1,2}_0(\O,\R^k)$ and $h_0(x)\in
W^{1,2}_0(\O,\R^k)$ there exists
$u(x,t)\in L^\infty\big(0,T_0;W^{1,2}_0(\O,\R^k)\big)\cap
W^{1,2}\big(0,T_0; W^{-1,2}(\O,\R^k)\big)$ and $v(x,t)\in
L^\infty\big(0,T_0;W^{1,2}_0(\O,\R^k)\big)\cap W^{1,2}\big(0,T_0;
W^{-1,2}(\O,\R^k)\big)$,
$u(x,t)$ and $v(x,t)$ are $W^{1,2}_0(\O,\R^k)$-weakly continuous on
$[0,T_0]$, $u(x,0)=w_0(x)$, $v(x,0)=h_0(x)$ and
$\big(u(x,t),v(x,t)\big)$ is a solution to
\begin{multline}\label{uravnllllgsnachlklimjjjaplneabstnedvkkkhh}
\begin{cases}\frac{du}{dt}(x,t)-\Delta_x
v(x,t)+\Theta\big(u(x,t),v(x,t),x,t\big)=0\quad\quad
\text{in}\;\;\O\times(0,T_0)\,,\\
\frac{dv}{dt}(x,t)+\Delta_x
u(x,t)+\Xi\big(u(x,t),v(x,t),x,t\big)=0\quad\quad
\text{in}\;\;\O\times(0,T_0)\,.\end{cases}
\end{multline}
\end{proposition}
\begin{proof}
Let $X_0:=\big\{\varphi\in W_0^{1,2}(\O,\R^k)\cap
W^{3,2}_{loc}(\O,\R^k):\Delta\varphi\in W_0^{1,2}(\O,\R^k) \big\}$
endowed with the norm
\begin{multline}\label{nrmpoprgbjuyhkkkjhjhh}
\|\varphi\|_{X_0}:=\sqrt{\|\nabla\Delta
\varphi\|^2_{L^2(\O,\R^{k\times N})}+\|\Delta
\varphi\|^2_{L^2(\O,\R^{k})}+\|\nabla
\varphi\|^2_{L^2(\O,\R^{k\times N})}+\|\varphi\|^2_{L^2(\O,\R^{k})}}
\\ \forall \varphi\in X_0\subset W_0^{1,2}(\O,\R^k)\cap
W_{loc}^{3,2}(\O,\R^k)\,.
\end{multline}
So $X_0$ is a separable reflexive Banach space (in fact it is a
Hilbert space). Next let $H_0:=W^{1,2}_0(\O,\R^k)$ endowed with the
standard scalar product $<\phi_1,\phi_2>_{H\times
H}=\int_\O\big(\nabla\phi_1(x):\nabla\phi_2(x)+\phi_1(x)\cdot\phi_2(x)\big)dx$
(a Hilbert space) and $\mathcal{T}_0\in \mathcal{L}(X_0;H_0)$ be a
trivial embedding operator from $X_0\subset W_0^{1,2}(\O,\R^k)$ into
$H_0=W_0^{1,2}(\O,\R^k)$. Then $\mathcal{T}_0$ is an injective
inclusion with dense image.
As before, in out notations,
$\big\{W^{1,2}_0(\O,\R^k)\big\}^*=W^{1,2}_0(\O,\R^k)$.

 Next, clearly, for every $h\in W^{-1,2}(\O,\R^k)$, there exists
unique $H_h\in W^{1,2}_0(\O,\R^k)$ such that $\Delta H_h-H_h=h$.
Then define $S_0\in\mathcal{L}\big(W^{-1,2}(\O,\R^k),X_0^*\big)$ by
\begin{multline}
\label{funcoprrrthhjjj} \big<\delta,S_0\cdot h\big>_{X_0\times
X_0^*}=\int_{\O}\bigg\{\Big((\nabla\Delta)\delta(x)-\nabla\delta(x)\Big):\nabla
H_h(x)+\Big((\Delta\delta(x)-\delta(x)\Big)\cdot H_h(x)\bigg\}dx\\
\forall \delta\in X_0
,\;\forall h\in
W^{-1,2}(\O,\R^k)\,.
\end{multline}
Then, since for every $\phi\in W^{1,2}_0(\O,\R^k)$ there exists
unique $\delta_\phi\in X_0$ such that
$\Delta\delta_\phi-\delta_\phi=\phi$ we deduce that $S_0$ is
injective inclusion (i.e. $\ker S_0=0$).
As before, $\{X_0,H_0,X_0^*\}$ is an evolution triple with the
corresponding inclusion operators $\mathcal{T}_0\in
\mathcal{L}(X_0;H_0)$ and $\widetilde{\mathcal{T}}_0\in
\mathcal{L}(H_0;X_0^*)$, as it was defined in Definition \ref{7bdf}
by
\begin{equation}\label{tildetkkkklllll}
\big<\delta,\widetilde{\mathcal{T}}_0\cdot \varphi\big>_{X_0\times
X_0^*}:=\big<\mathcal{T}_0\cdot \delta,\varphi\big>_{H_0\times
H_0}\quad\quad\text{for every}\; \varphi\in
H_0\;\text{and}\;\delta\in X_0\,.
\end{equation}
However,
\begin{multline}\label{tildetkkkklllllkklklklkl}
\big<\mathcal{T}_0\cdot \delta,\varphi\big>_{H_0\times
H_0}=\int_\O\Big(\nabla
\delta(x):\nabla\varphi(x)+\delta(x)\cdot\varphi(x)\Big)dx=\\
\int_\O\Big(\delta(x)-\Delta \delta(x)\Big)\cdot\varphi(x)\,dx=
\int_\O\Big(\delta(x)-\Delta \delta(x)\Big)\cdot\Big(\Delta
H_{L\cdot\varphi}(x)-H_{L\cdot\varphi}(x)\Big)dx=\\
\int_{\O}\bigg\{\Big((\nabla\Delta)\delta(x)-\nabla\delta(x)\Big):\nabla
H_{L\cdot\varphi}(x)+\Big((\Delta\delta(x)-\delta(x)\Big)\cdot
H_{L\cdot\varphi}(x)\bigg\}dx\\
=\big<\delta,(S_0\circ L)\cdot \varphi\big>_{X_0\times
X_0^*}\quad\quad\text{for every}\; \varphi\in
H_0\;\text{and}\;\delta\in X_0\,,
\end{multline}
where
$L\in\mathcal{L}\big(W^{1,2}_0(\O,\R^k),W^{-1,2}(\O,\R^k)\big)$ is a
trivial inclusion of $W^{1,2}_0(\O,\R^k)$ in $W^{-1,2}(\O,\R^k)$.
Thus plugging \er{tildetkkkklllllkklklklkl} into
\er{tildetkkkklllll} we obtain
\begin{equation}\label{tildetkkkklllllkkklklkklklklkkjjk}
\widetilde{\mathcal{T}}_0\cdot \varphi=S_0\cdot (L\cdot
\varphi)\quad\quad\text{for every}\; \varphi\in H_0\,.
\end{equation}
Next set
\begin{equation}\label{obshspkkklbjhjh}
X:=\Big\{\big(u(x),v(x)\big):\;u(x):\O\to\R^k,\;v(x):\O\to\R^k,\;
u(x)\in X_0,\;v(x)\in X_0\Big\}\,.
\end{equation}
In this space we consider the norm
\begin{equation}\label{nrmojiihgufhhh}
\|z\|_X:=\sqrt{\|u\|^2_{X_0}+\|v\|^2_{X_0}}\quad\quad\forall
z=(u,v)\in X\,.
\end{equation}
Thus $X$ is a separable reflexive Banach space. Next set
\begin{equation}\label{obshspkkklHHHggkg}
H:=\Big\{\big(u(x),v(x)\big):\;u(x):\O\to\R^k,\;v(x):\O\to\R^k,\;
u(x)\in H_0,\;v(x)\in H_0\Big\}\,.
\end{equation}
In this space we consider the scalar product
\begin{multline}\label{nrmojiihgufHHHguittui}
<z_1,z_2>_{H\times H}:=<u_1,u_2>_{H_0\times
H_0}+<v_1,v_2>_{H_0\times H_0}=\\ \int_\O\Big\{\nabla u_1(x):\nabla
u_2(x)+u_1(x)\cdot u_2(x)+\nabla v_1(x): \nabla v_2(x)+v_1(x)\cdot
v_2(x)\Big\}dx\\ \forall z_1=(u_1,v_1),z_2=(u_2,v_2)\in H\,.
\end{multline}
Then $H$ is a Hilbert space. Furthermore, consider
$T\in\mathcal{L}(X,H)$ by
\begin{equation}\label{nrmojiihgufoperjhjhjh}
T\cdot z=\big(\mathcal{T}_0\cdot u,\mathcal{T}_0\cdot
v\big)\quad\quad\forall z=(u,v)\in X\,.
\end{equation}
Then $T$ is an injective inclusion with dense image. Furthermore,
\begin{equation}\label{obshspkkklsoprgkgkg}
X^*:=\Big\{\big(u,v\big):\;
u\in X^*_0,\;v\in X^*_0\Big\}\,,
\end{equation}
where
\begin{equation}\label{funcoprrrthhjkkkhhhbjhjhjhyhgug}
\big<\delta,h\big>_{X\times X^*}=\big<\delta_0,h_0\big>_{X_0\times
X_0^*}+\big<\delta_1,h_1\big>_{X_0\times X_0^*}\quad\quad\forall
\delta=(\delta_0,\delta_1)\in X,\;\forall h=(h_0,h_1)\in X^*\,,
\end{equation}
and
\begin{equation}\label{nrmojiihgufdghsoprfyguik}
\|z\|_{X^*}:=\Big(\|u\|^2_{X^*_0}+\|v\|^2_{X^*_0}\Big)^{1/2}\quad\quad\forall
z=(u,v)\in X^*\,.
\end{equation}
Moreover, the corresponding operator $\widetilde{T}\in
\mathcal{L}(H;X^*)$, defined as in \er{tildet}, is defined by
\begin{equation}\label{nrmojiihgufoperhutildjkgbgk}
\widetilde{T}\cdot z=\big(\widetilde{\mathcal{T}_0}\cdot
u,\widetilde{\mathcal{T}_0}\cdot v\big)=\big(S_0\cdot (L\cdot
u),S_0\cdot (L\cdot v)\big)\quad\quad\forall z=(u,v)\in H\,.
\end{equation}
Thus $\{X,H,X^*\}$ is an evolution triple with the corresponding
inclusion operators $T\in \mathcal{L}(X;H)$ and $\widetilde{T}\in
\mathcal{L}(H;X^*)$, as it was defined in Definition \ref{7bdf}.

 Next let $\Lambda\in \mathcal{L}(H,X^*)$ be defined by
\begin{multline}\label{dvbgfghtfgjgyhjkghliujgjh}
\Lambda\cdot z:=\Big(-S_0\cdot(\Delta v-v),S_0\cdot(\Delta
u-u)\Big)\quad\quad\forall z=(u,v)\in H\\ (\text{i.e.}\;(\Delta
u-u)\in W^{-1,2}(\O,\R^k),\,(\Delta v-v)\in W^{-1,2}(\O,\R^k))\,,
\end{multline}
where $S_0$ is defined in \er{funcoprrrthhjjj}. Then using
\er{funcoprrrthhjkkkhhhbjhjhjhyhgug}
we deduce
\begin{multline}\label{rgtrutyiuibcgsgtuikikhljkl}
\Big<h,\Lambda\cdot (T\cdot h)\Big>_{X\times
X^*}=-\Big<u,S_0\cdot(\Delta v-v)\Big>_{X_0\times
X_0^*}+\Big<v,S_0\cdot(\Delta u-u)\Big>_{X_0\times
X_0^*}\\=-\int_{\O}\bigg\{\Big((\nabla\Delta u)(x)-\nabla
u(x)\Big):\nabla v(x)+\Big(\Delta u(x)-u(x)\Big)\cdot
v(x)\bigg\}dx\\+\int_{\O}\bigg\{\Big((\nabla\Delta v)(x)-\nabla
v(x)\Big):\nabla u(x)+\Big(\Delta v(x)-v(x)\Big)\cdot
u(x)\bigg\}dx=0\quad\quad\forall h=(u,v)\in X\,.
\end{multline}
Furthermore, for $t\in[0,T_0]$ let $F_t(z):H\to H$ be a function
defined by
\begin{multline}\label{jhgdueowuwyhfyguhikgjhng}
F_t(z):=\Big(\Theta\big(u(x,t),v(x,t),x,t\big)-v(x),\Xi\big(u(x,t),v(x,t),x,t\big)+u(x)\Big)\quad\quad
\forall z=(u,v)\in H\,,
\end{multline}
(we have $\Theta(u(x,t),v(x,t),x,t),\Xi(u(x,t),v(x,t),x,t)\in
W^{1,2}_0(\O,\R^k)$ for a.e. $t$), i.e.
\begin{multline}\label{jhgdueowuwyhfyguhik}
\Big<F_t(z), z_0\Big>_{H\times H}=\\
\int_\O\Bigg\{\bigg(\nabla_x\Big\{\Theta\big(u(x,t),v(x,t),x,t\big)\Big\}-\nabla
v(x)\bigg):\nabla
u_0(x)+\Big(\Theta\big(u(x,t),v(x,t),x,t\big)-v(x)\Big)\cdot
u_0(x)\\+\bigg(\nabla_x\Big\{\Xi\big(u(x,t),v(x,t),x,t\big)\Big\}+\nabla
u(x)\bigg):\nabla
v_0(x)+\Big(\Xi\big(u(x,t),v(x,t),x,t\big)+u(x)\Big)\cdot
v_0(x)\Bigg\}
dx\\
\forall z=(u,v)\in H,\;\forall z_0=(u_0,v_0)\in H\,.
\end{multline}
Then
\begin{equation}\label{roststlambdglkFFyhuhtyui99999999aplvjhfhjhfjgjfjg}
\|F_t(z)\|_{H}\leq C\|z\|_H+f(t)\quad\quad\forall z\in H,\;\forall
t\in[0,T_0]\,,
\end{equation}
for some constant $C>0$ and some $f(t)\in L^2(0,T_0;\R)$.
Furthermore, it satisfies the Lipschitz condition
\begin{equation}\label{roststlambdglkFFyhuhtyui99999999aplvjhfhjh}
\big\|\widetilde T\circ D F_t(z)\big\|_{\mathcal{L}(H;X^*)}\leq
C\quad\quad\forall z\in H,\;\forall t\in[0,T_0]\,.
\end{equation}
Moreover, since the embedding of $H=W^{1,2}_0(\O,\R^k)$ in
$L^2_{loc}(\O,\R^k)$ is compact, we obtain that if
$z_n\rightharpoonup z_0$ weakly in $H$ then $z_n\to z_0$ strongly in
$L^2_{loc}(\O,\R^k)$. Thus, by
\er{roststlambdglkFFyhuhtyui99999999aplvjhfhjhfjgjfjg} we obtain
$F_t(z_n)\rightharpoonup F_t(z_0)$ weakly in $H$. So $F_t$
is weak to weak continuous in $H$.
Then all the conditions of Corollary \ref{thhypppggg} satisfied and
by this Corollary for every $w_0\in W^{1,2}_0(\O,\R^k)$ and every
$h_0\in W^{1,2}_0(\O,\R^k)$ there exists
$\zeta(t)=\big(u(x,t),v(x,t)\big)\in L^\infty(0,T_0;H)$, such that
$\xi(t):=\widetilde T\cdot \big(\zeta(t)\big)\in W^{1,2}(0,T_0;X^*)$
and $\zeta(t)$ satisfy the following equation
\begin{equation}\label{uravnllllgsnachlklimhjhghaplhkgujgukyuiioy8utrgg}
\begin{cases}\frac{d \xi}{dt}(t)+\Lambda\cdot \zeta(t)+
\widetilde{T}\cdot F_t\big(\zeta(t)\big)=0\quad\text{for a.e.}\; t\in(0,T_0)\,,\\
\zeta(0)=\big(w_0(x),h_0(x)\big)\,,
\end{cases}
\end{equation}
where we assume that $\zeta(t)$ is $H$-weakly continuous on
$[0,T_0]$, as it was stated in Corollary \ref{vbnhjjmcor}. We can
rewrite \er{uravnllllgsnachlklimhjhghaplhkgujgukyuiioy8utrgg} as
follows. Let $\big(u(x,t),v(x,t)\big)=\zeta(t)$. Then $u(x,t)\in
L^\infty\big(0,T_0;W^{1,2}_0(\O,\R^k)\big)$, $v(x,t)\in
L^\infty\big(0,T_0;W^{1,2}_0(\O,\R^k)\big)$, $u(x,t)$ and $v(x,t)$
are $W^{1,2}_0(\O,\R^k)$-weakly continuous on $[0,T_0]$,
$u(x,0)=w_0(x)$, $v(x,0)=h_0(x)$ and by
\er{tildetkkkklllllkkklklkklklklkkjjk} and the definitions of
$\Lambda$ and $F_t$ we obtain
\begin{multline}\label{uravnllllgsnachlklimjjjaplneabstnedvgggkkkghhhggh}
-\bigg<\frac{\partial\delta}{\partial t}(x,t),S_0\cdot
u(x,t)\bigg>_{X_0\times X^*_0} +\bigg<\delta(x,t),S_0\cdot
\Big(-\Delta_x
v(x,t)+\Theta\big(u(x,t),v(x,t),x,t\big)\Big)\bigg>_{X_0\times
X_0^*}=0\\ \quad\quad\forall\delta(x,t)\in
C^1_c\big(0,T_0;X_0\big)\,,
\end{multline}
\begin{multline}\label{uravnllllgsnachlklimjjjaplneabstnedvgggkkkghhhgghkuouil}
-\bigg<\frac{\partial\delta}{\partial t}(x,t),S_0\cdot
v(x,t)\bigg>_{X_0\times X^*_0}
+\bigg<\delta(x,t),S_0\cdot\Big(\Delta_x
u(x,t)+\Xi\big(u(x,t),v(x,t),x,t\big)\Big)\bigg>_{X_0\times
X_0^*}\\=0\quad\quad\forall\delta(x,t)\in
C^1_c\big(0,T_0;X_0\big)\,.
\end{multline}
%
%
%
%
%
%
%
%
%
%
Then, by Lemma \ref{vlozhenie} we obtain $\frac{du}{dt}(x,t)\in
L^2\big(0,T_0; W^{-1,2}(\O,\R^k)\big)$ and $\frac{dv}{dt}(x,t)\in
L^2\big(0,T_0; W^{-1,2}(\O,\R^k)\big)$ and thus $u(x,t)\in
L^\infty\big(0,T_0;W^{1,2}_0(\O,\R^k)\big)\cap W^{1,2}\big(0,T_0;
W^{-1,2}(\O,\R^k)\big)$ and $v(x,t)\in
L^\infty\big(0,T_0;W^{1,2}_0(\O,\R^k)\big)\cap W^{1,2}\big(0,T_0;
W^{-1,2}(\O,\R^k)\big)$. Moreover
$\big(u(x,t), v(x,t)\big)$ solves
\er{uravnllllgsnachlklimjjjaplneabstnedvkkkhh}.
\end{proof}

\subsection{Incompressible Navier-Stokes equations and Magneto-Hydrodynamics}
Let $\O\subset\R^N$ be a domain. The initial-boundary value problem
for the incompressible Navier-Stokes Equations is the following one,
\begin{equation}\label{IBNSnew}
\begin{cases}\;(i)\;\;\,\frac{\partial v}{\partial t}+\Div_x (v\otimes
v)+\nabla_x p=\nu_h\Delta_x v+f\quad\quad
\forall(x,t)\in\O\times(0,T_0)\,,\\\,(ii)\;\; \Div_x
v=0\quad\quad\forall(x,t)\in\O\times(0,T_0)\,,
\\(iii)\;\;v(x,t)=\gamma(x,t)\quad\quad\forall(x,t)\in\partial\O\times(0,T_0)\,,
\\(iv)\,\;\;v(x,0)=v_0(x)\quad\quad\forall x\in\O\,.
\end{cases}
\end{equation}
Here $v=v(x,t):\O\times(0,T_0)\to\R^N$ is an unknown velocity,
$p=p(x,t):\O\times(0,T_0)\to\R$ is an unknown pressure, associated
with $v$, $\nu_h>0$ is a given constant hydrodynamical viscosity,
$f:\O\times(0,T_0)\to\R^N$ is a given force field
$\gamma=\gamma(x,t)$ is a given velocity on the boundary (which can
be nontrivial for fluid driven by its boundary) and $v_0:\O\to\R^N$
is a given initial velocity.

 The initial-boundary value problem for the incompressible
Magneto-Hydrodynamics is the following one,
\begin{equation}\label{IBNSmhdnew}
\begin{cases}\;(i)\;\;\,\frac{\partial v}{\partial t}+\Div_x (v\otimes
v)-\Div_x (b\otimes b)+\nabla_x p=\nu_h\Delta_x v+f\quad\quad
\forall(x,t)\in\O\times(0,T_0)\,,\\ \,(ii)\;\;\,\frac{\partial
b}{\partial t}+\Div_x (b\otimes v)-\Div_x (v\otimes b)=\nu_m\Delta_x
b\quad\quad \forall(x,t)\in\O\times(0,T_0)\,,\\
(iii)\;\; \Div_x v=0\quad\quad\forall(x,t)\in\O\times(0,T_0)\,,
\\
(iv)\;\; \Div_x b=0\quad\quad\forall(x,t)\in\O\times(0,T_0)\,,
\\
(v)\;\;v(x,t)=0
\quad\quad\forall(x,t)\in\partial\O\times(0,T_0)\,,
\\
(vi)\;\;b\cdot\vec
n=0\quad\quad\forall(x,t)\in\partial\O\times(0,T_0)\,,
\\(vii)
\;\;\sum_{j=1}^{N}\Big(\frac{\partial b_i}{\partial
x_j}-\frac{\partial b_j}{\partial x_i}\Big)\vec
n_j=0\quad\quad\forall(x,t)\in\partial\O\times(0,T_0),\;\forall
i=1,2,\ldots N\,,
\\(viii)\,\;\;v(x,0)=v_0(x)\quad\quad\forall x\in\O\,,\\
(ix)\,\;\;b(x,0)=b_0(x)\quad\quad\forall x\in\O\,.
\end{cases}
\end{equation}
Here $v=v(x,t):\O\times(0,T_0)\to\R^N$ is an unknown velocity,
$b=b(x,t):\O\times(0,T_0)\to\R^N$ is an unknown magnetic field,
$p=p(x,t):\O\times(0,T_0)\to\R$ is an unknown total pressure
(hydrodynamical+magnetic), $\nu_h>0$ and $\nu_m>0$ are given
constant hydrodynamical and magnetic viscosities,
$f:\O\times(0,T_0)\to\R^N$ is a given force field,
$v_0:\O\to\R^N$ is a given initial velocity, $b_0:\O\to\R^N$ is a
given initial magnetic field and $\vec n$ is a normal to
$\partial\O$.

Next if for some constant $\lambda\in\{0,1\}$ we consider the
system:
\begin{equation}\label{IBNSmhdnewgen}
\begin{cases}\frac{\partial v}{\partial t}+\Div_x (v\otimes
v)-\lambda \Div_x (b\otimes b)+\nabla_x p=\nu_h\Delta_x
v+f\quad\quad \forall(x,t)\in\O\times(0,T_0)\,,\\
\frac{\partial b}{\partial t}+\lambda \Div_x (b\otimes v)-\lambda
\Div_x (v\otimes b)=\nu_m\Delta_x
b\quad\quad \forall(x,t)\in\O\times(0,T_0)\,,\\
\Div_x v=0\quad\quad\forall(x,t)\in\O\times(0,T_0)\,,
\\
\Div_x b=0\quad\quad\forall(x,t)\in\O\times(0,T_0)\,,
\\
v(x,t)=\gamma(x,t)\quad\quad\forall(x,t)\in\partial\O\times(0,T_0)\,,
\\
b\cdot\vec n=0\quad\quad\forall(x,t)\in\partial\O\times(0,T_0)\,,
\\ \sum_{j=1}^{N}\Big(\frac{\partial b_i}{\partial
x_j}-\frac{\partial b_j}{\partial x_i}\Big)\vec
n_j=(\lambda/\nu_m)\big(\gamma\cdot\vec
n\big)b\quad\quad\forall(x,t)\in\partial\O\times(0,T_0),\;\forall
i=1,2,\ldots N\,,
\\ v(x,0)=v_0(x)\quad\quad\forall x\in\O\,,\\
b(x,0)=b_0(x)\quad\quad\forall x\in\O\,,
\end{cases}
\end{equation}
then for $\lambda=1$ and $\gamma\equiv 0$ this system will coincide
with \er{IBNSmhdnew}. On the other hand if $(v,b,p)$ is a solution
to \er{IBNSmhdnewgen} for $\lambda=0$, then $(v,p)$ will be a
solution to \er{IBNSnew}.

 If there exists a sufficiently regular
function $r=r(x,t):\O\times(0,T_0)\to\R^N$ such that
$r(x,t)=\gamma(x,t)$ $\forall(x,t)\in\partial\O\times(0,T_0)$ and
$\Div_x r\equiv 0$, then fix it and define the new unknown function
$u(x,t):=v(x,t)-r(x,t)$ and its initial value
$u_0(x):=v_0(x)-r(x,0)$. Then we can rewrite \er{IBNSmhdnewgen} in
the terms of $(u,b,p)$ as
\begin{equation}\label{IBNSmhdnewgenhom}
\begin{cases}\frac{\partial u}{\partial t}+\Div_x \big(u\otimes
u+r\otimes u+u\otimes r-\lambda b\otimes b\big)+\nabla_x p=\nu_h\Delta_x u+\hat f\quad \forall(x,t)\in\O\times(0,T_0)\,,\\
\frac{\partial b}{\partial t}+\lambda \Div_x (b\otimes u-u\otimes
b+b\otimes r-r\otimes b)=\nu_m\Delta_x
b\quad\quad \forall(x,t)\in\O\times(0,T_0)\,,\\
\Div_x u=0\quad\quad\forall(x,t)\in\O\times(0,T_0)\,,
\\
\Div_x b=0\quad\quad\forall(x,t)\in\O\times(0,T_0)\,,
\\
u=0\quad\quad\forall(x,t)\in\partial\O\times(0,T_0)\,,
\\
b\cdot\vec n=0\quad\quad\forall(x,t)\in\partial\O\times(0,T_0)\,,
\\ \sum_{j=1}^{N}\Big(\frac{\partial b_i}{\partial
x_j}-\frac{\partial b_j}{\partial x_i}\Big)\vec
n_j=(\lambda/\nu_m)\big(r\cdot\vec
n\big)b\quad\quad\forall(x,t)\in\partial\O\times(0,T_0),\;\forall
i=1,2,\ldots N\,,
\\ u(x,0)=u_0(x)\quad\quad\forall x\in\O\,,\\
b(x,0)=b_0(x)\quad\quad\forall x\in\O\,,
\end{cases}
\end{equation}
where $\hat f:=f+\Delta_x r-\partial_t r-div_x\,(r\otimes r)$. We
will provide
with the existence of solution  for the system \er{IBNSmhdnewgenhom}
for every constant $\lambda\in\{0,1\}$.

 We need some preliminaries.
\begin{definition}\label{dXYmhdnew}
Let $\O\subset\R^N$ be an open set. We denote:
\begin{itemize}
\item
By $\mathcal{V}_N=\mathcal{V}_N(\O)$ the space $\{\f\in
C_c^\infty(\O,\R^N):\,div\,\f=0\}$ and by $L_N=L_N(\O)$ the space,
which is the closure of $\mathcal{V}_N$ in the space $L^2(\O,\R^N)$,
endowed with the scalar product $\big<\f_1,\f_2\big>_{B_N}:=\int_\O
\f_1\cdot \f_2\, dx$ and the norm
$\|\f\|:=\big(\int_\O|\f|^2dx\big)^{1/2}$.


\item
By $V_N=V_N(\O)$ the closure of $\mathcal{V}_N$ in
$W^{1,2}_0(\O,\R^N)$ endowed with the scalar product
$\big<\f_1,\f_2\big>_{V_N}:=\int_\O \big(\nabla\f_1:\nabla
\f_2+\f_1\cdot\f_2\big) dx$ and the norm
$\|\f\|:=\big(\int_\O|\nabla \f|^2dx+\int_\O|\f|^2dx\big)^{1/2}$.


\item $C_c^\infty(\ov\O,\R^N):=\big\{\f:\O\to\R^N:\;\exists\bar \f\in
C_c^\infty(\R^N,\R^N)\;\text{s.t.}\;\bar\f(x)=\f(x)\;\forall
x\in\O\big\}$.

\end{itemize}
Furthermore, given $\f\in \mathcal{D}'(\O,\R^N)$ denote
\begin{equation}\label{bhygytbhdrhthjmhdmhdnew}
rot_x \f:=\bigg\{\frac{\partial \f_i}{\partial x_j}-\frac{\partial
\f_j}{\partial x_i}\bigg\}_{1\leq i,j\leq N}=\big(\nabla_x
f\big)-\big(\nabla_x f\big)^T\in \mathcal{D}'\big(\O,\R^{N\times
N}\big)\,,
\end{equation}
and define the linear space
\begin{equation}\label{bhygytbhdmhdnew}
B'_N=B'_N(\O):=\Bigg\{\f\in L_N:\;rot_x \f\in L^2\big(\O,\R^{N\times
N}\big)\Bigg\}\,,
\end{equation}
endowed with the scalar product $\big<\f_1,\f_2\big>_{B'_N}:=\int_\O
\big(\f_1\cdot \f_2+ (1/2)rot_x \f_1\cdot rot_x \f_2 \big)\,dx$ and
the corresponding norm
$\|\f\|_{B'_N}:=\big(<\f,\f>_{B'_N}\big)^{1/2}$. Then $B'_N$ is a
Hilbert space. Moreover, clearly $B'_N$ is continuously embedded in
$W^{1,2}_{loc}(\O,\R^N)\cap L_N$. We also denote a smaller space
$B_N=B_N(\O)$ as a closure of $B'_N(\O)\cap C_c^\infty(\ov\O,\R^N)$
in $B'_N(\O)$
endowed with the norm of $B'_N(\O)$ (clearly if the boundary of
domain $\O$ is sufficiently regular then $B_N$ and $B'_N$ coincide).
\end{definition}
\begin{proposition}\label{premainkkkmhdnew}
For every $r\in L^2\big(0,T_0;W^{1,2}(\O,\R^N)\big)\cap L^\infty$,
$f\in L^2\big(0,T_0;L^2(\O,\R^{N})\big)$, $g\in
L^2\big(0,T_0;L^2(\O,\R^{N\times N})\big)$, $\nu_h>0$, $\nu_m>0$,
$\lambda\in\{0,1\}$, $v_0(\cdot)\in L_N$ and $b_0(\cdot)\in L_N$
there exist $u(x,t)\in L^2(0,T_0;V_N)\cap L^\infty(0,T_0;L_N)$ and
$b(x,t)\in L^2(0,T_0;B_N)\cap L^\infty(0,T_0;L_N)$, such that
$u(\cdot,t)$ and $b(\cdot,t)$ are $L_N$-weakly continuous in $t$ on
$[0,T_0]$, $u(x,0)=v_0(x)$, $b(x,0)=b_0(x)$ and $u(x,t)$ and
$b(x,t)$ satisfy
\begin{multline}\label{nnhogfcuykuylk48mhdnew}
\int_0^{T_0}\int_\O \Bigg\{ \Big(u(x,t)\otimes u(x,t)+r(x,t)\otimes
u(x,t)+u(x,t)\otimes r(x,t)-\lambda b(x,t)\otimes
b(x,t)+g(x,t)\Big):\nabla_x
\psi(x,t)\\-f(x,t)\cdot\psi(x,t)+u(x,t)\cdot\partial_t\psi(x,t)\Bigg\}\,dxdt=\int_0^{T_0}\int_\O\nu_h
\nabla_x u(x,t):\nabla_x\psi(x,t)\,dxdt-\int_\O
v_0(x)\cdot\psi(x,0)\,dx\,,
\end{multline}
for every $\psi(x,t)\in C^1_c\big(\O\times[0,T_0),\R^N\big)\cap
C^1\big([0,T_0];V_N\big)$
and
\begin{multline}\label{nnhogfcuykuylk48mhdjjjmhdnew}
\int_0^{T_0}\int_\O \Bigg\{ \lambda\Big(b(x,t)\otimes
u(x,t)-u(x,t)\otimes b(x,t)+b(x,t)\otimes r(x,t)-r(x,t)\otimes
b(x,t)\Big):\nabla_x
\phi(x,t)\\+b(x,t)\cdot\partial_t\phi(x,t)\Bigg\}\,dxdt=\int_0^{T_0}\int_\O\frac{\nu_m}{2}rot_x
b(x,t):rot_x\phi(x,t)\,dxdt-\int_\O b_0(x)\cdot\phi(x,0)\,dx\,,
\end{multline}
for every $\phi(x,t)\in C^1_c\big(\R^N\times[0,T_0),\R^N\big)\cap
C^1\big([0,T_0];B_N\big)$.
I.e.
\begin{equation}\label{IBNSmhdnewgenhomhjkhgjjhfgjh}
\begin{cases}\frac{\partial u}{\partial t}+\Div_x \big(u\otimes
u+r\otimes u+u\otimes r-\lambda b\otimes b\big)+\nabla_x p=\nu_h\Delta_x u-f-\Div_x g\quad \forall(x,t)\in\O\times(0,T_0)\,,\\
\frac{\partial b}{\partial t}+\lambda \Div_x (b\otimes u-u\otimes
b+b\otimes r-r\otimes b)=\nu_m\Delta_x
b\quad\quad \forall(x,t)\in\O\times(0,T_0)\,,\\
\Div_x u=0\quad\quad\forall(x,t)\in\O\times(0,T_0)\,,
\\
\Div_x b=0\quad\quad\forall(x,t)\in\O\times(0,T_0)\,,
\\
u=0\quad\quad\forall(x,t)\in\partial\O\times(0,T_0)\,,
\\
b\cdot\vec n=0\quad\quad\forall(x,t)\in\partial\O\times(0,T_0)\,,
\\
rot_x b\cdot\vec n=(\lambda/\nu_m)\big(r\cdot\vec
n\big)b\quad\quad\forall(x,t)\in\partial\O\times(0,T_0),
\\ u(x,0)=u_0(x)\quad\quad\forall x\in\O\,,\\
b(x,0)=b_0(x)\quad\quad\forall x\in\O\,,
\end{cases}
\end{equation}
Moreover, if either $\lambda=0$ and $\O$ is bounded or $r(x,t)\equiv
0$,
then $u(x,t)$ and $b(x,t)$ satisfy
the energy inequality
\begin{multline}\label{yteqmhdnew}
\frac{1}{2}\int_\O \big|u(x,\tau)\big|^2dx+\frac{1}{2}\int_\O
\big|b(x,\tau)\big|^2dx+\int_0^\tau\int_\O\nu_h\big|\nabla_x
u(x,t)\big|^2\,dxdt\\
+\int_0^\tau\int_\O\frac{\nu_m}{2}\big|rot_x
b(x,t)\big|^2\,dxdt\leq\frac{1}{2}\int_\O
\big|v_0(x)\big|^2dx+\frac{1}{2}\int_\O
\big|b_0(x)\big|^2dx\\+\int_0^\tau\int_\O
\bigg(\Big\{g(x,t)+r(x,t)\otimes u(x,t)+u(x,t)\otimes
r(x,t)\Big\}:\nabla_x u(x,t)\\+\lambda\big\{b(x,t)\otimes
r(x,t)\big\}:rot_x b(x,t)-f(x,t)\cdot
u(x,t)\bigg)\,dxdt\quad\quad\forall \tau\in[0,T_0]\,.
\end{multline}
\end{proposition}
\begin{proof}
Fix $\nu_h>0$, $\nu_m>0$, $\lambda\in\{0,1\}$, $f\in
L^2\big(0,T_0;L^2(\O,\R^{N})\big)$ $g\in
L^2\big(0,T_0;L^2(\O,\R^{N\times N})\big)$, $r\in
L^2\big(0,T_0;W^{1,2}(\O,\R^N)\big)\cap L^\infty$, $v_0(\cdot)\in
L_N$ and $b_0(\cdot)\in L_N$. Next define the space $U'_N$ as a
closure of $\mathcal{V}_N$ with respect to the norm
\begin{equation}\label{efreuethjjhh}
\|\f\|_{U'_N}:=\|\f\|_{V_N}+\sup\limits_{x\in\O}|\f(x)|+\sup\limits_{x\in\O}|\nabla\f(x)|\,.
\end{equation}
and the space $D'_N$ as a closure of $B_N\cap
C_c^\infty(\ov\O,\R^N)$ with respect to the norm
\begin{equation}\label{efreuetghuytjhhh}
\|\f\|_{D'_N}:=\|\f\|_{B_N}+\sup\limits_{x\in\O}|\f(x)|+\sup\limits_{x\in\O}|\nabla\f(x)|\,.
\end{equation}
Then clearly $U'_N$ and $D'_N$ are separable Banach spaces, which,
however, are not reflexive. On the other hand, by Lemma
\ref{hilbcomban} there exist separable Hilbert spaces $U_N$ and
$D_N$ and bounded linear inclusion operators $A_1\in
\mathcal{L}(U_N;U'_N)$ and $A_2\in \mathcal{L}(D_N;D'_N)$, such that
$A_1$ and $A_2$ are injective, the image of $A_1$ is dense in $U'_N$
and the image of $A_2$ is dense in $B'_N$
On the other hand, clearly $U'_N$ is trivially embedded in $V_N$,
and the trivial embedding operator $I_1\in \mathcal{L}(U'_N;V_N)$ is
injective and has dense range in $V_N$. Similarly, $D'_N$ is
trivially embedded in $B_N$, and the trivial embedding operator
$I_2\in \mathcal{L}(D'_N;B_N)$ is injective and has dense range in
$B_N$. Therefore if we define
\begin{equation}\label{vyffuyfyfffbm}
Q_1:=I_1\circ A_1\in\mathcal{L}(U_N;V_N)\quad\text{and}\quad
Q_2:=I_2\circ A_2\in\mathcal{L}(D_N;B_N)\,,
\end{equation}
then $Q_1$ and $Q_2$ are injective
and having dense ranges in $V_N$ and $B_N$ respectively. Next define
$P_1\in \mathcal{L}(V_N;L_N)$ as a trivial inclusion of $V_N$ into
$L_N$ and $P_2\in \mathcal{L}(B_N;L_N)$ as a trivial inclusion of
$B_N$ into $L_N$. Then clearly $P_1$ and $P_2$ are injective and
having dense ranges in $L_N$. Finally define
\begin{equation}\label{vyffuyfyfffbmuiuyi}
\mathcal{T}_1:=P_1\circ
Q_1\in\mathcal{L}(U_N;L_N)\quad\text{and}\quad
\mathcal{T}_2:=P_2\circ Q_2\in\mathcal{L}(D_N;L_N)\,.
\end{equation}
Then $\mathcal{T}_1$ and $\mathcal{T}_2$ are injective
and
having dense ranges in $L_N$. Next set
\begin{equation}\label{obshspkkklbjhjhmhdnewjhhyuiy}
X:=\Big\{\big(\psi,\varphi\big):\; \psi\in U_N,\;\varphi\in
D_N\Big\}\,,
\end{equation}
In this spaces we consider the norm
\begin{equation}\label{nrmojiihgufhhhmhdnewvvfv}
\|x\|_X:=\sqrt{\|\psi\|^2_{U_N}+\|\varphi\|^2_{D_N}}\quad\quad\forall
x=(\psi,\varphi)\in X\,.
\end{equation}
Thus $X$
is a separable reflexive Banach space. Similarly set
\begin{equation}\label{obshspkkklbjhjhmhdnew}
Z:=\Big\{\big(\psi(x),\varphi(x)\big):\;\psi(x):\O\to\R^N,\;\varphi(x):\O\to\R^N,\;
\psi(x)\in V_N,\;\varphi(x)\in B_N\Big\}\,,
\end{equation}
In this spaces we consider the norm
\begin{equation}\label{nrmojiihgufhhhmhdnewtc bvb}
\|z\|_Z:=\sqrt{\|\psi\|^2_{V_N}+\|\varphi\|^2_{B_N}}\quad\quad\forall
z=(\psi,\varphi)\in Z\,.
\end{equation}
Thus $Z$
is also a separable reflexive Banach space.
Finally set
\begin{equation}\label{obshspkkklHHHggkgmhdnew}
H:=\Big\{\big(\psi(x),\f(x)\big):\;\psi(x):\O\to\R^N,\;\f(x):\O\to\R^N,\;
\psi(x)\in L_N,\;\f(x)\in L_N\Big\}\,.
\end{equation}
In this space we consider the scalar product
\begin{multline}\label{nrmojiihgufHHHguittuimhdnew}
<h_1,h_2>_{H\times H}:=<\psi_1,\psi_2>_{L_N\times
L_N}+<\f_1,\f_2>_{L_N\times L_N}\\ =\int_\O\Big\{\psi_1(x)\cdot
\psi_2(x)+\f_1(x)\cdot \f_2(x)\Big\}dx\quad\quad \forall
h_1=(\psi_1,\f_1),h_2=(\psi_2,\f_2)\in H\,.
\end{multline}
Then $H$ is a Hilbert space.
Furthermore, consider $Q\in\mathcal{L}(X,Z)$ by
\begin{equation}\label{nrmojiihgufoperjhjhjhmhdnewgjhbh}
Q\cdot h=\big(Q_1\cdot \psi,Q_2\cdot \f\big)\quad\quad\forall
h=(\psi,\f)\in X\,.
\end{equation}
Similarly set $P\in\mathcal{L}(Z,H)$ by
\begin{equation}\label{nrmojiihgufoperjhjhjhmhdnewgjhbhjhjhggg}
P\cdot z=\big(P_1\cdot \psi,P_2\cdot \f\big)\quad\quad\forall
z=(\psi,\f)\in Z\,,
\end{equation}
and consider $T\in\mathcal{L}(X,H)$ by
\begin{equation}\label{nrmojiihgufoperjhjhjhmhdnew}
T\cdot h=\big(\mathcal{T}_1\cdot \psi,\mathcal{T}_2\cdot
\f\big)\quad\quad\forall h=(\psi,\f)\in X\,,
\end{equation}
Thus clearly $T=P\circ Q$ and $T$ is an injective inclusion with
dense image. Furthermore,
\begin{equation}\label{obshspkkklsoprgkgkgmhdnew}
X^*:=\Big\{\big(\psi,\f\big):\;\psi\in (U_N)^*,\;\f\in
(D_N)^*\Big\}\,,
\end{equation}
where
\begin{equation}\label{funcoprrrthhjkkkhhhbjhjhjhyhgugmhdnew}
\big<\delta,h\big>_{X\times X^*}=\big<\delta_0,h_0\big>_{U_N\times
(U_N)^*}+\big<\delta_1,h_1\big>_{D_N\times (D_N)^*}\quad\quad\forall
\delta=(\delta_0,\delta_1)\in X,\;\forall h=(h_0,h_1)\in X^*\,.
\end{equation}
Thus $\{X,H,X^*\}$ is an evolution triple with the corresponding
inclusion operators $T\in \mathcal{L}(X;H)$ and $\widetilde{T}\in
\mathcal{L}(H;X^*)$, as it was defined in Definition \ref{7bdf}.
Next, let $\Phi(h):Z\to[0,+\infty)$ be defined by
\begin{multline*}
\Phi(h):=\frac{1}{2}\int_\O \bigg(\nu_h\big|\nabla_x
\psi(x)\big|^2+\frac{\nu_m}{2}\big|rot_x
\f(x)\big|^2+\big|\psi(x)\big|^2+\big|\f(x)\big|^2\bigg)dx\\ \forall
h=(\psi,\f)\in Z=\big(V_N, B_N\big)\,.
\end{multline*}
So the mapping $D\Phi(h):Z\to Z^*$ is linear
and monotone.
Furthermore,
for every $t\in[0,T_0]$ let $\Theta_t(\sigma):H\to (U_N)^*$ be
defined by
\begin{multline}\label{jhfjggkjkjhhkhkloopllpppmhdnew}
\Big<\delta,\Theta_t(\sigma)\Big>_{U_N\times (U_N)^*}:=\\-\int_\O
\Bigg\{\bigg(w(x)\otimes w(x)+r(x,t)\otimes w(x,t)+w(x,t)\otimes
r(x,t)-\lambda b(x)\otimes b(x)\bigg)+g(x,t)\Bigg\}:\nabla
\{A_1\cdot\delta\}(x)\,dx\\+\int_\O\Big(f(x,t)-w(x)\Big)\cdot\{A_1\cdot\delta\}(x)\,dx\quad\quad
\forall \sigma=(w,b)\in L_N\oplus L_N\equiv H,\; \forall\delta\in
U_N\,,
\end{multline}
Next for every $t\in[0,T_0]$ let $\Xi_t(\sigma):H\to (D_N)^*$ be
defined by
\begin{multline}\label{jhfjggkjkjhhkhkloopllpppmhdhoghughnew}
\Big<\delta,\Xi_t(\sigma)\Big>_{D_N\times (D_N)^*}:=\\-\int_\O
\lambda\Big(b(x)\otimes w(x)-w(x)\otimes b(x)+b(x)\otimes
r(x,t)-r(x,t)\otimes
b(x)\Big):\nabla\{A_2\cdot\delta\}(x)\,dx\\-\int_\O
b(x)\cdot\{A_2\cdot\delta\}(x)\,dx\quad\quad \forall \sigma=(w,b)\in
L_N\oplus L_N\equiv H,\; \forall\delta\in D_N\,,
\end{multline}
Finally for every
$t\in[0,T_0]$ let $F_t(\sigma):H\to X^*$ be defined by
\begin{equation}
\label{jhfjggkjkjhhkhkloopllpppmhdjnhuinew}
F_t(\sigma):=\big(\Theta_t(\sigma),\Xi_t(\sigma)\big)\quad\forall\sigma\in
H
\end{equation}
Then
$F_t(\sigma)$ is G\^{a}teaux differentiable at every $\sigma\in H$,
and
the derivative of $F_t(\sigma)$ satisfy the condition
\begin{equation}
\label{roststlambdglkFFjjjapljjjkhklpppmhdnew}
\|DF_t(\sigma)\|_{\mathcal{L}(H;X^*)}\leq
C\big(\|\sigma\|_H+1\big)\quad\forall \sigma\in H,\;\forall
t\in[0,T_0]\,,
\end{equation}
for some constant $C>0$. Moreover,
\begin{multline}
\label{roststlambdglkFFjjjapljjjkhklpppmhdnewghhjgj}
\Big<\delta,F_t(T\cdot \delta)\Big>_{X\times
X^*}=\Big<\psi,\Theta_t(T\cdot \delta)\Big>_{U_N\times
(U_N)^*}+\Big<\varphi,\Xi_t(T\cdot \delta)\Big>_{D_N\times
(D_N)^*}=\\
-\int_\O \Bigg\{\bigg(w(x)\otimes w(x)+r(x,t)\otimes
w(x,t)+w(x,t)\otimes r(x,t)-\lambda b(x)\otimes
b(x)\bigg)+g(x,t)\Bigg\}:\nabla
w(x)\,dx\\+\int_\O\Big(f(x,t)-w(x)\Big)\cdot w(x)\,dx-\int_\O
\lambda\Big(b(x)\otimes w(x)-w(x)\otimes b(x)+b(x)\otimes
r(x,t)-r(x,t)\otimes b(x)\Big):\nabla b(x)\,dx\\-\int_\O b(x)\cdot
b(x)\,dx \quad\quad\text{where}\;
w=A_1\cdot\psi,\;b=A_2\cdot\varphi\quad\forall
\delta=(\psi,\varphi)\in U_N\oplus D_N=X,\;\forall t\in[0,T_0]\,,
\end{multline}
Thus since $w=A_1\cdot\psi\in U'_N$ and $b=A_2\cdot\varphi\in D'_N$
we rewrite \er{roststlambdglkFFjjjapljjjkhklpppmhdnewghhjgj} as
follows,
\begin{multline}
\label{roststlambdglkFFjjjapljjjkhklpppmhdnewghhjgjjhggghjjhffgh}
\Big<\delta,F_t(T\cdot \delta)\Big>_{X\times X^*}=\int_\O
\Big(f(x,t)\cdot w(x)-g(x,t):\nabla w(x)\Big)\,dx-\int_\O
\Big(\big|w(x)\big|^2+\big|b(x)\big|^2 \Big)\,dx\\-\int_\O
\bigg(\big\{r(x,t)\otimes w(x)+w(x)\otimes r(x,t)\big\}:\nabla
w(x)+\lambda\big\{b(x)\otimes
r(x,t)\big\}:rot_x b(x)\,\bigg)dx\\
-\int_\O \frac{1}{2}\bigg\{w(x)\cdot\nabla_x
\big|w(x)\big|^2+\lambda w(x)\cdot\nabla_x \big|b(x)\big|^2-2\lambda
b(x)\cdot\nabla_x\Big(w(x)\cdot b(x)\Big)\bigg\}\,dx\\
\quad\quad\text{where}\;
w=A_1\cdot\psi,\;b=A_2\cdot\varphi\quad\forall
\delta=(\psi,\varphi)\in U_N\oplus D_N=X,\;\forall t\in[0,T_0]\,,
\end{multline}
On the other hand $w(x),b(x)\in L_N$
and thus $\Div_x\{\chi_\O w\}=\Div_x\{\chi_\O b\}$
in the sense of distributions (here $\chi_\O$ is characteristic
function of the set $\O$). Thus the last integral in
\er{roststlambdglkFFjjjapljjjkhklpppmhdnewghhjgjjhggghjjhffgh}
vanishes, and therefore, since $r(x,t)\in L^\infty$ we obtain
\begin{multline}
\label{roststlambdglkFFjjjapljjjkhklpppmhdnewghhjgjjhggghjjhffghhjkkgjghgfgghihhj}
\Big<\delta,F_t(T\cdot \delta)\Big>_{X\times X^*}=\int_\O
\Big(f(x,t)\cdot w(x)-g(x,t):\nabla w(x)\Big)\,dx-\int_\O
\Big(\big|w(x)\big|^2+\big|b(x)\big|^2 \Big)\,dx\\-\int_\O
\bigg(\big\{r(x,t)\otimes w(x)+w(x)\otimes r(x,t)\big\}:\nabla
w(x)+\lambda\big\{b(x)\otimes
r(x,t)\big\}:rot_x b(x)\,\bigg)dx\geq\\
-C\Big(\big\|Q\cdot\delta\big\|_Z+1\Big)\Big(\big\|T\cdot\delta\big\|_H+1\Big)-\mu(t)
\quad\text{where}\; w=A_1\cdot\psi,\;b=A_2\cdot\varphi \quad\forall
\delta\in =(\psi,\varphi)\in X
,\;\forall t\in[0,T_0].
\end{multline}
Here $\mu(t)\in L^1(0,T_0;\R)$ is some nonnegative function.

 Next  consider the sequence
of open sets $\{\O_j\}_{j=1}^{\infty}$ such that for every
$j\in\mathbb{N}$, $\O_j$ is compactly embedded in $\O_{j+1}$, and
$\cup_{j=1}^{\infty}\O_j=\O$. Then set $Z_j:=L^2\big(\O_j,\R^N\big)$
and
$\bar L_j\in \mathcal{L}(L_N,Z_j)$ by
$$\bar L_j\cdot \big(h(x)\big):=h(x)\llcorner\O_j\in L^2\big(\O_j,\R^N\big)=Z_j\quad\quad\forall h(x)\in L_N(\O)\,.$$
Thus, by the standard embedding theorems in the Sobolev Spaces, the
operators $\bar L_j\circ P_1\in \mathcal{L}(V_N,Z_j)$ and $\bar
L_j\circ P_2\in \mathcal{L}(B_N,Z_j)$ are compact for every $j$.
Moreover, if $\{\sigma_n\}_{n=1}^{\infty}\subset H$ is a sequence
such that $\sigma_n=(h_n,w_n)\rightharpoonup \sigma_0=(h_0,w_0)$
weakly in $H$ and $\bar L_j\cdot h_n\to \bar L_j\cdot h_0$ and $\bar
L_j\cdot w_n\to \bar L_j\cdot w_0$, strongly in $Z_j$ as $n\to
+\infty$ for every $j$,
then we have $h_n\to h_0$ and $w_n\to w_0$ strongly in
$L^2_{loc}(\O,\R^N)$ and thus, by
\er{jhfjggkjkjhhkhkloopllpppmhdjnhuinew} and
\er{roststlambdglkFFjjjapljjjkhklpppmhdnew} we must have
$F_t(\sigma_n)\rightharpoonup F_t(\sigma_0)$ weakly in $X^*$.

 Thus all the conditions of Theorem \ref{thhypppggghhhj}
are satisfied. Applying this Theorem we deduce that there exists a
function $h(t)\in L^2\big(0,T_0;Z\big)$ such that $\sigma(t):=P\cdot
h(t)$ belongs to $L^\infty(0,T_0;H)$, $\gamma(t):=\widetilde T\cdot
\sigma(t)$ belongs to $W^{1,2}(0,T_0;X^*)$ and $h(t)$ is a solution
to
\begin{equation}\label{uravnllllgsnachlklimjjjapltggtkojjlnmmhdnew}
\begin{cases}\frac{d \gamma}{dt}(t)+F_t\big(\sigma(t)\big)+Q^*\cdot D\Phi\big(h(t)\big)=0\quad\text{for a.e.}\; t\in(0,T_0)\,,\\
\sigma(0)=\big(v_0(x), b_0(x)\big)\,,
\end{cases}
\end{equation}
where we assume that $\sigma(t)$ is $H$-weakly continuous on
$[0,T_0]$ and $Q^*\in \mathcal{L}(Z^*,X^*)$ is the adjoint to $Q$
operator. Then by the definitions of $\Phi$ and $F_t$,
$h(x,t):=\big(u(x,t),b(x,t)\big)$ satisfies that $u(x,t)\in
L^2(0,T_0;V_N)\cap L^\infty(0,T_0;L_N)$ and $b(x,t)\in
L^2(0,T_0;B_N)\cap L^\infty(0,T_0;L_N)$,
$u(\cdot,t)$ and $b(\cdot,t)$ are $L_N$-weakly continuous in $t$ on
$[0,T_0]$, $u(x,0)=v_0(x)$, $b(x,0)=b_0(x)$ and $u(x,t)$ and
$b(x,t)$ satisfy
\begin{multline}\label{nnhogfcuykuylk48mhdnewjkklhl}
\int_0^{T_0}\int_\O \Bigg\{ \Big(u(x,t)\otimes u(x,t)+r(x,t)\otimes
u(x,t)+u(x,t)\otimes r(x,t)-\lambda b(x,t)\otimes
b(x,t)+g(x,t)\Big):\nabla_x
\big\{A_1\cdot\psi(t)\big\}(x)\\-f(x,t)\cdot\big\{A_1\cdot\psi(t)\big\}(x)+u(x,t)\cdot
\big\{A_1\cdot\partial_t\psi(t)\big\}(x)\Bigg\}\,dxdt\\=\int_0^{T_0}\int_\O\nu_h
\nabla_x u(x,t):\nabla_x\big\{A_1\cdot\psi(t)\big\}(x)\,dxdt-\int_\O
v_0(x)\cdot\big\{A_1\cdot\psi(0)\big\}(x)\,dx\,,
\end{multline}
for every $\psi(t)\in C^1\big([0,T_0];U_N\big)$ such that
$\psi(T_0)=0$ and
\begin{multline}\label{nnhogfcuykuylk48mhdjjjmhdnewhugyguiujh}
\int_0^{T_0}\int_\O \Bigg\{ \lambda\Big(b(x,t)\otimes
u(x,t)-u(x,t)\otimes b(x,t)+b(x,t)\otimes r(x,t)-r(x,t)\otimes
b(x,t)\Big):\nabla_x\big\{A_2\cdot
\phi(t)\big\}(x)\\+b(x,t)\cdot\big\{A_2\cdot\partial_t\phi(t)\big\}(x)\Bigg\}\,dxdt\\=\int_0^{T_0}\int_\O\frac{\nu_m}{2}rot_x
b(x,t):rot_x\big\{A_2\cdot\phi(t)\big\}(x)\,dxdt-\int_\O
b_0(x)\cdot\big\{A_2\cdot\phi(0)\big\}(x)\,dx\,,
\end{multline}
for every $\phi(t)\in C^1\big([0,T_0];D_N\big)$ such that
$\phi(T_0)=0$. Thus since the image of $A_1$ is dense in $U_N'$ and
the image of $A_2$ is dense in $D_N'$, we deduce that $u(x,t)$ and
$b(x,t)$ are solutions of \er{nnhogfcuykuylk48mhdnew} and
\er{nnhogfcuykuylk48mhdjjjmhdnew}.
%
%
%
%
%
%
%
%

 Next by
\er{roststlambdglkFFjjjapljjjkhklpppmhdnewghhjgjjhggghjjhffghhjkkgjghgfgghihhj}
and by the definition of $\Phi$ we have
\begin{multline}
\label{roststlambdglkFFjjjapljjjkhklpppmhdnewghhjgjjhggghjjhffghhjkkgjghgfgghihhjkjjhghggh}
\Big<\delta,Q^*\cdot D\Phi(Q\cdot\delta)+F_t(T\cdot
\delta)\Big>_{X\times X^*}=\\ \int_\O \bigg(\nu_h\big|\nabla_x
w(x)\big|^2+\frac{\nu_m}{2}\big|rot_x b(x)\big|^2+\int_\O
\Big(f(x,t)\cdot w(x)-g(x,t):\nabla w(x)\Big)\,dx\\-\int_\O
\bigg(\big\{r(x,t)\otimes w(x)+w(x)\otimes r(x,t)\big\}:\nabla
w(x)+\lambda\big\{b(x)\otimes r(x,t)\big\}:rot_x b(x)\,\bigg)dx\\
\quad\text{where}\; w=A_1\cdot\psi,\;b=A_2\cdot\varphi \quad\forall
\delta=(\psi,\varphi)\in X
,\;\forall t\in[0,T_0]\,.
\end{multline}
However, if $\O$ is bounded then the embedding operator $P_1$ is
compact. On the other hand, either $\lambda=0$ and $\O$ is bounded
or $r(x,t)\equiv 0$. Thus, by
\er{roststlambdglkFFjjjapljjjkhklpppmhdnewghhjgjjhggghjjhffghhjkkgjghgfgghihhjkjjhghggh}
together with Theorem \ref{thhypppggghhhj}, we finally deduce
\er{yteqmhdnew}.
%
%
%
%
%
%
%
%
%
%
\end{proof}

\end{document}